\documentclass[preprint,3p]{amsart}
\usepackage[T1]{fontenc}
\usepackage{geometry}
\usepackage{amsmath, mathtools}
\usepackage{pdfpages}
\usepackage{amstext}
\usepackage{amsthm}
\usepackage{amssymb}
\usepackage{stmaryrd}
\usepackage{graphicx}
\usepackage{bbm}
\usepackage{comment}
\usepackage{faktor,url}
\usepackage{xfrac}
\usepackage[hidelinks]{hyperref}
\usepackage{xcolor}

\makeatletter

\numberwithin{equation}{section}
\theoremstyle{plain}
\newtheorem{thm}{\protect\theoremname}[section]
  \theoremstyle{definition}
  \newtheorem{defn}[thm]{\protect\definitionname}
  \theoremstyle{remark}
  \newtheorem{rem}{\protect\remarkname}
  \theoremstyle{plain}
  \newtheorem{lem}[thm]{\protect\lemmaname}
  \theoremstyle{plain}
  \newtheorem{prop}[thm]{\protect\propositionname}
 \theoremstyle{plain}
  \newtheorem{conj}[thm]{\protect\conjecturename}

\makeatother

  \providecommand{\definitionname}{Definition}
  \providecommand{\lemmaname}{Lemma}
  \providecommand{\propositionname}{Proposition}
  \providecommand{\remarkname}{Remark}
\providecommand{\theoremname}{Theorem}
\providecommand{\conjecturename}{Conjecture}

\newcommand{\R}{\mathbb{R}}
\newcommand{\Z}{\mathbb{Z}}
\newcommand{\N}{\mathbb{N}}
\newcommand{\T}{\mathbb{T}}

\newcommand{\LL}{\mathcal{L}}

\newcommand{\OO}{\mathcal{O}}
\newcommand{\X}{\mathcal{X}}
\newcommand{\Y}{\mathcal{Y}}

\newcommand{\RR}{\mathcal{R}}
\newcommand{\ZZ}{\mathcal{Z}}
\newcommand{\C}{\mathbb{C}}
\newcommand{\F}{\mathcal{F}}

\newcommand{\G}{\mathcal{G}}

\newcommand{\NN}{\mathcal{N}}

\newcommand{\wt}{\widetilde }

\newcommand{\wh}{\widehat}

\newcommand{\Id}{\mathrm{Id}}

\newcommand{\dps}{\displaystyle}

\newcommand{\bigO}{\mathcal{O}}

\title{Parabolic saddles and Newhouse domains in Celestial Mechanics}

\author[M. Garrido]{Miguel Garrido}
\address[MG]{Departament de Matem\`atiques, Universitat Aut\`onoma de Barcelona, 080193, Barcelona, Spain}
\email{miguel.garrido@uab.cat}

\author[P. Mart\'in]{Pau Mart\'in}
\address[PM]{Departament de Matem\`atiques, Universitat Polit\`ecnica de Catalunya, Barcelona, 08028, Spain \& Centre de Recerca Matem\'atica, Bellaterra, 08194, Barcelona, Spain }
\email{p.martin@upc.edu}

\author[J. Paradela]{Jaime Paradela}
\address[JP]{{Department of Mathematics, University of Maryland}, College Park, 20776, MD, USA}
\email{paradela@umd.edu}

\begin{document}

\begin{abstract}
    In \cite{McGeheestablemfold} McGehee introduced a compactification of the phase space of the restricted 3-body problem by gluing a manifold of  periodic orbits ``at infinity''. Although from the dynamical point of view  these periodic orbits are parabolic (the linearization of the Poincar\'{e} map is the identity matrix), one of them, denoted here by $O$, possesses stable and unstable manifolds which, moreover,  separate the regions of bounded and unbounded motion. 
    
    This observation prompted the investigation of the homoclinic picture associated to $O$,  starting with the work of Alekseev \cite{AlekseevQR1,AlekseevQR2} and Moser \cite{Moserbook}.
    We continue this research and extend, to this degenerate setting, some classical results in the theory of homoclinic bifurcations. More concretely, we prove that there exist Newhouse domains $\mathcal N$ in parameter space (the ratio of masses of the bodies) and residual subsets $\mathcal R\subset \mathcal N$ for which the homoclinic class of $O$ has maximal Hausdorff dimension and is accumulated by generic elliptic periodic orbits.

 One of the main consequences of our work is the fact that, for a (locally) topologically large set of parameters of the restricted 3-body problem the union of its elliptic islands forms an unbounded subset of the phase space and, moreover, the closure of the set of generic elliptic periodic orbits contains hyperbolic sets with Hausdorff dimension arbitrarily close to maximal. Other instances of the restricted $n$-body problem such as the Sitnikov problem and the case $n=4$ are also considered.
  \end{abstract}

\maketitle

\tableofcontents

\section{Introduction}

In the 70's Newhouse \cite{Newhouse79} showed that there exist open domains in the space of  smooth surface diffeomorphisms, formed by diffeomorphisms which exhibit homoclinic tangencies. These regions are called \textit{Newhouse domains}.  Apart from the importance of their own existence (what, for example, disproved the $\mathcal C^2$-density of axiom A  diffeomorphisms on compact surfaces \cite{NewhousenondenseaxiomA}), a remarkable feature of the Newhouse domains is the complexity of the dynamics of the diffeomorphisms in this region. Indeed, Newhouse himself proved that there exists a residual set of dissipative diffeomorphisms in the Newhouse region which have infinitely many sinks \cite{Newhousesinks}. Parametric versions of Newhouse's results appeared in the work of Robinson \cite{Robinsonnewhouseparam}. In \cite{Duarteabundance,Duarteabundancebif,Duartenearid} Duarte extended the theory to cover the area preserving case (including also parametric formulations) with the role of sinks replaced by elliptic islands. Since then, quite a number of surprisingly exotic dynamical phenomena have been shown to be ``abundant'' inside the Newhouse domains. For example, existence of homoclinic classes of maximal Hausdorff dimension \cite{NewhousehtopHD}, superexponential growth of the number of periodic orbits \cite{Kaloshinsuperexp,KaloshinGorodetskioftensinks}, existence of homoclinic tangencies of arbitrary order of degeneracy and universal dynamical conjugacy  classes \cite{GSTtangencies}. Moreover, Newhouse domains have been shown to exist for a variety of physical systems such as the Van der Pol equation \cite{Levimemoirs}, billiards \cite{TuraevRomKedarbilliards}, Lorenz-like models \cite{MR1257718},  the standard map \cite{Duartestdmap}, and, very recently, in the space of smooth Beltrami fields on $\mathbb{R}^3$ \cite{BergerFlorioPeraltaSalas}, just to mention a few.

 From a historical point of view, the study of the homoclinic picture in dynamical systems started with the work of Poincar\'{e} on the non-integrability of the $3$-body problem. However, and despite the efforts of many, the complex dynamics associated to the homoclinic phenomena that take place in Celestial Mechanics models is still very far from being well understood even in low dimensional systems. One of the most popular of such models is the so-called restricted 3-body problem, which can be seen as a limit  of the 3-body problem in which one of the bodies has mass zero (see \cite{Wintner}). In this limit, the motion of the \textit{massless body} is influenced by, but does not influence, the motion of the heavy bodies, which from now on we call the \textit{primaries}.

The main result of the present work, Theorem~\ref{thm:main}, concerns the existence of Newhouse domains for the restricted 3-body problem\footnote{In order to give a precise statement, we first reduce the study of the dynamics of the restricted 3-body problem to a one-parameter family of  area preserving maps (see Section~\ref{sec:restricted3bpareapreserving}).}. The more interesting point in our construction is that these Newhouse domains are associated to a \textit{degenerate} fixed point ``at infinity'' (see Section~\ref{sec:infinity}). As already highlighted above, diffeomorphisms in the Newhouse domain often exhibit very rich dynamics. This is portrayed in our second main result, Theorem~\ref{thm:main2}, which follows as a corollary of our proof of Theorem~\ref{thm:main} and well known techniques developed by Gorodetski \cite{Gorodetskistdmap} (see also \cite{Duarteabundance,Duarteabundancebif}). Together, Theorems \ref{thm:main} and \ref{thm:main2} imply Theorem \ref{thm:main3}, in which we show that for a (locally) topologically large set of parameters of the restricted 3-body problem the union of its elliptic islands forms an unbounded subset of the phase space and, moreover, the closure of the set of generic elliptic periodic orbits contains hyperbolic sets with Hausdorff dimension arbitrarily close to maximal.\vspace{0.3cm}

 \begin{rem}
     Theorem~\ref{thm:main} has already been announced in \cite{GorodetskiKaloshinannouncement} and a proof has appeared in the unpublished manuscript \cite{GorodetskiK12} by Gorodetski and Kaloshin. We however believe that our approach, which is completely different, is more conceptual and provides a rather stronger toolbox to study homoclinic phenomena in Celestial Mechanics. For example, the results in \cite{GorodetskiK12}, based on a $\mathcal{C}^2$ analysis are, a priori, not enough to deduce  the second part of Theorem~\ref{thm:main2} and Theorem~\ref{thm:main3}. In Section~\ref{sec:OM} we discuss more in detail the differences between our work and \cite{GorodetskiK12}.
 \end{rem}
 \vspace{0.3cm}

    In Section~\ref{sec:4bp} we also show (see Theorems~\ref{thm:Sitnikov}~and~\ref{thm:main4bp}) that similar results hold  for the Sitnikov and restricted 4-body problem.

\subsection{The restricted 3-body problem as an area preserving map}\label{sec:restricted3bpareapreserving}

We normalize the masses of the primaries so that $m_1=1-\mu$ and $m_2=\mu$ for $\mu\in(0,1/2]$. For any value of $\mu$ we consider the corresponding circular periodic orbit 
\[
q_{1,\mu}(t)=-\mu (\cos t,\sin t),\qquad\qquad  q_{2,\mu}(t)=(1-\mu)(\cos t,\sin t)
\]
of the 2-body Problem. The dynamical system describing the motion of the massless body in the same plane as the primaries is a Hamiltonian system known as the restricted planar circular 3-body problem (RPC3BP)\footnote{To be precise one should avoid collisions in the definition of the configuration space. Since we will work always in the region of the phase space where the massless body is far from the primaries, we will abuse notation and refer to $q\in \mathbb{R}^2$ as the confguration space.}
 \[
H_\mu(q,p,t)=\frac{|p|^2}{2}-V_{\mu}(q,t),\qquad\qquad (q,p,t)\in\mathbb{R}^4\times\mathbb T,
\]
with
\[
 V_{\mu}(q,t)=\frac{1-\mu}{|q-q_{1,\mu}(t)|}+\frac{\mu}{|q-q_{2,\mu}(t)|}.
\]
To exploit the invariance under rotations of the Hamiltonian, we introduce polar coordinates $(r,\alpha)\in\mathbb{R}_+\times\mathbb{T}$ on the configuration space (that is, $q=(r\cos\alpha,r\sin\alpha)$) and let $(y,G)\in T^*(\mathbb{R}_+\times\mathbb{T})$ denote the conjugated radial and angular momenta. It is then straightforward to check that, in the new variables, the Hamiltonian only depends on the angle $\phi=\alpha-t$. Hence, the so-called \textit{Jacobi constant} $\mathcal{J}=H_\mu-G$ is a conserved quantity. For any $J$ sufficiently large it is possible to express $G=G(r,y,\phi;J)$ and reduce the study of the dynamics on the hypersurface $\Sigma_J=\{\mathcal{J}(r,y,\phi,G)=J\}$ to a 3-dimensional dynamical system on the variables $(r,\phi,y)$. Moreover, since
\[
\dot \phi=\frac{G}{r^2}-1
\]
 one easily sees that, for $J$ large enough, there exists an open set $U\subset \Sigma_J\cap\{\phi=0\}$ where the first return map 
\begin{equation}\label{eq:Poincaremap}
\begin{split}
f_\mu:U\subset \Sigma_J\cap\{\phi=0\}&\to\Sigma_J\cap\{\phi=0\}\\
(r,y)&\mapsto (\bar r,\bar y)
\end{split}
\end{equation}
is well defined. In this setting, it is therefore natural to ask whether there exist Newhouse domains for the one-parameter family of  area preserving maps $\{f_\mu\}_{\mu\in(0,1/2]}$ in \eqref{eq:Poincaremap}. To the best of our knowledge, the first complete proof of this result has only appeared recently in \cite{Giraltcoorbitalchaos} (see also Section~\ref{sec:OM} where we discuss the unpublished work \cite{GorodetskiK12}).  In \cite{Giraltcoorbitalchaos} the authors prove that for a sequence of values of $\mu$ converging to zero there exists a quadratic homoclinic tangency (which unfolds generically) to a certain hyperbolic fixed point of the map $f_\mu$. Then, the existence of Newhouse domains is implied by the classical work of Duarte.

In the present work we are interested in homoclinic phenomena associated to a rather distinctive fixed point of the map $f_\mu$ which is located ``at infinity''. This fixed point is at the center of an important conjecture  (see Conjecture~\ref{eq:Kolmogorovconj}) in Celestial Mechanics. Moreover,  the analysis of the homoclinic picture associated to this point is rather challenging from the dynamical point of view since it is degenerate (of parabolic nature), and therefore, the classical theory does not apply.

\subsection{A degenerate fixed point at infinity}\label{sec:infinity}
It is well known that the dynamics of the 2 Body Problem, can be classified according to the value of the energy\footnote{The 2 Body Problem is a 1 degree of freedom Hamiltonian system, thus, integrable.}: negative energy levels correspond to elliptic (bounded) motions while positive energy levels correspond to hyperbolic (unbounded) motions. Separating these, there exists a codimension one submanifold, the so-called parabolic motions. The proper tool to study unbounded motions is McGehee's change of variables \cite{McGeheestablemfold}, which compactifies the configuration space by adding one point at infinity. This change of variables, given by the transformation $ h_{MG}:(x,y)\mapsto (r,y)= (2/x^2,y)$ will be studied in detail in Section~\ref{subsec:McGehee}. The main upshot of this compactification is that in McGehee's coordinates, the submanifold of parabolic motions can be seen as a homoclinic loop to the point $O=(0,0)$, which corresponds in the original coordinates to $r=\infty$, $y=0$ (see Figure~\ref{fig:McGehee}). This fixed point is however degenerate, the linearization of the dynamics (in McGehee's variables) coincides with the identity.

\begin{figure}
\centering
\includegraphics[scale=0.60]{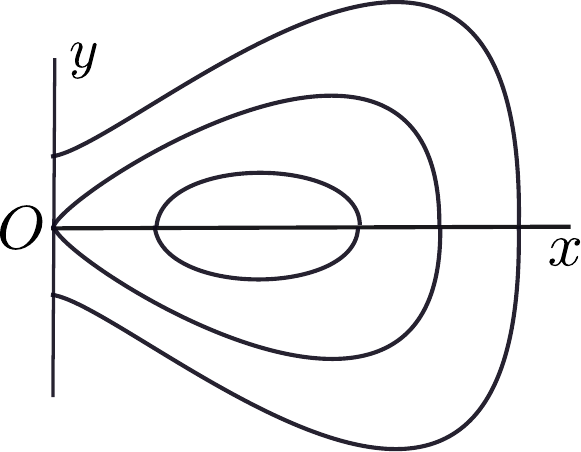}
\caption{Phase space of the $2$-body problem in McGehee coordinates. The stable and unstable manifolds of $O$ coincide along a homoclinic loop which separates the regions of bounded and unbounded motions.}
\label{fig:McGehee}
\end{figure}

In the RPC3BP, for any choice of $\mu\in (0,1/2]$, there also exists a ``fixed point at infinity'' (indeed, when the massless body is very far, the Hamiltonian of the RPC3BP is close to that of the 2BP). To be more precise and study the local dynamics in a neighborhood of this point we make use of McGehee's compactification and define the map 
\begin{equation}\label{eq:mainpoincaremap}
P_\mu=h_{MG}^{-1}\circ f_\mu \circ h_{MG}
\end{equation}
where $f_\mu$ is the one-parameter family of (area preserving) diffeomorphisms \eqref{eq:Poincaremap}. We will see in Section~\ref{subsec:McGehee} that $P_\mu$ is of the form 
\begin{equation}\label{eq:mainpoincaremapformula}
P_\mu=h^{-1}_{MG}\circ f_\mu\circ h_{MG}:(x,y)\mapsto (x+x^3(y+r_\mu(x,y;\mu)),\ y-x^3(x+s_\mu(x,y;\mu)))
\end{equation}
with $r_\mu, s_\mu=\mathcal O_3(x)$ and  real-analytic. A straightforward computation shows that $DP_\mu=\mathrm{Id}$. It is a classical result of McGehee that despite being degenerate, this fixed point has local stable and unstable manifolds \cite{McGeheestablemfold}. However, contrary to the situation in the 2-body problem, the (global) continuation of these manifolds does not coincide along a homoclinic loop and one expects these manifolds to intersect transversally and ``chaotic dynamics'' to take place. Indeed, by analyzing carefully the dynamics of a map of the form \eqref{eq:mainpoincaremapformula}, and the continuation of the local stable and unstable manifolds of the degenerate saddle $O$, Alekseev \cite{AlekseevQR1,AlekseevQR2} and  Moser \cite{Moserbook} showed the existence of hyperbolic sets homoclinically related to $O$ in a different configuration of the restricted 3-body problem known as the Sitnikov example. Then, Llibre and Sim\'o adapted their techniques to deal with the RPC3BP \cite{llibre1980oscillatory}. Since then, quite a number of extensions of these ideas have been used to construct hyperbolic sets in many Celestial Mechanics models and parameter ranges. In very rough terms, one could say that the main tools developed by Alekseev, Moser and the works after them consist on understanding the local dynamics arbitrarily close to $O$ at a $\mathcal C^1 $ level.

The following is our first main result.
\begin{thm}\label{thm:main}
There exists a Newhouse domain $\mathcal N\subset(0,1/2]$ (which contains $\{0\}$ in its closure) for the one-parameter family of diffeomorphisms $P_\mu=h_{MG}^{-1}\circ f_\mu\circ h_{MG}$. Namely, for all $\mu\in\mathcal N$ there exists  $\Lambda_\mu$ which is a basic set for $P_\mu$ and an orbit of homoclinic tangency to $\Lambda_\mu$ which unfolds generically. Moreover, for any $\mu\in\mathcal N$, the basic set $\Lambda_\mu$ is  homoclinically related to the degenerate saddle $O$. 
\end{thm}
\vspace{0.2cm}

\begin{rem}
Of course, Theorem~\ref{thm:main} implies the existence of Newhouse domains for the family $\{f_\mu\}_{\mu\in(0,1/2]}$. We state the theorem for $P_\mu$ because we can then make sense of the degenerate saddle $O$ and the nature of the degeneracy of this fixed point becomes apparent from the expression \eqref{eq:mainpoincaremapformula}.

    Indeed, Theorem~\ref{thm:main}  holds for any one-parameter family of Hamiltonian diffeomorphisms $P_\mu$ of the form \eqref{eq:mainpoincaremapformula} which satisfy $P_\mu^{-1}\circ R=R\circ P_\mu$ (where $R$ is the involution $R(x,y)=(y,x)$) and for $\mu=0$ possess a homoclinic loop to $O$ which splits generically for $\mu\neq 0$ (as long as the form of the splitting can be determined).
\end{rem}
\vspace{0.2cm}

To prove Theorem~\ref{thm:main}, we find a sequence $\{\mu_n\}_{n\in\mathbb N}$ converging to zero for which $P_{\mu_n}$ possesses an orbit of quadratic homoclinic tangency between $W^{u,s}(O)$  and we show that these tangencies unfold generically. The main step in the proof of  Theorem~\ref{thm:main} is then understanding the local dynamics close to $O$ at a $\mathcal C^2$ level. The first key ingredient for this analysis is the construction of a suitable normal form by means of an iterative scheme. From this normal form one can already check that the classical version of the Lambda lemma  does not hold for the degenerate saddle $O$ (see Section~\ref{sec:faillambdalemma}). One of the main novelties of our work is to prove the existence and control the asymptotics of the Shilnikov maps associated to the degenerate saddle $O$. These maps encode the information of the local dynamics of orbits which pass arbitrarily close to $O$ (see Section~\ref{sec:localdynamicsoutline}) and can be used to ``replace'' the classical Lambda Lemma.

The construction of the Shilnikov maps also allows us to develop an asymptotic analysis of suitably renormalized iterations of the map close to the homoclinic tangency (as was done in \cite{MoraRomero} for the case of a hyperbolic fixed point). Surprisingly, despite the essentially nonlinear nature of the flow and several differences with the hyperbolic case, for each $n$ sufficiently large, the (critical) conservative H\'{e}non map shows up in this limit process (as in the hyperbolic case) for suitable open intervals of parameters converging to the bifurcation value $\mu_n$. Then, the existence of a  family of basic sets with persitent tangencies  follows directly from Theorem A in  \cite{Duarteabundance} (see also Theorem 3 in \cite{Gorodetskistdmap}). In order to prove that these sets are homoclinically related to  the degenerate saddle $O$ we  adapt the construction by Palis and Takens in   \cite{PalisTakensbook}, c.f. Proposition~1 of Section~6.4. (see also \cite{Duarteabundancebif}). The adaptation is non-trivial as the classical Lambda lemma does not hold and the proof involves constructing suitable invariant foliations and understanding their behavior under the local dynamics close to $O$.

To state our next result we recall the definition of a homoclinic class. 

\begin{defn}
  Given $\mu\in (0,1/2]$ we define the \textit{homoclinic class} $H(O,P_\mu)$ of the degenerate saddle $O$ as the closure of the union of all transverse homoclinic points to $O$.
\end{defn}

The following result follows from ideas developed by Duarte and Gorodetski in the non-degenerate setting (i.e. for homoclinic phenomena associated to a hyperbolic fixed point). In Section~\ref{sec:pfthm2} we provide a sketch of the proof to see that, making use of the techniques developed in the proof of Theorem~\ref{thm:main}, their arguments do extend to our degenerate setting.

\begin{thm}\label{thm:main2}
    Let $\mathcal N\subset(0,1/2]$ be the Newhouse domain constructed in Theorem~\ref{thm:main} for the one-parameter family of maps $P_\mu$ in \eqref{eq:mainpoincaremap}. There exists a residual subset $\mathcal R\subset \mathcal N$ such that, for all $\mu\in \mathcal R$
    \begin{itemize}
    \item $H(O,P_\mu)$ contains hyperbolic sets of Hausdorff dimension arbitrarily close to 2. In particular,
        \[
        \mathrm{dim}_H\, H(O,P_\mu)=2.
        \]
        \item $H(O,P_\mu)$ is accumulated by generic elliptic points of the map $P_\mu$.
    \end{itemize}
\end{thm}

\begin{rem}
    In \cite{NewhousehtopHD} Newhouse showed that the set of diffeomorphisms for which every homoclinic class is of maximal Hausdorff dimension is $\mathcal C^1$ generic among the diffeomorphisms of a compact surface preserving an area form (see also \cite{DownarowiczNewhouse} for higher degree of smootheness in the dissipative case). However, checking this property in a given parametric family of diffeomorphisms is rather challenging. Understanding how often (in the $\mathcal C^r$ topology for $r\geq 2$) diffeomorphisms in the Newhouse region have homoclinic classes with positive Lebesgue measure is a much harder problem.
\end{rem}
\vspace{0.2cm}

   Although large from the topological viewpoint, residual sets can have zero measure (for instance, the set of Liouville numbers). It would be interesting to understand what is the Lebesgue measure of the set of parameters for which Theorem~\ref{thm:main2} holds. A related question is to understand the set of parameters for which the number of periodic points grows superexponentially.

Finally, we believe it is worth to highlight the following (informal) reformulation of  the result in Theorem~\ref{thm:main2}. 

\begin{thm}\label{thm:main3}
For a residual set of parameters in the Newhouse domain $\mathcal N\subset(0,1/2]$ of the  restricted circular 3-body problem, the union of its elliptic islands forms an unbounded subset of the phase space\footnote{That is, there exist generic elliptic periodic orbits at which the distance between the massless body and the primaries is arbitrarily large.}. Moreover, for these parameters, the closure of the set of generic elliptic periodic orbits contains hyperbolic sets with Hausdorff dimension arbitrarily close to maximal.
\end{thm}

\subsection{The set of oscillatory motions}\label{sec:OM}
In 1922 the French mathematician Jean Chazy classified all the possible final motions of the restricted $3$-body problem.
\begin{thm}[Chazy, 1922 \cite{Chazyclassification}]\label{thm:Chazy}
Every solution of the restricted $3$-body problem defined for all (future) times belongs to one of the following classes
\begin{itemize}
\item B (bounded): $\sup_{t\geq 0} |q(t)|<\infty$.
\item P (parabolic) $|q(t)|\to \infty$ and $|\dot{q}(t)|\to 0$ as $t\to\infty$.
\item H (hyperbolic): $|q(t)|\to \infty$ and $|\dot{q}(t)|\to c>0$ as $t\to\infty$.
\item O (oscillatory) $\limsup_{t\to\infty} |q(t)|=\infty$ and $\liminf_{t\to\infty} |q(t)|<\infty$.
\end{itemize}
\end{thm}
Notice that this classification also applies for $t\to-\infty$. We distinguish both cases adding a superindex $+$ or $-$ to each of the cases, e.g. $H^+$ and $H^-$.
Compared to the other types of motions, Oscillatory Motions are special in the sense that they are a genuine feature of the restricted $3$-body problem (Bounded, Parabolic and Hyperbolic motions exist in the 2 Body Problem). Their existence was first proved by Sitnikov in 1960 \cite{sitnikov1960existence}. A more transparent proof by Moser appeared twenty years later in \cite{Moserbook}. In his approach, Moser understands that the set of Oscillatory Motions, which we will denote by $OM$, corresponds to a particular subset of the \textit{homoclinic class} of the fixed point $O$. Namely, if we denote by $\omega_P,\alpha_P$ the omega and alpha-limit sets associated to the map $P_\mu=h_{MG}^{-1}\circ f_\mu\circ h_{MG}$ (recall that $h_{MG}$ stands for the change of variables to McGehee's coordinates and $P$ is the set of Parabolic motions)
\begin{equation}\label{eq:oscillatorymotions}
OM=\{z=(x,y)\in\mathbb{R}^2\colon O\in \omega_P(z)\cap\alpha_P(z)\}\setminus P.
\end{equation}
Regarding the abundance of the set of oscillatory motions, in his 1970 ICM address Alekseev  conjectured the following \cite{AlekseevICM}.
\begin{conj}[Alekseev]\label{eq:Kolmogorovconj}
    The set $OM$ has Lebesgue measure zero.
\end{conj}
This conjecture, which probably goes back to Kolmogorov \cite{Alekseevfinalmotions} (see the discussion in  \cite{GorodetskiK12}), was coined by Arnold  \textit{the fundamental problem of Celestial Mechanics}. It remains wide open. The only partial results have been obtained by Gorodetski and Kaloshin in \cite{GorodetskiK12}, where they have shown that, for a residual set of parameters, the Hausdorff dimension of $OM$ is maximal.

Their proof can be divided into two halves. In the first part they analyze the existence of quadratic homoclinic tangencies to $O$.  This step, which involves a $\mathcal C^2$ analysis of the dynamics close to $O$ is rather different from our approach. Indeed, the authors in \cite{GorodetskiK12} analyze directly the evolution of 2-jets close to the degenerate saddle. This leads to a rather complicated 4-dimensional system of differential equations. 

We believe that the tools used in the present work (normal form + Shilnikov maps + renormalization) are slightly more conceptual and can be generalized more easily to study homoclinic phenomena in  higher dimensional models such as the planar 3-body Problem. In particular, an extension of our techniques should allow us to compute the so-called separatrix map for the 3-body problem (close to parabolic-elliptic motions). This might be of potential interest for the study of Arnold diffusion and related problems.

From their analysis of the local dynamics and the existence of transverse intersections between $W^{u,s}(O)$, Gorodetski and Kaloshin prove the existence of a hyperbolic periodic orbit $z_\mu$ which is homoclinically related to $O$ and  exhibits a quadratic homoclinic tangency which unfolds generically. This leads to the existence of parameter subsets for which $P_\mu$ possesses  basic sets with persistent tangencies and large Hausdorff dimension which are homoclinically related to $z_\mu$ (and therefore to $O$).

In the second half of \cite{GorodetskiK12}, the authors make use of ideas coming from thermodynamic formalism to obtain a lower bound on the Hausdorff dimension of $OM$. Namely, the arguments above are enough to show that, for a residual subset of parameters, the homoclinic class of $O$ is of maximal Hausdorff dimension. For a hyperbolic fixed point, having a homoclinic class of Hausdorff dimension 2 implies that the set of points which contain it in both their alpha and omega limit sets also have Hausdorff dimension 2. The authors in \cite{GorodetskiK12} show that this classical result carries over to the degenerate saddle $O$ and conclude that, for a residual set of parameters, the set \eqref{eq:oscillatorymotions} has maximal Hausdorff dimension.
\vspace{0.2cm}
\begin{rem}
A priori, the techniques in \cite{GorodetskiK12}, which rely on a $\mathcal C^2$ analysis of the local dynamics are not enough to prove the second part of Theorem~\ref{thm:main2} (and therefore Theorem~\ref{thm:main3}). Moreover, the combination of the techniques developed for the proof of Theorem~\ref{thm:main} and Theorem~\ref{thm:main2} together with the ideas developed by Gorodetski and Kaloshin in the second half of \cite{GorodetskiK12} gives a different (more robust) proof of their result about the abundance of oscillatory motions. 
\end{rem}

\subsection{Newhouse domains in Celestial Mechanics}\label{sec:4bp}
To illustrate the robustness of our techniques, we present two additional results dealing with the existence of Newhouse domains in other popular dynamical systems found in the Celestial Mechanics literature. Namely, analogous results to Theorems~\ref{thm:main},~\ref{thm:main2}~and~\ref{thm:main3} also hold for the Sitnikov and the restricted 4-body problem. For the sake of brevity we only state the analogous results to Theorem~\ref{thm:main3}.

\subsubsection{The Sitnikov model}
In this configuration the two primaries have equal masses and revolve around each other in  Keplerian ellipses of eccentricity $\epsilon\in[0,1)$. The massless body moves on the line perpendicular to the plane of motion of the primaries (the invariance of this line under the dynamics is a consequence of the symmetry of the masses of the primaries). It was introduced by Sitnikov \cite{sitnikov1960existence} to provide the first example of oscillatory motions and has been widely studied in the works of Alekseev and Moser \cite{Alekseevfinalmotions,Moserbook}. For $\epsilon=0$ the primaries move on circles and the system is integrable. For $\epsilon>0$  the dynamics can be reduced to an area-preserving map of $\mathbb R^2$ by studying the time-one map.
\vspace{0.2cm}
\begin{thm}\label{thm:Sitnikov}
   There exists a locally topologically large set of parameters $\epsilon\in (0,1)$ of the  Sitnikov problem for which the union of its elliptic islands forms an unbounded subset of the phase space. Moreover, for these parameters, the closure of the set of generic elliptic periodic orbits contains hyperbolic sets with Hausdorff dimension arbitrarily close to maximal.
\end{thm}

A sketch of the proof of this Theorem, which follows entirely from the ideas deployed to prove Theorems~\ref{thm:main}~and~\ref{thm:main2}, and a formula for the splitting of separatrices of the corresponding parabolic periodic point  at infinity obtained in \cite{SitnikovPerezChavela}, is provided in Appendix~\ref{sec:Sitnikov}.

\subsubsection{The restricted 4-body problem}
Consider the two-parameter family of periodic orbits of the 3-body problem known as the Lagrange triangles (see \cite{Wintner}). We study the motion of a massless body interacting with three massive bodies which move on a Lagrange triangular periodic orbit. This is a parametric family of two degrees of freedom dynamical systems (the parameters are the masses $m_1=\mu_1,m_2=\mu_2$ and $m_3=1-\mu_1-\mu_2$ of the heavy bodies).  Indeed, for all values of the masses, since the Lagrange triangles have \textit{circular} symmetry, it is possible to perform a symplectic reduction and boil down the study of this system to that of a map  $f_{\mu_1,\mu_2,\mathrm{4BP}}$ of the form \eqref{eq:Poincaremap}.
\vspace{0.3cm}

\begin{thm}\label{thm:main4bp}
There exists a (locally) topologically large set of parameters $\mu\in (0,1/2]$ of the  the restricted $4$-body problem with $\mu_1=\mu_2=\mu$ for  which the union of its elliptic islands forms an unbounded subset of the phase space. Moreover, for these parameters, the closure of the set of generic elliptic periodic orbits contains hyperbolic sets with Hausdorff dimension arbitrarily close to maximal.
\end{thm}
\vspace{0.3cm}

The proof of this result follows from the proof of Theorems~\ref{thm:main}~and~\ref{thm:main2} together with the analysis of the invariant manifolds of $O$ carried out in Appendix~\ref{sec:splitting4bp}. Indeed, in Appendix~\ref{sec:splitting4bp} we prove that for $\mu\sim 1/3$ the invariant manifolds of $O$ (for the map $f_{\mu,\mu,\mathrm{4BP}}$) exhibit a quadratic homoclinic tangency which unfolds generically.

\subsection*{Acknowledgments}
J.P. wants to express his gratitude to V. Kaloshin for hintful conversations and valuable suggestions, as well as for sharing with him the unpublished manuscript \cite{GorodetskiK12}.
M.G. has been partially supported by the Spanish Government grants PID2022-136613NB-I00 and PRE2020-096613 and the Catalan Government grant 2021SGR00113. P.M. has been partially supported by
the grant PID2021-123968NB-I00, funded by the Spanish State Research Agency through the
programs MCIN/AEI/10.13039/501100011033 and “ERDF A way of making Europe” and by the Spanish State Research Agency,
through the Severo Ochoa and María de Maeztu Program for Centers and Units of Excellence
in R\&D (CEX2020-001084-M).

\subsection{Organization of the article}
In Section~\ref{sec:outline} we give a sketch of the main steps in the proof of Theorems~\ref{thm:main} and~\ref{thm:main2}, including the introduction of McGehee's compactification and the reduction to a 3-dimensional system with a parabolic periodic orbit at the origin. In Section~\ref{sec:normalform} we study the local flow around this degenerate periodic orbit and then use this analysis to construct in Section~\ref{sec:Shilnikov} the so-called Shilnikov maps. In Section~\ref{sec:tangency} we show the existence of a sequence of parameter values for which there exists a quadratic homoclinic tangency between the stable and unstable manifolds of the parabolic orbit and show that it unfolds generically. Finally, in Section~\ref{sec:returnmap} we show that suitably renormalized iterates of the map close to the homoclinic tangency converge to the H\'{e}non map.

\section{Outline of the proof}\label{sec:outline}
In this section we sketch the main ingredients in the proof of Theorems~\ref{thm:main}~and~\ref{thm:main2}. In Lemmas~\ref{lem:symplecticred}~and~\ref{lem:McGehee} we exploit the symmetries of the system to reduce the dimension and then introduce McGehee's coordinates, what allows us to study unbounded motions. Then, in Proposition~\ref{prop:normalform} we obtain a normal form for the dynamics in McGehee's coordinates. In Proposition~\ref{prop:mainresultrenormalization} we show that a suitably renormalized iterate of the map close to the homoclinic tangency converges to the H\'{e}non map. This implies, making direct use of results of Duarte \cite{Duarteabundance}, the existence of basic sets with persistent tangencies. Then, in Section~\ref{sec:homrelated}  we show that these basic sets are homoclinically related to the degenerate saddle $O$. Finally, Section~\ref{sec:pfthm2} is devoted to provide the details and references which constitute the proof of Theorem~\ref{thm:main2}.

\subsection{Reduction to a 3-dimensional system with a degenerate periodic orbit}\label{subsec:McGehee}
The first step consists on reducing the dimension by exploiting the invariance by rotations. We recall that the change of variables to polar coordinates $\Psi_{pol}:(r,\alpha,y,G)\to (q,p)$ is symplectic. We let 
\[
U_{\mu}(r,\alpha-t)=V_\mu(r\cos\alpha,r\sin\alpha,t) - \frac{1}{r}
\]
and denote by 
\begin{equation}\label{eq:Hamiltonianpolar}
H_{\mathrm{pol},\mu}(r,\alpha,y,G,t)=H_\mu\circ\Psi_{pol}(r,\alpha,y,G,t)= \frac{y^2}{2}+\frac{G^2}{2 r^2}-\frac{1}{r}-U_\mu(r,\alpha-t)
\end{equation}
the Hamiltonian of the system in polar coordinates.
\begin{lem}\label{lem:symplecticred}
The Jacobi constant $\mathcal J=H_{\mathrm{pol},\mu}-G$ is a conserved quantity for the flow of \eqref{eq:Hamiltonianpolar}. Moreover, if we fix any value $J$ of the Jacobi constant  sufficiently large, the dynamics of the Hamiltonian system \eqref{eq:Hamiltonianpolar} on the level $\{\mathcal J=J\}$ is given by the  Hamiltonian
\[
H_{\mathrm{red},\mu}(r,y,\phi;J)= \frac{y^2}{2}+\frac{J^2}{2 r^2}-\frac{1}{r}+\frac{1}{r^2}\mathcal{O}(r^{-1}, y^2),\qquad (r,y,\phi)\in\mathbb{R}_+\times\mathbb{R}\times\mathbb{T}.
\]
\end{lem}

Although this lemma is standard, for the sake of self completeness, we provide a proof in the beginning of Section~\ref{sec:normalform}. We now apply McGehee's change of variables composed with a rotation of angle $\pi/4$. The following lemma is a simple computation.
\begin{lem}\label{lem:McGehee}
Let $\widetilde\Psi_{MG}$ be the change of variables given by $\widetilde\Psi_{MG}=\Psi_{MG}\circ R_{\pi/4}(q,p)\mapsto (r,y)$ where 
\[
R_{\pi/4}:(q,p)\mapsto (\frac{1}{\sqrt{2}}(q+p),\frac{1}{\sqrt{2}}(p-q)) \qquad\qquad \Psi_{MG}:(x,y)\mapsto (2/x^{2},y)
\]
 Then, 
\[
\Psi^*(\mathrm{d}r\wedge\mathrm{d}y)=\frac{-2^{7/2}}{(q+p)^3}\mathrm{d}q\wedge\mathrm{d}p
\] 
and 
\begin{equation}\label{eq:McgeheeHam}
K_\mu(q,p,\phi)=H_{\mathrm{red},\mu}\circ\Psi(q,p,\phi)=- qp+\mathcal{O}_4(q,p).
\end{equation}
Namely, (after time rescaling) the equations of motion in coordinates $(q,p,t)$ are
\begin{equation}\label{eq:flownearinfty3d}
\dot q=(q+p)^3\partial_p K(q,p,\phi),\qquad \dot p=-(q+p)^3\partial_q K(q,p,\phi),\qquad \dot \phi=-2^{7/2}.
\end{equation}
\end{lem}

The dynamical system \eqref{eq:flownearinfty3d} can be seen as a time-dependent Hamiltonian system given by the time-dependent Hamiltonian $K(q,p,\phi)$ on the (singular) symplectic manifold $(\mathbb{R}^2, \omega)$ where 
\[
\omega=\frac{1}{(q+p)^3}\mathrm{d}q\wedge\mathrm{d}p.
\]

\subsection{Local dynamics close to $O$}\label{sec:localdynamicsoutline}
In the next section we analyze the dynamics on a neighborhood of the (degenerate) periodic orbit
\[
\mathcal{O}=\{q=p=0,\  \phi\in\mathbb{T}\}.
\]
It is well known that, on a neighborhood of a hyperbolic periodic orbit of a real-analytic Hamiltonian of 1+1/2  degrees of freedom it is possible to analytically conjugate the dynamics to an integrable (autonomous) real-analytic Hamiltonian  known as the Birkhoff-Moser normal form. The proof of the convergence of this conjugacy, which first appeared in the work of Birkhoff, was established by Moser in \cite{Moseranalyticinvariants}, who showed that the radius of convergence of its Taylor series at the periodic orbit is strictly positive. 

The  periodic orbit $\mathcal O$ is degenerate so Moser's result does not apply. Indeed, we know from the work of \cite{McGeheestablemfold} that the local unstable and stable manifolds $W^{u,s}(\mathcal O)$ are, a priori, only $\mathcal C^\infty$ at $\mathcal O$ so any conjugacy straightening these invariant manifolds cannot, a priori, be real-analytic at $\mathcal O$. In the following proposition we obtain a suitable normal form for the flow around the degenerate periodic orbit $\mathcal{O}$.

\begin{prop}\label{prop:normalform}
Let $\bar \kappa>0$ be fixed and small enough and consider, for each $\rho>0$, the sectorial domain 
\begin{equation}\label{eq:sectorialdom}
V_\rho=\{z\in B_{\rho}(0)\subset\mathbb C \colon |\mathrm{Im}z| < \rho \kappa |\mathrm{Re}z|\}.
\end{equation}
For any $N \ge 3$, there exist $\bar \kappa>0$ and $\rho>0$ small enough such that, for any $\mu\in (0,1/2]$ there exists a real-analytic change of variables $\Omega: V_{\rho/2}\times V_{\rho/2}\times\mathbb T_\rho\to V_{\rho}\times V_{\rho}\times\mathbb T_\rho$, such that
\begin{itemize}
\item $\Omega$ preserves the form $\omega$, i.e. $\Omega_*\omega=\omega$,
\item In the new variables
\begin{equation}\label{eq:normalform}
\mathcal K(q,p,\phi)=K \circ\Omega(q,p,\phi)=-qp+q^4p^4\  k(q,p,\phi),\qquad k(q,p,\phi)= \mathcal{O}_N(q,p)
\end{equation}
uniformly on $V_{\rho/2}\times V_{\rho/2}\times\mathbb T_\rho$.
\end{itemize}
\end{prop}

The proof of Proposition~\ref{prop:normalform} is deferred to Section~\ref{sec:normalform}. It consists of two main steps. First, we  introduce a change of variables $\chi$ which straightens the stable and unstable manifolds, i.e. in the new variables $\{q=0\}$ (respectively $\{p=0\}$) corresponds to the local unstable (stable) manifold so 
\[
K\circ\chi(q,p,\phi)=-qp(1+\mathcal{O}_1(q,p)).
\]
The existence of this change of variables is an easy consequence of the mere existence of suitable parametrizations of these local invariant manifolds, which was shown in \cite{McGeheestablemfold}.  Then, to obtain the normal form \eqref{eq:normalform} we develop an iterative scheme to successively find symplectic changes of variables  (with respect to the singular form $\omega$) which kill terms of increasingly higher order. Then, we prove that the limit transformation (obtained as an infinite product of close to identity transformations) is well defined. 
\vspace{0.2cm}
\begin{rem}
    We notice that the normal form \eqref{eq:normalform} is a priori non-integrable.
    Understanding whether one can conjugate the dynamics on a neighborhood of $\mathcal O$ to an integrable normal form would require a much deeper analysis. For our purposes it is enough to conjugate the system to the normal form \eqref{eq:normalform}.
\end{rem}
 \vspace{0.2cm}

\subsubsection{Failure of the classical Lambda lemma}\label{sec:faillambdalemma}
The local flow associated to the Hamiltonian $\mathcal K$ in Proposition~\ref{prop:normalform} is essentially nonlinear. Namely
\begin{equation}\label{eq:lambdalemmaflow}
\dot q=(q+p)^3\partial_p \mathcal K(q,p,\phi)\qquad\qquad\dot p=-(q+p)^3\partial_q \mathcal K(q,p,\phi)
\end{equation}
We now observe that the classical version of the Lambda lemma (see for example Lemma 7.1 in \cite{PalisdeMelo}) does not hold. Indeed, take a point $z=(q,0)\in W_{loc}^s(\mathcal O)$ and consider the vector $v=(0,1)\in T_z^*\mathbb{R}^2$. Then, it is an easy exercise to check that the evolution of $v$ under the flow \eqref{eq:lambdalemmaflow} is given by the linear differential equation
\[
\dot v=q^3(t)\begin{pmatrix}
    -4&-3\\
    0&1
\end{pmatrix} v
\]
where 
\[
q(t)=\frac{q}{(1+3q^3t)^{1/3}}.
\]
Therefore, reparametrizing $\mathrm{d}s/\mathrm{d}t=q^{3}(t)$ the vector $\tilde v(s)=v(t(s))$ is given by 
\[
\tilde v(s)=\exp(s A)v\qquad\qquad A=\begin{pmatrix}
    -4&-3\\
    0&1
\end{pmatrix} 
\]
and, after an straightforward computation
\[
\lim_{t\to+\infty} \frac{1}{\lVert v(t)\rVert}v(t)=\lim_{s\to+\infty} \frac{1}{\lVert \tilde v(s)\rVert}\tilde v(s)= \frac{1}{\sqrt{34}}(-3,5) .
\]
That is, the image of a disk transverse to $W^s_{loc}(\mathcal O)$ may accumulate to $W^u(\mathcal O)$ but this accumulation is not $\mathcal C^1$ at $\mathcal O$. In the following we will see that this annoyance is indeed rather irrelevant for our purposes. Indeed, we will  show that the accumulation is $\mathcal C^2$ away from $\mathcal O$ \footnote{In his construction of non-trivial hyperbolic sets homoclinically related to $\mathcal O$ in \cite{Moserbook}, Moser already observed that the accumulation is at least $\mathcal C^1$ away from $\mathcal O$.}. The main tool that we will use to describe the local dynamics are the so-called Shilnikov maps, which we introduce in the next section.

\subsubsection{The Shilnikov stable and unstable maps}
Proposition~\ref{prop:normalform} allows us to control the dynamics of points whose orbit spends arbitrarily large times close to the degenerate saddle periodic orbit $\mathcal O$. This is crucial to prove, in Proposition~\ref{prop:Shilnikov}, the existence and asymptotics of the so-called Shilnikov maps (see \cite{Shilnikov67a,Shilnikov67b} and see also  \cite{MoraRomero,GSTtangencies} for their use in the renormalization scheme at a quadratic tangency), also denoted in the literature as Shilnikov data. Before stating Proposition~\ref{prop:Shilnikov}, for the sake of self-completness, we find convenient to first recall the definition of the Shilnikov maps in the case of a hyperbolic fixed point of a 2-dimensional map. To that end we consider for a moment a diffeomorphism $f:\mathbb{R}^2\to\mathbb{R}^2$ for which the origin $O=(0,0)$ is a hyperbolic fixed point and assume that $W^{s}_{loc}(O)=\{p=0\}$ and $W^{u}_{loc}(O)=\{q=0\}$. Let $a>0$ be a sufficiently small fixed constant and define the subset
\begin{equation}\label{eq:dfnsigmadelta}
\Sigma=\{(\xi,\eta)\in\mathbb R^2\colon |\xi-a|,|\eta-a|\leq a/2\}.
\end{equation}
Define also, for any $(\xi_0,\eta_0)\in \Sigma$, the sections
\[
\sigma^s_{\xi_0}=\{q=\xi_0,\ 0\leq p\leq a/4 \},\qquad\qquad \sigma^u_{\eta_0}=\{p=\eta_0,\  0\leq q\leq a/4\},
\]
which are, respectively, transverse to the local stable and local unstable manifolds. For any $n\in\mathbb N$ denote by $\gamma^s_{\xi_0,n}$ the unique connected component of $f^n(\sigma^s_{\xi_0})\cap \{(q,p)\in[0,3a/2]^2\}$ which contains the point $f^{n}(\xi_0,0)$ and  by $\gamma^u_{\eta_0,n}$ the unique connected component of $f^{-n}(\sigma^u_{\eta_0})\cap \{(q,p)\in[0,3a/2]^2\}$ which contains the point $f^{-n}(0,\eta_0)$ (see Figure~\ref{fig:Shilnikov}). Then, by the $\lambda$-lemma (see \cite{PalisdeMelo}, for instance), if $a$ is small enough, for any $n\in\mathbb N$ large enough, the sets
\[
\gamma^s_{\xi_0,n}\cap \sigma^u_{\eta_0}, \qquad\qquad \gamma^{u}_{\eta_0,n}\cap \sigma^s_{\xi_0},
\]
consist each of a unique point (see Figure~\ref{fig:Shilnikov}). The maps
\begin{equation}
\begin{array}{rl}
    F^s_{n}:\Sigma  & \to\mathbb R^2 \\
    (\xi,\eta) &\mapsto  \gamma^{u}_{\eta,n}\cap \sigma^s_\xi
\end{array}\qquad\qquad \begin{array}{rl}
    F^u_{n}:\Sigma  & \to\mathbb R^2 \\
    (\xi,\eta) &\mapsto  \gamma^s_{\xi,n}\cap \sigma^u_\eta
\end{array}
\end{equation}
are known as the Shilnikov stable and unstable maps associated to the map $f$ and the hyperbolic fixed point $O$.

\begin{figure}
\centering
\includegraphics[scale=0.65]{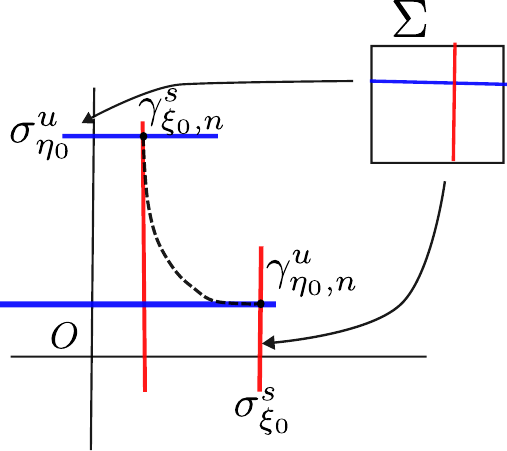}
\caption{Construction of the Shilnikov data.}
\label{fig:Shilnikov}
\end{figure}

In order to introduce these maps in the context of the Hamiltonian $\mathcal K$ in Proposition~\ref{prop:normalform}, we first consider the return map on the section $\{\phi=0\}$, which we denote by  $\Psi_{\mathcal K}:(q,p)\mapsto (\bar q,\bar p)$. Clearly $O=(0,0)$ is a (degenerate) fixed point for $\Psi_{\mathcal K}$ and its local stable (unstable) manifold is given by the $q$ axis ($p$ axis). Let $a>0$ be now such that
\[
\{(q,p)\in[0,3a/2]^2\}\subset V_\rho \times V_\rho.
\]
For any $n\in\mathbb N$ denote by $\psi^s_{\xi_0,n}$ the unique connected component of $\Psi_{\mathcal K}^n(\sigma^s_{\xi_0})\cap \{(q,p)\in[0,3a/2]^2\}$ which contains the point $\Psi^{n}_{\mathcal K}(\xi_0,0)$ and  by $\psi^u_{\eta_0,n}$ the unique connected component of $\Psi^{-n}_{\mathcal K}(\sigma^u_{\eta_0})\cap \{(q,p)\in[0,3a/2]^2\}$ which contains the point $\Psi^{-n}_{\mathcal K}(0,\eta_0)$. Suppose that, for $a>0$ sufficiently small and $n\in\mathbb N$ large enough, for any $(\xi_0,\eta_0)\in \Sigma$, the sets 
\begin{equation}\label{eq:Shilnikovassumption}
\psi^s_{\xi_0,n}\cap \sigma^u_{\eta_0} \qquad\qquad \psi^{u}_{\eta_0,n}\cap \sigma^s_{\xi_0},
\end{equation}
consist each of a unique point. Then, as in the case of a hyperbolic fixed point, we can define the stable and unstable Shilnikov maps by
\begin{equation}\label{eq:Shilnikovmapsdefn}
\begin{array}{rl}
    \Psi^s_{n}:\Sigma  & \to\mathbb R^2 \\
    (\xi,\eta) &\mapsto  \psi^{u}_{\eta,n}\cap \sigma^s_\xi
\end{array}\qquad\qquad \begin{array}{rl}
    \Psi^u_{n}:\Sigma  & \to\mathbb R^2 \\
    (\xi,\eta) &\mapsto  \psi^s_{\xi,n}\cap \sigma^u_\eta
\end{array}
\end{equation}

Of course, a priori, these maps might not be well defined. Namely, the assumption that \eqref{eq:Shilnikovassumption} consist of a unique point needs to be justified since the fixed point $O$ is degenerate so the classical $\lambda$-lemma cannot be applied (see Section~\ref{sec:faillambdalemma}). In the next proposition we show that the maps \eqref{eq:Shilnikovmapsdefn} are indeed well defined and, moreover, we provide asymptotic expressions for them. For the renormalization scheme in Section~\ref{sec:renormalization}, it will be crucial that we are able to describe the behavior of these maps on a suitable complex neighbourhood of $\Sigma$ in \eqref{eq:dfnsigmadelta}. To that end, for $\delta>0$ we let 
\begin{equation}\label{eq:complexextensionsigma}
\Sigma_\delta=\{\tau=(\xi,\eta)\in\mathbb{C}^2\colon \mathrm{dist}(\tau,\Sigma)\leq \delta\}.
\end{equation}

\begin{prop}\label{prop:Shilnikov}
Let $a>0$ be sufficiently small. Then, there exists  $T_0\gg1$ such that for all $\mu\in (0,1/2]$, all $T\in\mathbb N$,  $T\geq T_0$ and $\delta\gtrsim T^{-1}$,
the Shilnikov stable and unstable maps $\Psi^{s}_T$, $\Psi^u_T$ in \eqref{eq:Shilnikovmapsdefn} are well defined real-analytic maps on $\Sigma_\delta$  of the form 
\begin{equation}\label{eq:formShilnikovmaps}
\Psi^{s}_T(\xi,\eta)=(\xi,y_T(\xi,\eta)),\qquad\qquad \Psi^{u}_T(\xi,\eta)=(x_T(\xi,\eta),\eta),
\end{equation}
Moreover, uniformly for $(\xi,\eta)\in \Sigma_\delta$,
\begin{itemize}
    \item They satisfy the asymptotic expressions
\begin{equation}\label{eq:asymptoticShilnikov}
\begin{split}
x_T(\xi,\eta)= &\left(\frac{\pi}{16 T}\right)^{2/3} \eta^{-1}\left( 1+\mathcal{O}(T^{-1})\right),\\
y_T(\xi,\eta)=& \left(\frac{\pi}{16 T}\right)^{2/3} \xi^{-1}\left( 1+\mathcal{O}(T^{-1})\right),
\end{split}
\end{equation}
\item Their first derivatives satisfy
\begin{align}\label{eq:asymptoticShilnikovderiv}
\partial_\eta x_T(\xi,\eta)= &-\left(\frac{\pi}{16 T}\right)^{2/3} \eta^{-2}\left(1+\mathcal O(T^{-1})\right)\\
\partial_\xi x_T(\xi,\eta)=&\frac{2}{3T^{5/3}}\left(\frac{\pi}{16 }\right)^{2/3} (\xi\eta)^{-5/2}\left(1+\mathcal O(a)\right)\label{eq:asymptoticShilnikovderiv2}
\end{align}
and 
\begin{align}\label{eq:asymptoticShilnikovderivy}
\partial_\xi y_T(\xi,\eta)= &-\left(\frac{\pi}{16 T}\right)^{2/3} \xi^{-2}\left(1+\mathcal O(T^{-1})\right)\\
\partial_\eta y_T(\xi,\eta)=&\frac{2}{3T^{5/3}}\left(\frac{\pi}{16 }\right)^{2/3} (\xi\eta)^{-5/2}\left(1+\mathcal O(a)\right)\label{eq:asymptoticShilnikovderiv2y}
\end{align}
\item For $\alpha=(\alpha_\xi,\alpha_\eta)$ with $|\alpha|=2$ (in multiindex notation)
\[
\partial^\alpha x_T(\xi,\eta)=\mathcal{O}(T^{-2/3})\qquad\qquad \partial^\alpha y_T(\xi,\eta)=\mathcal{O}(T^{-2/3})
\]
\end{itemize}
\end{prop}

As we will see in Sections~\ref{subsec:tangencies}, ~\ref{sec:renormalization},~\ref{sec:pfthm1}~and~\ref{sec:pfthm2}, the Shilnikov maps are an extremely useful tool to understand the homoclinic picture at the degenerate saddle $O$. The fact that $\partial_\xi x_T\neq 0$ and  $\partial_\eta y_T\neq 0$ guarantees that $\Psi_{T}^{s,u}$ are local diffeomorphisms. This is crucial for developing the renormalization scheme in Section~\ref{sec:renormalization}.

\begin{rem}
    From now on, we will work on three different scales and we make use of three different coordinate systems:
    \begin{itemize}
        \item We will denote by $(q,p)$ the macroscopic coordinates introduced in Proposition~\ref{prop:normalform}.
        \item We will use $(\xi,\eta)$ to denote coordinates on $\Sigma$ (the domain for the Shilnikov maps). One can think of these coordinates as describing the intermediate scale as they zoom in regions which are at distance $\mathcal O(T^{-2/3})$ of the local invariant manifolds.
        \item In Proposition~\ref{prop:mainresultrenormalization} we introduce microscopic coordinates $(Q,P)$ which zoom in a region of size $\mathcal O(T^{-5/3}, T^{-10/3})$.
    \end{itemize}
\end{rem}

\subsection{Secondary homoclinic tangencies to the degenerate saddle~$O$}\label{subsec:tangencies}

It has been shown in \cite{MartinezPinyol} that, for all $\mu>0$ sufficiently small, $W^{s}(O)\pitchfork W^u(O)\neq\emptyset$. For a sequence of parameter values one can also prove the existence of (secondary) quadratic homoclinic tangencies.

\begin{prop}\label{lem:homoclinictangencies}
There exists a sequence $\{\mu_n\}_{n\in\mathbb N}$ converging to zero such that, for every $\mu\in\{\mu_n\}_{n\in\mathbb N}$, there exists a point $z_n\in \mathbb{R}^2$ at which the map $P_{\mu_n}$ has a quadratic homoclinic tangency between $W^{u,s}(O)$ which unfolds generically with $\mu$.
\end{prop}
The proof of Proposition~\ref{lem:homoclinictangencies} follows from an idea which already appears in the work of Duarte \cite{Duarteabundance} (see also \cite{GorodetskiK12}). We however provide a more direct proof in Section~\ref{sec:tangency} which makes use of the Shilnikov stable and unstable maps constructed in Proposition~\ref{prop:Shilnikov}.

\subsection{Renormalization in a neighborhood of a homoclinic tangency}\label{sec:renormalization}
We finally prove that large enough iterates of the map $\Psi_{\mathcal K}$ (recall that we denote by $\Psi_{\mathcal K}$ the return map to the section $\{\phi=0\}$ associated to the flow of the Hamiltonian $\mathcal K$ in Proposition~\ref{prop:normalform}) close to the homoclinic tangency can be renormalized, with the critical H\'{e}non map showing up in the limit process. This can be seen as an extension of the result in \cite{MoraRomero} to the degenerate saddle $O$. The main tool involved in this construction are the Shilnikov maps constructed in Proposition~\ref{prop:Shilnikov}, which in the present setting are essentially nonlinear (in the case of a hyperbolic fixed point these maps can be approximated by linear maps).

In order to describe the domain where the rescaled dynamics converges to the H\'{e}non map we denote by $B_{10}\subset \mathbb C^2$ the complex ball of radius $10$ centered at the origin. 

\begin{prop}\label{prop:mainresultrenormalization}
    Let $P_\mu$ denote the Poincar\'{e} map in \eqref{eq:mainpoincaremap} and let $\mu\in (0,1/2)$ be such that $W^s(O)$ and $W^u(O)$ possess a quadratic homoclinic tangency which unfolds generically. Denote by $z$ the corresponding  homoclinic tangency for the map $P_{\mu}$. Then,  there exists $L\in\mathbb{N}$ such that, for any $T\in\mathbb N$ sufficiently large, there exist:
    \begin{itemize}
        \item An affine reparametrization $\varphi_{T}:[-2,2]\to\mathbb{R}$, which maps $\varphi_{T}:\kappa\mapsto \mu=\varphi_{T}(\kappa)$,
        \item A real-analytic change of variables $\Psi_{\kappa,T}: B_{10}\to \mathbb{C}^2$
    \end{itemize}
    such that, as $T\to\infty$,
    \[
    \varphi_{T}([-2,2])\to \mu,\qquad\qquad \Psi_{\kappa,T}(B_{10})\to z,
    \] 
    and,  uniformly for $(Q,P)\in B_{10}$, \begin{equation}\label{eq:convergenceHenon}
    \Psi_{\kappa,T}^{-1}\circ P_{\varphi_{T}(\kappa)}^{L+T}\circ \Psi_{\kappa,T}(Q,P)=(P,\kappa-Q-P^2)+\mathcal{O}(T^{-2/3}).
    \end{equation}
\end{prop}
\vspace{0.3cm}

\begin{rem}
It may come as a surprise that we can actually show that \eqref{eq:convergenceHenon} holds uniformly on the complex ball $B_{10}$ (which is a fixed domain independent of $T)$. Indeed, the change of variables $\Psi_{\kappa,T}$ is defined through the Shilnikov maps \eqref{eq:Shilnikovmapsdefn}, which as we have seen in Proposition~\ref{prop:Shilnikov}, are only defined on a $\delta$ complex neighbourhood of $\Sigma$ with $\delta\gtrsim T^{-1}$. The reason why \eqref{eq:convergenceHenon} holds uniformly on $B_{10}$ is that $\Psi_{\kappa,T}$ also involves a suitable rescaling by a factor $\alpha\sim T^{-5/3}$. Namely, in the original variables (in which the Shilnikov maps are defined) we are looking at regions of size $T^{-5/3}$, which is much smaller than $\delta$, around a point contained in the real plane.
\end{rem}\vspace{0.2cm}
\begin{rem}\label{rem:convergence}
We highlight that, for any fixed $r\in\mathbb N$, the estimate in \eqref{eq:convergenceHenon} guarantees that, for sufficiently large $T$, the map   $\Psi_{\kappa,T}^{-1}\circ P_{\varphi_T(\kappa)}^{L+T}\circ \Psi_{\kappa,T}$ converges to the H\'{e}non map in the $\mathcal C^{r}$ topology uniformly for $(\xi,\eta)\in[-5,5]^2$.
\end{rem}
\vspace{0.3cm}

The proof of Proposition~\ref{prop:mainresultrenormalization} is deferred to Section~\ref{sec:returnmap}. We are now ready to finish the proof of Theorem~\ref{thm:main}.

\subsection{Proof of Theorem~\ref{thm:main}}\label{sec:pfthm1}

We divide the proof in two parts.
\subsubsection{Basic sets with robust homoclinic tangencies}
The first part of Theorem~\ref{thm:main} follows from classical results of Duarte \cite{Duarteabundance}. Indeed, Theorem A in that paper shows the existence of a Newhouse domain, which contains $\kappa=-1$ in its closure,  for the H\'{e}non map
\[
(Q,P)\mapsto (P,\kappa-Q-P^2).
\]
 By \eqref{eq:convergenceHenon}, for any $\kappa\in[-2,2]$, for a fixed $L$ large enough and any  sufficiently large $T$,  there exists a value of $\mu=\varphi_{n,T}(\kappa)$ and a domain in which the  map $P_\mu^{L+T}$ is, modulo a conjugacy,  $\mathcal{O}(T^{-2/3})$ close in the $\mathcal C^2$ topology (see Remark~\ref{rem:convergence}) to the H\'{e}non map with a parameter $\kappa$. This implies that, for each $n\in\mathbb N$ sufficiently large and each $\mu_n$ as in Proposition~\ref{lem:homoclinictangencies}, there exists an open set $U_n$ containing $\mu_n$ in its closure such that, for any $\mu\in U_n$, there exists a basic set $\Lambda_\mu$ of $P_\mu$ which exhibits robust homoclinic tangencies.

\subsubsection{The basic sets $\Lambda_\mu$ are homoclinically related to $\mathcal O$}\label{sec:homrelated}

It remains to prove that the sets $\Lambda_\mu$ are homoclinically related to the degenerate saddle $\mathcal O$. The proof follows the same lines as that of Proposition 1, Chapter 6.4 in the book \cite{PalisTakensbook} by Palis and Takens (see also \cite{Duarteabundancebif}), for the non-degenerate case (i.e. wild sets homoclinically related to a hyperbolic fixed point). However, some of the objects in their construction are not well defined in our setting and must be replaced by appropriate substitutes. We therefore reproduce their proof and indicate, when needed, the corresponding modifications. Throughout this section we fix a value of $n\in\mathbb N$ and omit the dependence on $n$ from the notation.

\begin{figure}
\centering
\includegraphics[scale=0.85]{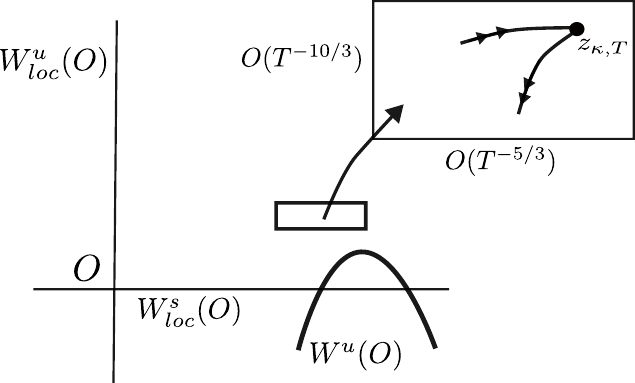}
\caption{The hyperbolic fixed point $z_{\kappa,T}$ and its local invariant manifolds.}
\label{fig:Localmanifolds}
\end{figure}

Roughly speaking the idea of the proof is the following. Fix a value of $T$ sufficienly large. Then, for any $\kappa$ close to $-1$, the map $P^{L+T}_{\varphi_T(\kappa)} $ has a hyperbolic fixed point $Z_{\kappa,T}$ which is contained in the basic set $\Lambda_{\varphi_T(\kappa)}$ (see Figure~\ref{fig:Localmanifolds}). We denote by 
\[
z_{\kappa,T}=\Psi_{\kappa,T}(Z_{\kappa,T})
\]
the corresponding point in $(q,p)$ coordinates (the coordinates given by Proposition~\ref{prop:normalform}). Let $W^{u,s}(z_{\kappa,T})$ denote their unstable and stable manifolds and suppose that there exists a compact piece of $W^{u}(z_{\kappa,T})$ (resp. $W^{s}(z_{\kappa,T})$) which is arbitrarily close (taking $T$ large enough) to a compact piece of $W^{u}(O)$ (resp. $W^s(O)$) which covers several (more than one) fundamental domains. Then, as $W^{u}(O)$ and $W^s(O)$ intersect transversally \cite{MartinezPinyol}, for $T$ large enough $W^{u}(z_{\kappa,T})\pitchfork W^s(O)\neq\emptyset$ and $W^{s}(z_{\kappa,T})\pitchfork W^u(O)\neq\emptyset$.

The main difficulty to carry out this argument is that, Proposition~\ref{prop:mainresultrenormalization} only gives information on very small pieces of $W^{u,s}_{loc}(z_{\kappa,T})$. In the following we show how to enlarge these pieces while controlling their geometry. 
To that end, we define the sets 
\[
V=\Psi^u_T([a/2,3a/2]^2),\qquad\qquad H=\Psi^s_T([a/2,3a/2]^2),
\]
and the vertical and horizontal foliations $\mathcal F_v, \mathcal F_h$ of $V$ and $H$ whose leaves correspond, respectively, to the curves (see Figure~\ref{fig:Foliations})
\begin{align*}
\gamma_{v,\xi}(\eta):[-a/2,a/2]&\mapsto \Psi^{u}_T(a+\xi,a+\eta),\qquad\qquad \xi\in[-a/2,a/2],\\
\gamma_{h,\eta}(\xi):[-a/2,a/2]&\mapsto \Psi^{s}_T(a+\xi,a+\eta),\qquad\qquad \eta\in[-a/2,a/2].
\end{align*}

\begin{figure}
\centering
\includegraphics[scale=0.95]{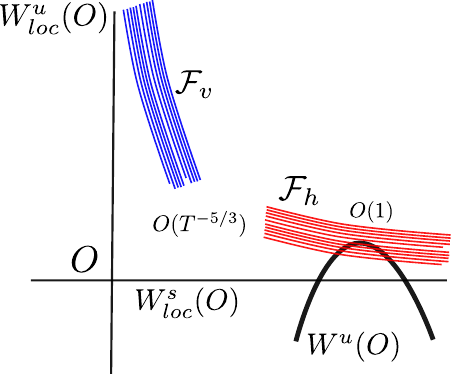}
\caption{The horizontal and vertical foliations constructed using the Shilnikov maps.}
\label{fig:Foliations}
\end{figure}

The foliations $\mathcal F_v,\mathcal F_h$ substitute the classical foliations associated to the linearizing coordinates in a neighborhood of a hyperbolic fixed point (which are the ones used in \cite{PalisTakensbook}). The main properties of these foliations which we will use are gathered in Lemmas~\ref{lem:c1accumulation},~\ref{lem:conefields}~and~\ref{lem:lineoftangencies}. Namely, Lemma~\ref{lem:c1accumulation} shows that the leaves of these foliations can be used as a reference against which images of pieces of the local manifolds of $Z_{\kappa,T}$ can be compared to. Then, Lemma~\ref{lem:conefields} shows that these foliations also carry dynamical information about the local dynamics around the saddle, what is key to guarantee that iterates of suitable pieces of the invariant manifolds of $Z_{\kappa,T}$ approach the leaves of these foliations and, moreover, are enlarged. Finally, in Lemma~\ref{lem:lineoftangencies} we control the angle between $\mathcal F_h$ and the image of $\mathcal F_v$ under the global dynamics.

\begin{lem}\label{lem:c1accumulation}
For $T$ large enough, the leaves of $\mathcal F_v$ (resp. $\mathcal F_h$) are arbitrarily close in the $\mathcal C^1$ topology to the compact arc
    \[
    \gamma^u=W^{u}_{loc}(O)\cap \{a/2\leq p\leq a\},\qquad\qquad (\text{resp. } \gamma^s=W^{s}_{loc}(O)\cap \{a/2\leq q\leq a\}),
    \]
    which covers several fundamental domains for $a$ small enough (but independent of $T$).
\end{lem}
\begin{proof}
    The proof is an straightforward consequence of the definition of the foliations $\mathcal F_{v,h}$ and the asymptotic formulas in Proposition~\ref{prop:Shilnikov}. The fact that $\gamma^{u,s}$ cover several fundamental domains for $a$ small enough follows from the fact that on $\gamma^u$ we have $\dot p=p^4(1+\mathcal O(p^3))$ and on $\gamma^s$ we have $\dot q=-q^4(1+\mathcal O(q^3))$ (see Proposition~\ref{prop:normalform}).
\end{proof}

We now show that the foliations $\mathcal F_v,\mathcal F_h$ can be used to define a ``cone field'' which encodes the local dynamics. To that end, the introduction of some notation is in order. We denote by $\Psi^T_{loc}:H\to V$ the map
\begin{equation}\label{eq:defnlocalmap}
\Psi_{loc}^T =\Psi_{\mathcal K}^T|_{H},
\end{equation}
which describes the local dynamics. We also introduce the map $\Psi_{glob}:V\to \mathbb R^2$
\[
\Psi_{glob}=\Psi_{\mathcal{K}}^L|_V,
\]
which describes the global dynamics from $V$ to (a neighbourhood of) $H$. 
\begin{lem}\label{lem:conefields}
    Let $(\xi,\eta)\in[a/2,3a/2]^2$ and let $T$ be sufficiently large. Let $(q,p)=\Psi_T^s(\xi,\eta)\in H$ and let $w\in T_{(q,p)} H$ such that 
    \[
     \angle (w,\mathcal F_h)\gtrsim T^{-5/3}.
    \]
    Then, $\Psi_{loc}^T(q,p)\in V$,
    \[
    \lVert D\Psi_{loc}^T(q,p)w \rVert \gtrsim T^{5/3}\lVert w\rVert 
    \]
    and 
     \[
     \angle (D\Psi^T_{loc}(q,p)w,\mathcal F_v)\lesssim T^{-5/3}.
    \]
\end{lem}
\begin{proof}

The idea is to show that the contracting eigendirection associated to the local map $\Psi_{loc}^T$ is sufficiently close to the tangent vector to the horizontal foliation $\mathcal F_h$. Since $\Psi^T_{loc}=\Psi^u_T\circ (\Psi^s_T)^{-1}$ it is an easy computation to show that, if we write
\[
\hat \alpha(q,p)=\partial_\xi x_T(q,y_T^{-1}(q,p)) \qquad\qquad \hat \beta(q,p)=\partial_\eta x_T(q,y_T^{-1}(q,p))
\]
and
\[
\tilde\beta(q,p)=\partial_\xi y_T(q,y_T^{-1}(q,p)) \qquad\qquad \tilde\alpha(\xi,\eta)=\partial_\eta y_T(q,y_T^{-1}(q,p))
\]
then
\[
D \Psi^T_{loc}=\begin{pmatrix}
\hat\alpha& \hat\beta\\
0&1
\end{pmatrix}\begin{pmatrix}
1&0\\
-\frac{\tilde\beta}{\tilde\alpha}& \frac{1}{\tilde\alpha}
\end{pmatrix}=\begin{pmatrix}
-\frac{\hat\beta\tilde\beta}{\tilde\alpha}+\hat\alpha& \frac{\hat\beta}{\tilde\alpha}\\
-\frac{\tilde\beta}{\tilde\alpha}&\frac{1}{\tilde\alpha},
\end{pmatrix}
\]
and is an easy computation to check that the eigenvalues of $D\Psi_{loc}^T$ are given by 
\[
\lambda_\pm=\frac{\gamma\pm\sqrt{\gamma^2-4\tilde\alpha^2}}{2\tilde\alpha}\qquad\qquad\gamma=1-\hat\beta\tilde\beta+\hat\alpha\tilde\alpha.
\]
From the definition of $\hat\alpha,\tilde\alpha,\hat\beta$ and $\tilde\beta$ we thus have 
\[
\lambda_+=\tilde\alpha^{-1}(1+\mathcal{O}(T^{-4/3}))\qquad\qquad \lambda_-=\lambda_+^{-1}.
\]
Finally, one can check that the eigenvector associated to the contracting eigenvalue is of the form $v_-=(1,\ \tilde\beta(1+\mathcal{O}(T^{-2/3}))$ and the proof follows since, at a point $(q,p)=\Psi_s(\xi,\eta)\in H$, the vector $(1,\tilde\beta(q,p))$ is tangent to the foliation $\mathcal F_h$.
\end{proof}

We now control the angles between $\Psi_{glob} (\mathcal F_v)$ and $\mathcal F_h$.
\begin{lem}\label{lem:lineoftangencies}
There exists a line $\ell_\kappa\subset V$ such that \begin{itemize}
    \item $\angle (\ell_\kappa, \mathcal F_v)\geq \pi/4$
    \item At a point $(q,p)\in \ell_\kappa$ the angle between the image under $\Psi_{glob}$ of the tangent vector to $\mathcal F_v$ at $(q,p)$ and the tangent vector to $\mathcal F_h$ at $\Psi_{glob}(q,p)$ is zero.
    \item The angle between the image under $\Psi_{glob}$ of a vector tangent to $\mathcal F_v$ and the vector tangent to $\mathcal F_h$ at the image point is (locally) strictly increasing as we move away from $\ell_\kappa$.
\end{itemize}
    
\end{lem}

To prove Lemma~\ref{lem:lineoftangencies} we need to use some of the results established in the proof of Proposition~\ref{prop:mainresultrenormalization}. Thus, its proof is deferred to Section~\ref{sec:returnmap}. In the next result, we describe the local geometry of the invariant manifolds of the point $z_{\kappa,T}$.

\begin{figure}
\centering
\includegraphics[scale=1]{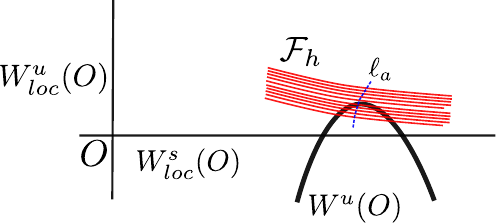}
\caption{The line of tangencies is transverse to the leaves of the horizontal foliation.}
\label{fig:linetangencies}
\end{figure}

\begin{lem}\label{lem:locgeomunstabmanifoldbasicset}
   Denote by $\gamma^{u,s}_{loc}$ the unique connected component of $W^{u,s}(z_{\kappa,T})$ which contains $z_{\kappa,T}$ and is contained in the ball of radius $T^{-10/3}$ around $z_{\kappa,T}$. Then,
   \begin{itemize}
   \item $\gamma^u_{loc}$ is arbitrarily close (for large $T$) to a leave of $\Psi_{glob}(\mathcal F_{v})$.
       \item $\angle (\gamma_{loc}^u,\mathcal F_v)\gtrsim T^{-5/3}$.
   \end{itemize}
   
\end{lem}
The proof of Lemma~\ref{lem:locgeomunstabmanifoldbasicset} is carried out in Section~\ref{sec:returnmap}.  Finally, we show how to use Lemmas~\ref{lem:c1accumulation},~\ref{lem:conefields},~\ref{lem:lineoftangencies}~and~\ref{lem:locgeomunstabmanifoldbasicset} to study images of $\gamma^{u}_{loc}$ under $\Psi_{glob}\circ\Psi^T_{loc}$. By Lemmas~\ref{lem:locgeomunstabmanifoldbasicset}~and~\ref{lem:conefields} the image of $\gamma^{u}_{loc}$ under the local map $\Psi_{loc}^T$ gets expanded by a factor $T^{5/3}$ and forms an angle $\mathcal O(T^{-5/3})$ with $\mathcal F_v$. Moreover $\Psi_{loc}^T(\gamma^{u}_{loc})$ moves away from the line of tangencies $\ell_\kappa$. Thus, $\Psi_{glob}\circ\Psi_{loc}^T(\gamma^u_{loc})$ forms an angle larger than $\sim T^{-5/3}$ with $\mathcal F_h$. Therefore, we can repeat this argument again until the image of  $\gamma_{loc}^u$ under a suitable iterate of $ \Psi_{glob}\circ \Psi^{T}_{loc}$ is arbitrarily close to a (sufficiently large) compact piece of leaf of $\Psi_{glob}(\mathcal F_v)$.
The proof of Theorem~\ref{thm:main} is complete.

\subsection{Proof of Theorem~\ref{thm:main2}}\label{sec:pfthm2}
As already mentioned, the proof of Theorem~\ref{thm:main2} follows directly from the ideas developed by Gorodetski in the proof of Theorem 4 in \cite{Gorodetskistdmap}. As almost all results in homoclinic bifurcation theory, these ideas make use of the Lambda lemma. We have already discussed (see Section~\ref{sec:faillambdalemma}) that the classical version of the Lambda lemma fails in our degenerate setting. Namely, if we take a disk transverse to, let's say, the local stable manifold, and we iterate it $T\in\mathbb{N}$ times, we only control the geometry of the part which is ``sufficiently'' far away from $O$. This is of course not an issue as we do not need to accumulate pieces of the unstable manifold which are arbitrarily close to $O$ but only a compact part which contains at least one fundamental domain.

\subsubsection{Dimension of $H(O,P_\mu)$}\label{sec:HD}
It is straightforward to check that   the very same argument in Section 6 of \cite{Gorodetskistdmap} \footnote{Replacing the renormalization result of Mora-Romero \cite{MoraRomero} with Proposition~\ref{prop:mainresultrenormalization}.} allows us to conclude that $\mathrm{dim}_H H(O,P_\mu)=2$ once we verify that there exists a dense subset of parameters in $\mathcal N$ for which $O$ has a homoclinic orbit of quadratic tangency which unfolds generically. This  is a consequence of the fact that for all $\mu\in \mathcal N$ there exists a basic set $\Lambda_\mu$ homoclinically related to $O$ and which exhibits persistent tangencies which unfold generically (see Theorem~\ref{thm:main} and Figure~\ref{fig:Homrelated}). Indeed, this implies that $W^{s,u}(\Lambda_\mu)\subset\overline {W^{s,u}(O)}$ and the claim follows.

\subsubsection{Accumulation by generic elliptic periodic orbits}\label{sec:ellipticpoints}
We now complete the proof of the second part of Theorem~\ref{thm:main2}, which (as the first part) follows entirely from the arguments developed in Section 6 of \cite{Gorodetskistdmap}. The only reason why we reproduce these arguments in this section is to indicate where the Lambda lemma is used so one can check that the classical version can be replaced by a suitable reformulation.

 Let $\mu\in \mathcal N$ and let $Q_\mu$ be a transverse homoclinic point to $O$ for the map $P_\mu$. Let $I_{Q_\mu}\subset\mathcal N$ be the maximal subinterval for which $Q_\mu$ can be continued. All transverse homoclinic points for all values of $\mu\in\mathcal N$ generate a countable number of subintervals $\{I_s\}_{s\in\mathbb N}$. We now claim that, for each $I_s$ there exists a residual subset $R_s\subset I_s$ such that for $\mu\in R_s$ the transverse homoclinic point $Q_\mu$ is accumulated by elliptic periodic points. Then, we define 
\[
\mathcal R=\bigcap_{s\in\mathbb N} \tilde R_s\qquad\qquad \tilde R_{s}= R_s\cup (\mathcal N\setminus I_s).
\]
It remains to verify the claim.  To that end we first notice that there exists a dense subset $\tilde I_s\subset I_s$ such that for any $\mu\in \tilde I_s$, $O$ has a homoclinic orbit $z_\mu$ of quadratic tangency  which unfolds generically. 

\begin{figure}
\centering
\includegraphics[scale=1]{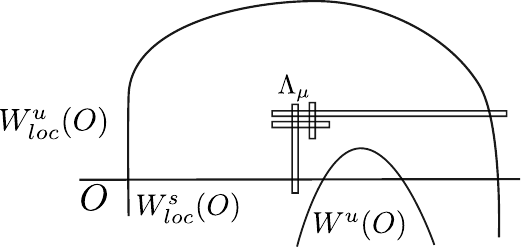}
\caption{The basic set $\Lambda_\mu$ is homoclinically related to $O$.}
\label{fig:Homrelated}
\end{figure}

Let $\{\tilde\mu_n\}_{n\in\mathbb N}\subset\widetilde I_s$ be a dense subsequence and denote by $z_{\tilde\mu_n}$ the corresponding sequence of orbits of quadratic homoclinic tangency.  Then, Proposition~\ref{prop:mainresultrenormalization} and the proof of Theorem~A in  \cite{MoraRomero}  imply that, for any $m\in\mathbb N$, there exists an open and dense subset $R_{s,m}\subset \tilde I_s$ such that, for $\mu \in R_{s,m}$, there exists an elliptic periodic point of $P_\mu$  contained in a $\mathcal O(1/m)$ neighborhood of some $z_{\tilde\mu_n}$. It is at this step of the proof that \cite{Gorodetskistdmap} uses the Lambda Lemma, in the setting of that paper, to show that for any $\mu\in R_{s,m}$ also a $\mathcal O(1/m)$ neighborhood of the corresponding transverse homoclinic point contains an elliptic periodic point. As, without loss of generality we can suppose that $Q_\mu$ is uniformly away from the degenerate saddle, it also holds in our setting that for any $\mu\in R_{s,m}$ a $\mathcal O(1/m)$ neighborhood of $Q_\mu$ contains an elliptic periodic point (make use, for instance, of the definition of the Shilnikov maps and their asymptotic properties in Proposition~\ref{prop:Shilnikov}). The claim holds defining $R_s=\bigcap_{m\in\mathbb N}R_{s,m}$. This shows the truth of Theorem~\ref{thm:main2}.

\section{Normal form coordinates}\label{sec:normalform}
This section is dedicated to the proof of Proposition~\ref{prop:normalform}, which provides a suitable normal form to show the existence of Shilnikov coordinates in Section~\ref{sec:Shilnikov}. First we provide a proof of Lemma~\ref{lem:symplecticred}. The proof is standard and we dispose of it quickly.

\begin{proof}[Proof of Lemma~\ref{lem:symplecticred}]
    The equation $H_{\mathrm{pol},\mu}(r,y,\phi,G)-G=J$ defines an analytic function $G = \widetilde{G} (r,y,\phi;J)$ satisfying 
	\begin{align*}
		\label{def:G_reduction_Poincare_Cartan}
		\widetilde{G} (r,y,\phi;J) &= r^2 \left(-1 + \sqrt{1 - \frac{2}{r^2} \left(\frac{y^2}{2} - \frac{1}{r} - U_{\mu}(r, \phi) - J\right)} \, \right) \\
        &= \sum_{n=1}^{\infty} 2^n {(-1)}^n {\frac{1}{2}\choose n} r^{-2n+2} {\left(\frac{y^2}{2} - \frac{1}{r} - U_{\mu}(r, \phi) - J\right)}^n \\
		&= -\left(\frac{y^2}{2} - \frac{1}{r} - U_{\mu}(r, \phi) - J\right) + \frac{1}{r^2} \left(\frac{J^2}{2} + \mathcal O(r^{-1},y^2)\right).
	\end{align*}
    By Poincar\'{e}-Cartan reduction, the equations of motion defined by the Hamiltonian $H_{\mathrm{pol},\mu}$ on the level $\{\mathcal J=J\}$ are equivalent to the equations of motion of
    \[
    H_{\mathrm{red},\mu}(r,y,\phi;J)=-\widetilde{G} (r,y,\phi;J)+J = \frac{y^2}{2} - \frac{1}{r}  + \frac{J^2}{2 r^2} + \frac{1}{r^2} {\mathcal{O}}\left(r^{-1},y^2\right).
    \]
\end{proof}

\begin{rem}
\label{paritat_en_t}
It is well known that the potential $U$ in~\eqref{eq:Hamiltonianpolar} is even in $\phi = \alpha - t$. Hence, so is $H_{\mathrm{red},\mu}(r,y,\phi;J)$. The same happens to the Hamiltonian of the Sitnikov configuration~\eqref{Sitnikov} and the restricted planar 4 body problem~\eqref{eq:H4bp}, when two of the primaries have equal masses.
\end{rem}

From now on, we write $t$ instead of $\phi$ as the time variable. Together,  Lemmas~\ref{lem:symplecticred}~and~\ref{lem:McGehee}  show that  the Hamiltonian of the restricted 3-body problem is conjugated to the flow of a Hamiltonian of the form
\begin{equation}\label{def:HamiltoniaH}
	H(q,p,t,I) = \NN(q,p,I) + H_1(q,p,t)
\end{equation}
with respect to the 2-form
\begin{equation}
	\label{def:omega}
	\omega = \frac{1}{(q+p)^3} dq\wedge dp + dt \wedge dI,
\end{equation}
where
\begin{equation}
	\label{def:Hamiltonia_N_i_Hu}
	\NN(q,p,I)  = -qp + I
\end{equation}
and
\begin{equation}
	\label{def:Hamiltonia_Hu}
	H_1(q,p,t)  = (q+p)^3 \widetilde H_1(q,p,t),
\end{equation}
is even in $(q,p)$ and even in $t$, that is, the Taylor expansion of $H_1$ in  $(q,p)$ has
only even powers and their coefficients are even periodic functions of $t$, and $\widetilde H_1 = \OO_1(q,p)$.

\begin{thm}
	\label{thm:forma_normal}
	Let $H(q,p,t,I)$ be a Hamiltonian of the form \eqref{def:HamiltoniaH} with $\omega$ as in \eqref{def:omega}, $\mathcal N$ as in \eqref{def:Hamiltonia_N_i_Hu} and $H_1$ as in \eqref{def:Hamiltonia_Hu}. Let $N \in \N$ be fixed and let $\rho>0$ be small enough. Let $V_\rho$ be the sectorial domain in \eqref{eq:sectorialdom}. There exists a close to the identity change of coordinates 
 \[
 \Phi:V_{\rho/2}\times V_{\rho/2}\times\mathbb T_\rho\to V_{\rho}\times V_{\rho}\times\mathbb T_\rho
 \]
 of class $\mathcal C^{N}$ with respect to $(q,p)$ at $(q,p) = (0,0)$ and real-analytic away from $(q,p)=(0,0)$, preserving the 2-form $\omega$, such that $H\circ\Phi=\NN+ \wt H_1$, with
	\[
	 \wt H_1(q,p,t) =  (qp)^4 g_N(q,p,t),
	\]
	where  $ g_N(q,p,t) = \OO_N(q,p)$.
\end{thm}

The rest of the section is devoted to the proof of Theorem~\ref{thm:forma_normal}, which is performed in two steps. 
In what follows, we will use the Poisson bracket associated to the 2-form $\omega$ in~\eqref{def:omega}, that is, for functions $F(q,p,t,I)$ and $G(q,p,t,I)$,
\begin{equation}
	\label{def:parentesi_de_Poisson}
	\{F,G\}_{\mid(q,p,t,I)} = (q+p)^3 \left(\frac{\partial F}{\partial q}\frac{\partial G}{\partial p}-\frac{\partial F}{\partial p}\frac{\partial G}{\partial q}\right)_{\mid(q,p,t,I)}+ \left(
	\frac{\partial F}{\partial t}\frac{\partial G}{\partial I}-\frac{\partial F}{\partial I}\frac{\partial G}{\partial t}\right) _{\mid(q,p,t,I)}.
\end{equation}

\subsection{First step: preliminary transformations and straightening of the invariant manifolds}

\begin{prop}
 	\label{prop:redressament_de_les_varietats}
 	Let $H$ be the Hamiltonian in Theorem \ref{thm:forma_normal}. Let $N,M\in \N$ be fixed, with $N - 2M$ large enough. Then, for $\rho>0$ small enough there  exists a close to identity change of coordinates $\Phi:V_{\rho/2}\times V_{\rho/2}\times\mathbb T_\rho\to V_{\rho}\times V_{\rho}\times\mathbb T_\rho$ satisfying $\pi_t \Phi = \Id$, preserving the 2-form $\omega$, such that $H\circ\Phi=\NN + qp\widehat H_1$ with
  \[
  \widehat H_1(q,p,t) = \sum_{j=0}^{M-1} (qp)^j\left(h_j(q,t) + \tilde h_j (p,t)\right) + (qp)^M   \wt H_{2}(q,p,t),
  \]
  and, for some $\rho>0$, $h_j=\mathcal O_{N-2j}(q), \tilde h_j= \mathcal O_{N-2j}(p)$  for $j=1,\dots,M-1$ and 
  $\wt H_2:V_{\rho} \times V_{\rho}\times  \T_{\sigma}  \to \C$, analytic, $\mathcal C^\infty$ at $\{ 0 \} \times V_{\rho}\times  \T_{\rho}  \cup V_{\rho} \times \{0\}\times  \T_{\rho}$ and $\wt H_2(q,p,t) = \OO_{N-2M}(q,p)$. 
 \end{prop}
 
The proof of this result is technical and is deferred to Appendix \ref{sec:appendixtechnical}. We now show how to complete the proof of Theorem \ref{thm:forma_normal}. First we introduce a suitable functional setting.

\subsection{Some Banach spaces and their properties}

We start with some definitions. Let $\rho>0$ and $\sigma >0$ and let $V_\rho$ as in \eqref{eq:sectorialdom}. We define the spaces of functions
\begin{equation}
\label{def:espaisXiY}
\begin{aligned}
\X_{r,\rho} & = \{f: V_{\rho}\times \T_\sigma \to \C\mid \text{analytic, \;$\|f\|_{r,\rho} < \infty$}\}, \\
\Y_{N,\rho} & = \{F: V_{\rho} \times V_{\rho} \times \T_\sigma \to \C\mid \text{analytic,  $\|F\|_{N,\rho} < \infty$}\},
\end{aligned}
\end{equation}
where $\T_\sigma = \{t \in \C/2\pi\Z\mid |\Im t | < \sigma\}$ and, for $f\in \X_{r,\rho}$ and $F\in \Y_{N,\rho}$,
\[
\begin{aligned}
  \|f\|_{r,\rho} & = \sup_{(z,t)\in V_{\rho}\times \T_\sigma} |z^{-r} f(z,t)|,\\
  \|F\|_{N,\rho} & = \sup_{(q,p,t)\in V_{\rho}\times V_{\rho}\times\T_\sigma} |(|q|+|p|)^{-N} F(q,p,t)|.
\end{aligned}
\]
It is clear that $\X_{r,\rho}$ and $\Y_{N,\rho}$ are Banach spaces and the norms satisfy
\[
\|f g \|_{r+s,\rho} 
\le \|f\|_{r,\rho} \|g \|_{s,\rho},  
\qquad
\|F G\|_{N+M,\rho} 
\le \|F \|_{N,\rho} \|G\|_{M,\rho}.
\]
Furthermore, if $s>0$  and $M>0$,
\[
\|f\|_{r-s,\rho}  \le \rho^{s} \|f\|_{r,\rho}, \qquad
 \|F\|_{N-M,\rho}  \le (2\rho)^M\|F\|_{N,\rho}.
\]
Given $k\ge 1$, $\ell \ge k$, $N-2\ell \ge 0$, we also define the space
\begin{multline}
\label{def:espaiZ}
\ZZ_{k,\ell,N,\rho}   = \{  H: V_{\rho} \times V_{\rho} \times \T_\sigma \to \C\mid 
H(q,p,t) = \sum_{j=k}^{\ell-1} q^jp^j (h_j(q,t)+\tilde h_j(p,t)) + q^\ell p^\ell \wt H(q,p,t), 
\\ 
\text{with  $h_j,\tilde h_j \in \X_{N-2j,\rho}$, $j=k,\dots,\ell-1$, $\wt H \in \Y_{N-2\ell,\rho}$}\}, 
\end{multline}
which, with the norm
\[
\llbracket H\rrbracket_{k,\ell,N,\rho}  = \sum_{j=k}^{\ell-1} \left(\|h_j\|_{N-2j,\rho}+\|\tilde h_j\|_{N-2j,\rho}\right)  + \|\wt H\|_{N-2\ell,\rho}, 
\]
is a Banach space. 
If $\ell = k$, we  will understand the sum in the definition of $\ZZ_{\ell,\ell,N,\rho}$ and its norm as empty. The following lemma is a straightforward computation.
\begin{lem}
\label{prop:normaNnormaNmes2j}
There exists $K\ge 0$ such that, if $H \in \ZZ_{k+k',\ell+\ell',N+N',\rho}$, with $k,k',\ell, \ell', N, N' \ge 0$, $\ell \ge k$, , $\ell+\ell' \ge k+k'$, $N+N'-2(\ell+\ell')\ge 0$, then $H \in \ZZ_{k,\ell,N,\rho}$ and
\begin{equation}
\label{fita:normaNnormaNmes2j}
\llbracket H\rrbracket_{k,\ell,N,\rho} \le K^{N'}\rho^{N'} \llbracket H\rrbracket_{k+k',\ell+\ell',N+N',\rho}.
\end{equation} 
\end{lem}

We remark that, if $H \in \ZZ_{k,\ell,N,\rho}$, and $N-2\ell \ge 0$, the corresponding functions $h_j,\tilde h_j \in \X_{N-2j,\rho}$, $j=k,\dots,\ell-1$, and $\wt H \in \Y_{N-2\ell,\rho}$ are uniquely defined. Indeed, $h_k(q,t) = (H(q,p,t)/(q^k p^k))_{\mid p =0}$, and the rest of the functions are obtained analogously. Hence, we have that
\[
\ZZ_{k,\ell, N,\rho} = \ZZ_{N,\rho}^k \oplus \ZZ_{k+1,\ell,N,\rho},
\]
where
\begin{equation}
\label{def:ZZ1_i_wtZZ}
\ZZ_{N,\rho}^k  = \{ H: V_{\rho} \times V_{\rho} \times \T_\sigma \to \C\mid 
 H(q,p,t) = q^k p^k(h_k(q,t)+\tilde h_k(p,t)), \text{ with  $h_k,\tilde h_k \in \X_{N-2k,\rho}$}\}  .
\end{equation}
The space $\ZZ_{N,\rho}^k$ is a Banach space with the induced norm. We will denote by $\pi_k: \ZZ_{k,\ell,N,\rho}  \to \ZZ_{N,\rho}^k$, the projection onto $\ZZ_{N,\rho}^k$, and $\tilde\pi_{k} = \Id -\pi_k: \ZZ_{k,\ell,N,\rho}  \to \ZZ_{k+1,\ell,N,\rho}$. Clearly,  $\|\pi_k\| = \|\tilde\pi_{k}\| =1$. Given $k'<k$  and $H\in\mathcal Z_{k,\ell,N,\rho}$ we will abuse notation and write $\pi_{k'} H=0$. 

\subsection{Poisson bracket properties and Lie series}
\label{sec:parentesi_de_Poisson_i_series_de_Lie}
The following elementary lemma will play a crucial role in what follows.

\begin{lem}
    \label{lem:producte_de_H_per_f_i_tilde_f}
Assume $0<\rho <1$, $k\ge 1$, $\ell \ge k+1$, $N\ge 2\ell$, and $M\ge \ell-k$. Then, there exists $C>0$ such that, if  $f,\tilde f \in \X_{M,\rho}$ and $H \in \ZZ_{k,\ell,N,\rho}$,  $\wh H(q,p,t) = H(q,p,t)(f(q,t)+\tilde f(p,t)) \in \ZZ_{k,\ell,N+M,\rho}$, $\llbracket \wh H \rrbracket_{k,\ell,N+M,\rho} \le C \llbracket H \rrbracket_{k,\ell,N,\rho}(\|f\|_{M,\rho}+\|\wt f\|_{M,\rho})$  and $\llbracket \pi_1 \wh H \rrbracket_{k,\ell,N+M,\rho} \le  \llbracket \pi_1 H \rrbracket_{k,\ell,N,\rho}(\|f\|_{M,\rho}+\|\wt f\|_{M,\rho})$.
\end{lem}

\begin{proof}  Let $H(q,p,t) = \sum_{j=k}^{\ell-1} q^j p^j (h_j(q,t)+\tilde h_j(p,t)) +q^\ell p^\ell \wt H(q,p,t)$. Then, since, for $j=k,\dots,\ell-1$,
\[
q^jp^j (h_j+\tilde h_j) (f+\tilde f)  = q^jp^j (h_jf+\tilde h_j \tilde f) + q^\ell p^\ell \frac{1}{q^{\ell-j}p^{\ell-j}}(h_j \tilde f + \tilde h_j f), 
\]
and 
\[
\left\|\frac{1}{q^{\ell-j}p^{\ell-j}}(h_j \tilde f + \tilde h_j f) \right\|_{N+M-2\ell,\rho} \le 
\llbracket H \rrbracket_{k,\ell,N,\rho}(\|f\|_{M,\rho}+\|\wt f\|_{M,\rho}),
\]
both claims follow immediately. 

\end{proof}

\begin{lem}
\label{lem:multiplicacio_per_qp} Assume $0<\rho <1$, $n\ge 0$.
For $H \in \ZZ_{k,\ell,N,\rho}$, let $\mathsf{j}_n(H)(q,p,t) = q^n p^n H(q,p,t)$. Then
$\mathsf{j}_n:  \ZZ_{k,\ell,N,\rho} \to \ZZ_{k+n,\ell+n,N+2n,\rho}$ is linear and $\|\mathsf{j}_n\| \le 1$.
Furthermore, $\mathsf{j}_n(H) \in \ZZ_{\min\{k+n,\ell\},\ell,N,\rho}$, with 
$\llbracket \mathsf{j}_n(H)\rrbracket_{\min\{k+n,\ell\},\ell,N,\rho} \le \rho^{2n} \llbracket H\rrbracket_{k,\ell,N,\rho} $.

\end{lem}
\begin{proof}
    It is a straightforward computation. 
\end{proof}

\begin{lem}
\label{lem:derivades_a_ZNrho} Assume $0<\rho <1$.
There exists $C>0$ such that, for any $0 < \rho' < \rho$, if $H \in \ZZ_{k,\ell,N,\rho}$, then
\begin{enumerate}
\item $q\partial_q H \in \ZZ_{k,\ell,N,\rho'}$ and 
\[\llbracket q\partial_q H\rrbracket_{k,\ell,N,\rho'} \le \frac{C}{\rho-\rho'} \llbracket H\rrbracket_{k,\ell,N,\rho},  \qquad \llbracket \pi_1 (q\partial_qH)\rrbracket_{k,\ell,N,\rho'} \le \frac{C}{\rho-\rho'} \llbracket\pi_1 H\rrbracket_{k,\ell,N,\rho},
\]
\item $p\partial_p H \in \ZZ_{k,\ell,N,\rho'}$ and
\[
\llbracket p\partial_p H\rrbracket_{k,\ell,N,\rho'} \le \frac{C}{\rho-\rho'} \llbracket H\rrbracket_{k,\ell,N,\rho}, \qquad \llbracket \pi_1 (p\partial_pH)\rrbracket_{k,\ell,N,\rho'} \le \frac{C}{\rho-\rho'} \llbracket\pi_1 H\rrbracket_{k,\ell,N,\rho}.
\]
\end{enumerate}
\end{lem}

\begin{proof} 
It follows from Cauchy estimates and  a straightforward computation.
\end{proof}

\begin{lem}
\label{lem:derivades_de_pf+qtildef} Assume $0<\rho <1$. Let $M,n\in \N$, $M \ge n \ge 1$. Let $f,\tilde f \in \X_{M,\rho}$, $F(q,p,t) = p^n f(q,t)+q^n\tilde f (p,t)$. Then there exists $C>0$ such that, for any $0 < \rho' < \rho$,
\begin{enumerate}
\item $\partial_p F(q,p,t) = q^n p^{n-1} (g(q,t)+\tilde g(p,t))$, with $\|g\|_{M-n,\rho'}\le \frac{C}{\rho-\rho'} \|f\|_{M,\rho}$ $ \|\tilde g\|_{M-n,\rho'}\le \frac{C}{\rho-\rho'} \|\tilde f\|_{M,\rho}$,
\item $\partial_q F(q,p,t) = q^{n-1}p^n (h(q,t)+\tilde h(p,t))$, with $\|h\|_{M-n,\rho'}\le \frac{C}{\rho-\rho'} \|f\|_{M,\rho}$ $ \|\tilde h\|_{M-n,\rho'}\le \frac{C}{\rho-\rho'} \|\tilde f\|_{M,\rho}$.
\end{enumerate}
\end{lem}

\begin{proof} It follows from Cauchy estimates and  a straightforward computation.
\end{proof}

Combining these lemmas, we have the following.
We recall that, given $H(q,p,t)$, $G(q,p,t)$, its Poisson bracket is
\[
\{H,G\}(q,p,t) = (q+p)^3 (\partial_q H \partial_p G - \partial_p H \partial_q G).
\]

\begin{lem}[Poisson bracket lemma]
\label{lem:parentesi_de_Poisson}
Let $k\ge 1$, $\ell \ge k$, $N-2\ell \ge 0$, $n\ge 1$, $M\ge n$. There exists $\tilde C>0$ such that, for any $1 \ge \rho_0\ge\rho >\rho' >0$, the following holds. Let $H \in \ZZ_{k,\ell,N,\rho}$,   $f,\tilde f \in \X_{M,\rho_0}$, $F(q,p,t) = p^n f(q,t)+q^n\tilde f (p,t)$. Then, $\{H,F\} \in \ZZ_{k+n-1,\ell+n-1,N+M+n+1,\rho'}$  and 
\[
\llbracket \{H,F\} \rrbracket_{k+n-1,\ell+n-1,N+M+n+1,\rho'}  \le \frac{\tilde C }{(\rho_0-\rho')(\rho-\rho')} \llbracket H\rrbracket_{k,\ell,N,\rho}(\|f\|_{M,\rho_0}+\|\tilde f\|_{M,\rho_0}). 
\]
In particular, if $n\ge 2$, $\pi_1 \{H,F\} = 0$. If $n=1$, 
\[
\llbracket \pi_1 \{H,F\} \rrbracket_{k+n-1,\ell+n-1,N+M+n+1,\rho'} 
 \le  \frac{\tilde C }{(\rho_0-\rho')(\rho-\rho')} \llbracket \pi_1 H\rrbracket_{k,\ell,N,\rho}(\|f\|_{M,\rho_0}+\|\tilde f\|_{M,\rho_0}).
\]
\end{lem}

\begin{proof}
By Lemma~\ref{lem:derivades_de_pf+qtildef},
\[
\begin{aligned}
    q^{-1}\partial_p F(q,p,t) & = q^{n-1} p^{n-1} (g(q,t)+\tilde g (p,t)),\\
    p^{-1}\partial_q F(q,p,t) & = q^{n-1} p^{n-1}(h(q,t)+\tilde h (p,t))
\end{aligned}
\]
with $\|g\|_{M-n,\rho'}, \|h\|_{M-n,\rho'} \le C (\rho_0-\rho')^{-1}\|f\|_{M,\rho_0}$ and $\|\tilde g\|_{M-n,\rho'}, \|\tilde h\|_{M-n,\rho'} \le C (\rho_0-\rho')^{-1}\|\tilde f\|_{M,\rho_0}$.
Hence, by Lemmas~\ref{lem:derivades_a_ZNrho},~\ref{lem:multiplicacio_per_qp} and~\ref{lem:producte_de_H_per_f_i_tilde_f}, for some constant $K$ (depending only on the constants in those lemmas),
\[
\begin{aligned}
\llbracket \partial_q H \partial_p F\rrbracket_{k+n-1,\ell+n-1,N+M+n-2,\rho'} & = \left\llbracket q\partial_q H (q p)^{n-1} (g+\tilde g )\right\rrbracket_{k+n-1,\ell+n-1,N+M+n-2,\rho'} \\
& \le C \llbracket q\partial_q H \rrbracket_{k,\ell,N,\rho'} \left( \| g\|_{M-n,\rho_1} +  \| \tilde g\|_{M-n,\rho_1}\right) \\
& \le K \frac{1}{\rho_0-\rho'} \frac{1}{\rho-\rho'}\llbracket H\rrbracket_{k,\ell,N,\rho}(\|f\|_{M,\rho_0}+\|\tilde f\|_{M,\rho_0}).
\end{aligned}
\]
With the same argument, we have that
\[
\llbracket \partial_p H \partial_q F\rrbracket_{k+n-1,\ell+n-1,N+M+n-2,\rho'}\le K \frac{1}{\rho_0-\rho'} \frac{1}{\rho-\rho'}\llbracket H\rrbracket_{k,\ell,N,\rho}(\|f\|_{M,\rho_0}+\|\tilde f\|_{M,\rho_0}).
\]
Now, taking into account that $(q+p)^3 = q^3 + p^3 + qp(3q+3p)$, on the one hand we have that, by Lemma~\ref{lem:producte_de_H_per_f_i_tilde_f}
\begin{multline*}
\llbracket (q^3+p^3) (\partial_q H \partial_p F - \partial_p H \partial_q F)\rrbracket_{k+n-1,\ell+n-1,N+M+n+1,\rho'} \\
\begin{aligned}
 & \le C
\llbracket \partial_q H \partial_p F - \partial_p H \partial_q F\rrbracket_{k+n-1,\ell+n-1,N+M+n-2,\rho'}(\|q^3\|_{3,\rho'}+\|p^3\|_{3,\rho'}) \\
& \le K \frac{1}{\rho_0-\rho'} \frac{1}{\rho-\rho'}\llbracket H\rrbracket_{k,\ell,N,\rho}(\|f\|_{M,\rho_0}+\|\tilde f\|_{M,\rho_0}).
\end{aligned}
\end{multline*}
In particular, if $n=1$,
\begin{multline*}
\llbracket \pi_1 (q^3+p^3) (\partial_q H \partial_p F - \partial_p H \partial_q F)\rrbracket_{k,\ell,N+M+2,\rho'} \\
\begin{aligned}
 & \le C
\llbracket \pi_1(\partial_q H \partial_p F - \partial_p H \partial_q F)\rrbracket_{k,\ell,N+M-1,\rho'}(\|q^3\|_{3,\rho'}+\|p^3\|_{3,\rho'}) \\
& \le K \frac{1}{\rho_0-\rho'} \frac{1}{\rho-\rho'}\llbracket \pi_1 H\rrbracket_{k,\ell,N,\rho}(\|f\|_{M,\rho_0}+\|\tilde f\|_{M,\rho_0}).
\end{aligned}
\end{multline*}
Also, on the other hand, by Lemma~\ref{lem:multiplicacio_per_qp}, 
\begin{multline*}
\llbracket qp (3q+3p) (\partial_q H \partial_p F - \partial_p H \partial_q F)\rrbracket_{k+n,\ell+n,N+M+n+1,\rho'} \\
\begin{aligned}
 & \le 
\llbracket (3q+3p) (\partial_q H \partial_p F - \partial_p H \partial_q F)\rrbracket_{k+n-1,\ell+n-1,N+M+n-2,\rho'} \\
& \le 
\|3q+3p\|_{1,\rho'} \llbracket \partial_q H \partial_p F - \partial_p H \partial_q F\rrbracket_{k+n-1,\ell+n-1,N+M+n-2,\rho'} \\
& \le 
K \frac{1}{\rho_0-\rho'} \frac{1}{\rho-\rho'}\llbracket H\rrbracket_{k,\ell,N,\rho}(\|f\|_{M,\rho_0}+\|\tilde f\|_{M,\rho_0}).
\end{aligned}
\end{multline*}
Hence, the claim is proven.
\end{proof}

\begin{prop}
    \label{prop:serie_de_Lie}
Let $k\ge 1$, $\ell \ge k$, $N-2\ell \ge 0$, $n\ge 1$, $M\ge n$. There exist $K, \wt K>0$ such that, for any $1 \ge \rho >\rho' >0$, the following holds.  Let  $f,\tilde f \in \X_{M,\rho}$, $F(q,p,t) = p^n f(q,t)+q^n\tilde f (p,t)$ and let $\Phi_F^s$ be the flow of the Hamiltonian $F$ (with respect to the 2-form~\eqref{def:omega}). Then,  if 
\[
\frac{1}{(\rho-\rho')^2} (\|f\|_{M,\rho}+\|\tilde f\|_{M,\rho}) < K,
\]
for $H \in \ZZ_{k,\ell,N,\rho}$,  $H \circ \Phi_F^1 \in \ZZ_{N,\rho'}$, 
\[
\begin{aligned}
\llbracket H \circ \Phi_F^1\rrbracket_{k,\ell,N,\rho'} & \le \left( 1+  \wt K \frac{\rho^{M+n+1}}{(\rho-\rho')^2} (\|f\|_{M,\rho}+\|\tilde f\|_{M,\rho})\right) \llbracket H \rrbracket_{k,\ell,N,\rho},\\
\llbracket \pi_1 H \circ \Phi_F^1\rrbracket_{k,\ell,N,\rho'} & \le \left( 1+ \wt K \frac{\rho^{M+n+1} }{(\rho-\rho')^2} (\|f\|_{M,\rho}+\|\tilde f\|_{k,\ell,M,\rho})\right) \llbracket\pi_1 H \rrbracket_{k,\ell,N,\rho},
\end{aligned}
\]
$H \circ \Phi_F^1-H \in  \ZZ_{k+n-1,\ell+n-1,N+M+n+1,\rho'}$,
\[
\begin{aligned}
\llbracket H \circ \Phi_F^1-H\rrbracket_{k+n-1,\ell+n-1,N+M+n+1,\rho'} & \le   \frac{\wt K}{(\rho-\rho')^2}  \llbracket H \rrbracket_{k,\ell,N,\rho}  (\|f\|_{M,\rho}+\|\tilde f\|_{M,\rho}), \\
\llbracket \pi_1 H \circ \Phi_F^1-\pi_1 H\rrbracket_{k+n-1,\ell+n-1,N+M+n+1,\rho'} & \le   \frac{\wt K}{(\rho-\rho')^2}  \llbracket \pi_1 H \rrbracket_{k,\ell,N,\rho}  (\|f\|_{M,\rho}+\|\tilde f\|_{M,\rho}),
\end{aligned}
\]
and 
$H \circ \Phi_F^1-H -\{H,F\}\in  \ZZ_{k+2(n-1),\ell+2(n-1),N+2(M+n+1),\rho'}$,
\[
\llbracket H \circ \Phi_F^1-H-\{H,F\}\rrbracket_{k+2(n-1),\ell+2(n-1),N+2(M+n+1),\rho'} \le  \frac{ \wt K}{(\rho-\rho')^4}  \llbracket H \rrbracket_{k,\ell,N,\rho}  (\|f\|_{M,\rho}+\|\tilde f\|_{M,\rho})^2.
\]
\end{prop}
\begin{proof}
We recall that
\[
H \circ \Phi_F^1 = \sum_{j\ge 0} \frac{1}{j!} \mathrm{ad}_F^j H,
\]
where $\mathrm{ad}_F^0 H = H$ and $\mathrm{ad}_F^j H = \{\mathrm{ad}_F^{j-1} H, F\}$, for $j\ge 1$.
Then,  applying Lemma~\ref{lem:parentesi_de_Poisson} $j$ times,
\begin{multline}
\label{ineq:fita_adjunt_fina}
\llbracket \mathrm{ad}_F^j H\rrbracket_{k+j(n-1),\ell+j(n-1),N+j(M+n+1),\rho'} \\
\begin{aligned}
 & = \llbracket \{\mathrm{ad}_F^{n-1} H, F\}\rrbracket_{k+j(n-1),\ell+j(n-1),N+j(M+n+1),\rho-(\rho-\rho')} \\
& \le \frac{C }{(\rho-\rho')^2/n} \llbracket \mathrm{ad}_F^{n-1} H\rrbracket_{k+(j-1)(n-1),\ell+(j-1)(n-1),N+(j-1)(M+n+1),\rho-(j-1)(\rho-\rho')/j} \\
& \times (\|f\|_{M,\rho}+\|\tilde f\|_{M,\rho}) \\
&\le \left(\frac{C }{(\rho-\rho')^2/j} (\|f\|_{M,\rho}+\|\tilde f\|_{M,\rho})\right)^j \llbracket H\rrbracket_{k,\ell,N,\rho}.
\end{aligned}
\end{multline}
In particular, 
\[
 \begin{aligned}
 \llbracket \mathrm{ad}_F^j H\rrbracket_{k,\ell,N,\rho'}   
& \le (2\rho)^{j(N+n+1)} \llbracket \mathrm{ad}_F^j H\rrbracket_{k+j(n+1),\ell+j(n+1),N+j(M+n+1),\rho'} \\
&\le \left(\frac{C (2\rho)^{N+n+1}}{(\rho-\rho')^2/j} (\|f\|_{M,\rho}+\|\tilde f\|_{M,\rho})\right)^j \llbracket H\rrbracket_{k,\ell,N,\rho}.
\end{aligned}
\]
hence, for a $\wt K$ large enough,
\[
\begin{aligned}
\llbracket H \circ \Phi_F^1\rrbracket_{k,\ell,N,\rho'} & \le \sum_{j\ge 0} \frac{1}{j!} \llbracket \mathrm{ad}_F^j H\rrbracket_{k,\ell,N,\rho'} \\
& \le \sum_{j\ge 0} \left(\frac{C (2\rho)^{M+n+1}e}{(\rho-\rho')^2} (\|f\|_{M,\rho}+\|\tilde f\|_{M,\rho})\right)^j \llbracket H\rrbracket_{k,\ell,N,\rho} \\
& \le \left( 1+ \wt K \frac{\rho^{M+n+1}}{(\rho-\rho')^2} (\|f\|_{M,\rho}+\|\tilde f\|_{M,\rho})\right) \llbracket H \rrbracket_{k,\ell,N,\rho}.
\end{aligned}
\]
Also, using directly~\eqref{ineq:fita_adjunt_fina},
\[
\begin{aligned}
\llbracket H \circ \Phi_F^1-H\rrbracket_{k+n-1,\ell+n-1,N+M+n+1,\rho'} &\le   \sum_{j\ge 1} \frac{1}{j!} \llbracket \mathrm{ad}_F^j H\rrbracket_{k+j(n-1),\ell+j(n-1),N+j(M+n+1),\rho'} \\
& \le 
 \sum_{j\ge 1}
\left(\frac{C }{(\rho-\rho')^2/j} (\|f\|_{M,\rho}+\|\tilde f\|_{M,\rho})\right)^j \llbracket H\rrbracket_{k,\ell,N,\rho},
\end{aligned}
\]
which implies the claim. The same argument produces the bound for $H \circ \Phi_F^1-H-\{H,F\}$.

The bounds for $\pi_1$ projections are obtained analogously.
\end{proof}

\subsection{The quadratic scheme}
\label{sec:esquema_quadratic}
Next lemma deals with the solutions of some linear partial differential equations we will need later. For $k\ge 1$, we introduce the first order linear partial differential operators
\begin{equation}
\label{def:operadors_diferencials_L_i_wt_L}
    \LL_k f = k q^{3-k} f - q^{4-k} \partial_q f + q^{-k}\partial_t f, \qquad \wt \LL_k \tilde f = -k p^{3-k} \tilde f + p^{4-k} \partial_p \tilde f + p^{-k}\partial_t \tilde f.
\end{equation}

\begin{lem}
\label{lem:Inversa_de_loperador_lineal}
Let $N\ge 4$, $k \ge 0$. There exists $C>0$ such that, for any $h \in \X_{N,\rho}$, there exists $f \in \X_{N-k-3,\rho}$, with $\|f\|_{N-k-3,\rho}\le C\|h\|_{N,\rho}$, satisfying $\LL_k f = h$. 
The same claim applies to the equation $\wt \LL_k \tilde f = \tilde h$.
\end{lem}
\begin{proof}
We observe that $\LL_k f = h$ is equivalent to
\begin{equation}
\label{eq:invertint_operador_lineal}
- q^4 \partial_q (f/q^k) + \partial_t (f/q^k) = h.
\end{equation}
With the change $q = 3^{-1/3} u^{-1/3}$ and introducing $\eta(u,t) = 3^{k/3} u^{k/3} f(3^{-1/3} u^{-1/3},t)$ and $\varphi(u,t) = h(3^{-1/3} u^{-1/3},t)$, equation~\eqref{eq:invertint_operador_lineal} becomes
\begin{equation}
\label{∂ef:equacio_coeficients_constants}
\partial_u \eta + \partial_t \eta = \varphi.
\end{equation}
Observe that the change $q \to u$ sends functions in $\X_{N,\rho}$ to functions in 
$\wt \X_{N/3,\rho}$, where
\[
\wt \X_{r,\rho} =  \{\eta: \wt V_{\rho}\times \T_\sigma \to \C\mid \text{analytic, \;$\|\eta\|_{r,\rho} < \infty$}\},
\]
with $\wt V_{\rho}$ being the image by the change of variables of the sector $V_{\rho}$, that is,
\[
\wt V_{\rho} = \{u \in \C \mid |\Im u| < \tan(3 \arctan(\kappa \rho)) \Re u, \; |u| > 3^{-1} \rho^{-3}\}
\]
and
\[
\|\eta\|_{r,\rho} =\sup_{(u,t)\in \wt V_{\rho}\times \T_\sigma} |u^r \eta(u,t)|.
\]
Since $\varphi \in \wt \X_{N/3,\rho}$ and $N/3 > 1$, Lemma~10.4 in~\cite{https://doi.org/10.48550/arxiv.2207.14351} ensures that equation~\eqref{∂ef:equacio_coeficients_constants} admits a solution $\eta \in \wt \X_{N/3-1,\rho}$, with
$\|\eta\|_{N/3-1,\rho} \le K \|\varphi\|_{N/3,\rho}$, for some $K>0$. The claim follows immediately.
\end{proof}

\begin{rem}
\label{rem:G_i_tilde_G}
In view of Lemma~\ref{lem:Inversa_de_loperador_lineal}, we define, for $H_1 \in \ZZ_{k,\ell,N,\rho}$, with $\pi_1 H_1 = q^kp^k(h_1+\tilde h_1)$, $h_1,\tilde h_1 \in \X_{N-2k,\rho}$,
the maps $\G_k, \wt \G_k: \ZZ_{k,\ell,N,\rho}\to \X_{N-k-3,\rho}$, such that $\G_k H_1 = f$ and $\wt \G_k H_1 = \tilde f$, where
$f$ and $\tilde f$ are the solutions given by Lemma~\ref{lem:Inversa_de_loperador_lineal} of the equations $\LL_k f = h_1$ and $\wt \LL_k \tilde f = \tilde h_1$, respectively. They satisfy
\[
\|\G_k H_1\|_{N-k-3,\rho}, \|\wt \G_k H_1\|_{N-k-3,\rho} \le C \llbracket \pi_k H_1 \rrbracket_{k,\ell,N,\rho}.
\]
\end{rem}

\begin{lem}[Iterative lemma]
\label{lem:lemma_iteratiu}
Let $H = \NN +H_1$, where $\NN$ was introduced in~\eqref{def:Hamiltonia_N_i_Hu},  and
$H_1 \in \ZZ_{1,\ell,N,\rho}$. Consider the Hamiltonian 
\[
F = p \G_1 H_1+ q \wt \G_1 H_1
\]
and let $\Phi_F^1$ be the time-one map defined by $F$. If, for $0<\rho'<\rho$,
\begin{equation}
\label{cond:lema_iteratiu}
K \frac{\rho^{N-2}}{(\rho-\rho')^2} \llbracket \pi_1 H_1\rrbracket_{1,\ell,N,\rho} < 1,
\end{equation}
then $H \circ \Phi_F^1 = \NN + H_2$, with  $H_2 \in \ZZ_{1,\ell,N,\rho'}$ satisfying
\[
\llbracket \pi_1 H_2\rrbracket_{1,\ell,N,\rho'}  \le  \frac{ K \rho^{N-2}}{(\rho-\rho')^3} \llbracket \pi_1 H_1 \rrbracket_{1,\ell,N,\rho}^2
\]
and 
\[
\llbracket\tilde \pi_1 H_2 \rrbracket_{2,\ell,N,\rho'} \leq \left(1+\frac{ K \rho^{N-2}}{(\rho-\rho')^2} \llbracket \pi_1 H_1 \rrbracket_{1,\ell, N,\rho} \right)  \left(\frac{K}{\rho-\rho'}\llbracket\pi_1 H_1 \rrbracket_{1,\ell,N,\rho}+\llbracket \tilde\pi_1 H_1 \rrbracket_{2,\ell,N,\rho}\right),
\]
for a constant $K>0$ which is independent of $\rho,\rho'$ and $N$.
\end{lem}

\begin{proof}
Let $f = \G_1 H_1$, $\tilde  f = \wt \G_1 H_1$. By Lemma~\ref{lem:Inversa_de_loperador_lineal}, they satisfy $f, \tilde f \in \X_{N-4,\rho}$ with $\|f\|_{N-4,\rho}, \|\tilde f\|_{N-4,\rho}\le C \llbracket\pi_1 H_1 \rrbracket_{1,\ell,N,\rho}$. By Proposition~\ref{prop:serie_de_Lie}, under the above hypotheses, $H \circ \Phi_F^1$ is well defined. It satisfies
\[
H \circ \Phi_F^1 = \NN + H_2,
\]
where
\begin{equation}
\label{descomposicio_H1prima}
 H_2 =  H_1^{[0]} + H_1^{[1]} +  H_1^{[2]} +H_1^{[3]}
\end{equation}
and
\begin{equation}
\label{def:H10...H14}
\begin{aligned}
 H_1^{[0]} & = \{\NN,F\} + \pi_1 H_1, \\ 
 H_1^{[1]} & = \NN \circ \Phi_F^1 - \NN - \{\NN,F\}, \\
H_1^{[2]} & =  (\pi_1H_1) \circ \Phi_F^1 - \pi_1H_1, \\
H_1^{[3]} & =  (\tilde\pi_1H_1 )\circ \Phi_F^1  
\\
\end{aligned}
\end{equation}
We estimate each term separately.  For $ H_1^{[0]}$ 
a direct computation shows that
\[
 H_1^{[0]} = q^2 p^2(g_2+\tilde g_2) + q^3 p^3 (g_3+\tilde g_3) + q^4 p^4 (g_4+\tilde g_4),
\]
where
\[
\begin{aligned}
g_2 & = -3 (f-q \partial_q f), &\qquad \tilde g_2 & = 3 (\tilde f-p \partial_p \tilde f) \\
g_3 & = -\frac{3}{q^2} \left(f - q \partial_q f \right), &\qquad \tilde g_3 & = \frac{3}{p^2} \left(\tilde f - p \partial_p \tilde f \right)  \\
g_4 & = -\frac{1}{q^4} \left(f - q \partial_q f \right), & \qquad \tilde g_4 & = \frac{1}{p^4} \left(\tilde f - p \partial_p \tilde f \right).
\end{aligned}
\]
Hence, $H_1^{[0]}\in \ZZ_{2,\ell, N,\rho'}$. In particular, $\pi_1  H_1^{[0]} = 0$ and, since $\|f\|_{N-4,\rho}, \|\tilde f\|_{N-4,\rho}\leq K \llbracket \pi_1 H_1 \rrbracket_{1,\ell,N,\rho}$, it follows from Lemmas~\ref{lem:multiplicacio_per_qp}~and~\ref{lem:derivades_a_ZNrho} that
\begin{equation}
\label{fita:H_sub_1_super_0}
  \llbracket H_1^{[0]}\rrbracket_{2,\ell,N,\rho'} \le \frac{K_1}{\rho-\rho'} \llbracket \pi_1 H_1 \rrbracket_{1,\ell,N,\rho}.
\end{equation}
Now, a direct application of Proposition~\ref{prop:serie_de_Lie} shows that 
\begin{align*}
\label{fita:H_sub_1_super_1}
\llbracket  H_1^{[1]}\rrbracket_{1,\ell, N,\rho'}   \le &\frac{K_1 \rho^{2N-4}}{(\rho-\rho')^3} \llbracket \pi_1 H_1 \rrbracket_{1,\ell,N,\rho}^2,\\
\llbracket H_1^{[2]}\rrbracket_{1,\ell,N,\rho'}\le & \frac{K_1 \rho^{N-2}}{(\rho-\rho')^2}\llbracket\pi_1 H_1\rrbracket_{1,\ell,N,\rho}^2,\\
\llbracket H_1^{[3]}\rrbracket_{2,\ell,N,\rho'}\le & \left(1+ \frac{K_1 \rho^{N-2}}{(\rho-\rho')^2}\llbracket \pi_1H_1\rrbracket_{1,\ell,N,\rho}\right)\llbracket \tilde\pi_1 H_1\rrbracket_{2,\ell,N,\rho}.
\end{align*}
Indeed, from last item of Proposition~\ref{prop:serie_de_Lie} (taking $k=n=1$ and $M=N-4$)  
\begin{align*}
\llbracket  H_1^{[1]}\rrbracket_{1,\ell, N,\rho'}   \le &\rho^{2N-4}\llbracket  H_1^{[1]}\rrbracket_{1,\ell, 3N-4,\rho'} \leq \frac{K_1 \rho^{(2N-4)}}{(\rho-\rho')^3} \left(\lVert f\rVert_{N-4,\rho}+\lVert \tilde f\rVert_{N-4,\rho} \right)^2\\
\leq &  \frac{K_1 \rho^{(2N-4)}}{(\rho-\rho')^3}\llbracket \pi_1 H_1 \rrbracket_{1,\ell,N,\rho}^2.
\end{align*}
Also, from Proposition~\ref{prop:serie_de_Lie} (again taking $k=n=1$ and $M=N-4$)
\[
\begin{aligned}
\llbracket  H_1^{[2]}\rrbracket_{1,\ell, N,\rho'}   \le \rho^{N-2} \llbracket  H_1^{[2]}\rrbracket_{1,\ell, 2N-2,\rho'} &\leq \frac{K_1 \rho^{N-2}}{(\rho-\rho')^2} \llbracket \pi_1 H_1 \rrbracket_{1,\ell,N,\rho} \left(\lVert f\rVert_{N-4,\rho}+\lVert \tilde f\rVert_{N-4,\rho} \right) \\
&\leq   \frac{K_1 \rho^{N-2}}{(\rho-\rho')^2}\llbracket \pi_1 H_1 \rrbracket_{1,\ell,N,\rho}^2
\end{aligned}
\]
and (taking $k=2,n=1$ and $M=N-4$) 
\[
\llbracket  H_1^{[3]}\rrbracket_{2,\ell, N,\rho'}   \le \left(1+ \frac{K_1 \rho^{N-2}}{(\rho-\rho')^2}\llbracket \pi_1 H_1 \rrbracket_{1,\ell,N,\rho}  \right)\llbracket \tilde \pi_1 H_1 \rrbracket_{2,1,N,\rho}.
\]
The estimate for $\pi_1 H_2$ follows from  $\pi_1  H_1^{[0]} =\pi_1  H_1^{[3]}= 0$.


\end{proof}

We now run the iterative Lemma~\ref{lem:lemma_iteratiu} to show that it is possible to kill all the terms of the form $qp(h_1(q,t)+\tilde h_1(p,t))$  appearing in the normal form given by Proposition~\ref{prop:redressament_de_les_varietats}. 

\begin{prop}
\label{prop:primer_pas_de_forma_normal_total}
Let $N,\ell\in \N$, with $N$ and $N-\ell$ large enough. Let 
\[
H = \NN  + H_1,
\]
be the normal form given by Proposition~\ref{prop:redressament_de_les_varietats}. Then, 
there exists a symplectic change of variables $\Phi:V_{\rho/2} \times V_{\rho/2} \times \T_\sigma  \to V_\rho \times V_\rho \times \T_\sigma $ such that 
\[
H \circ \Phi = \NN + H^{(2)},
\]
with  $H^{(2)} \in  \ZZ_{2,\ell,N,\rho/2}$ and 
\[
 \llbracket H^{(2)}\rrbracket_{2,\ell,N,\rho/2}\lesssim \llbracket\pi_1 H_1 \rrbracket^{5/6}_{1,\ell,N,\rho_0}+\llbracket\tilde\pi_1 H_1 \rrbracket_{2,\ell,N,\rho_0}.
\]
\end{prop}

\begin{proof}
Observe first that for $N$ large enough by decreasing it (if necessary) we can assume that  $
\llbracket \pi_1 H_1\rrbracket_{1,\ell,N,\rho}$ is sufficiently small. We now define the sequences $(\varepsilon_k)_k$, $(\delta_k)_k$, $(\rho_k)_k$, and $(\kappa_k)_k$, $k\ge 0$, in the following way. For $k=0$, we take
\[
\rho_0 = \rho, \quad\varepsilon_0 = \llbracket \pi_1 H_1 \rrbracket_{1,\ell,N,\rho_0}, \quad \delta_0 = \varepsilon_0^{1/6}, \quad \tilde\varepsilon_0 = \max\{\delta_0^{-1}\varepsilon_0,\llbracket \tilde \pi_1 H_1 \rrbracket_{2,\ell,N,\rho_0}\}
\]
and, for $k\ge 1$ 
\[
\varepsilon_k =  \varepsilon_{k-1}^{3/2}, \qquad \delta_k = \varepsilon_k^{1/6}, \qquad \rho_k = \rho_{k-1}- \delta_{k-1}, \qquad \tilde\varepsilon_k=(1+2\varepsilon_k^{2/3})\tilde \varepsilon_{k-1}
\]
Suppose now that, at step $k\in\mathbb N$ we have a Hamiltonian 
\[
\tilde H_n=\mathcal N+ H_n,
\]
with  $H_n\in \mathcal Z_{1,\ell, N,\rho_n}$ 
\[
\llbracket \tilde \pi_1 H_n \rrbracket_{1,\ell,N,\rho_n}\leq \tilde\varepsilon_n, \qquad
\llbracket \pi_1 H_n\rrbracket_{1,\ell,N,\rho_n}\leq  \varepsilon_n.
\]
Then, Lemma~\ref{lem:lemma_iteratiu} shows the existence of $F_n\in\mathcal Z_{1,\ell,N,\rho_n}$ such that 
\[
H_n\circ\Phi_{F_n}=\mathcal N+H_{n+1}
\]
with $H_{n+1}\in \mathcal Z_{1,\ell, N,\rho_{n+1}}$ 
\[
\llbracket \tilde \pi_1 H_n \rrbracket_{1,\ell,N,\rho_n}\leq  \left(1+ \delta_n^{-2}\varepsilon_n\right) \tilde\varepsilon_n+\delta_n^{-1}\varepsilon_n\leq\left(1+ 2\delta_n^{-2}\varepsilon_n\right) \tilde\varepsilon_n=\tilde\varepsilon_{n+1}
\]
and
\[
\llbracket \pi_1 H_n\rrbracket_{1,\ell,N,\rho_n}\leq  \delta_n^{-3}\varepsilon_n^2=\varepsilon_{n+1}.
\]
The conclusion follows from a standard argument.

\end{proof}

\subsection{Proof of Theorem~\ref{thm:forma_normal}}
We now complete the proof of Theorem~\ref{thm:forma_normal}.
Let $\rho>0$, $2\leq n\leq 3$ and consider a Hamiltonian of the form
\[
\tilde H^{(n)}=\mathcal N+ H^{(n)}
\]
with $\llbracket \pi_n H^{(n)}\rrbracket_{n,\ell,N,\rho}$ sufficiently small.
Then, the very same argument deployed in Proposition~\ref{prop:primer_pas_de_forma_normal_total} shows that there exists a change of variables $\Phi^{(n)}:V_{\rho/2} \times V_{\rho/2} \times \T_\sigma  \to V_\rho \times V_\rho \times \T_\sigma $
\[
\tilde H^{(n+1)}=\tilde H^{(n)}\circ\Phi_F=\mathcal N+ H^{(n+1)}
\]
with   $H^{(n+1)} \in  \ZZ_{n+1,\ell,N,\rho/2}$ and 
\[
\llbracket H^{(n+1)}\rrbracket_{n+1,\ell,N,\rho/2}\lesssim \llbracket \pi_n H^{(n)} \rrbracket^{5/6}_{n+1,\ell,N,\rho}+\llbracket \tilde\pi_n H^{(n)} \rrbracket_{n+1,\ell,N,\rho}.
\]
Thus, it is possible to find a symplectic  change of variables $\Phi:V_{\rho/8} \times V_{\rho/8} \times \T_\sigma  \to V_\rho \times V_\rho \times \T_\sigma $ which recasts the Hamiltonian obtained in Proposition~\ref{prop:redressament_de_les_varietats} as
\[
H=\mathcal N+\mathcal H^{(4)}
\]
with 
$\mathcal H^{(4)} \in  \ZZ_{4,4,N,\rho/8}$ and 
\[
\llbracket \mathcal H^{(4)}\rrbracket_{4,4,N,\rho/8}\lesssim  \llbracket H_1 \rrbracket_{1,\ell,N,\rho}.
\]
The proof of Theorem~\ref{thm:forma_normal} is complete.

\section{Construction of the Shilnikov maps}\label{sec:Shilnikov}
In this section we prove the existence of stable and unstable Shilnikov maps in \eqref{eq:Shilnikovmapsdefn}. Namely, we prove Proposition~\ref{prop:Shilnikov}.

\subsection{Time reparametrization of the flow close to the degenerate saddle}

In order to study the dynamics close to the degenerate saddle it is convenient to first rescale time so that the flow of the Hamiltonian \eqref{eq:normalform} is equivalent to that of the system of integro-differential equations 
\begin{equation}\label{eq:rescaleddynamics}
\dot q (s)=\partial_p \mathcal K(q(s),p(s),t(q,p,s)),\qquad\qquad\dot p(s)=-\partial_q \mathcal K(q(s),p(s),t(q,p,s)).
\end{equation}
where $t(q,p,s)$ stands for the functional on the space of continuous curves $(q,p):[0,s]\to\mathbb{R}^2$ given by
\[
t(q,p,s)=\int_0^s (q(\sigma)+p(\sigma))^{-3}\mathrm{d}\sigma.
\]
Now notice that, for any $s_*>0$ and any $\xi>0, \eta>0$, solutions to the fixed point equation 
\begin{equation}\label{eq:fixedpointShilnikov}
z=\mathcal{F}_{\xi,\eta}(z,s)
\end{equation}
where $\mathcal{F}_
{\xi,\eta}=(F_{\xi,\eta},G_{\xi,\eta})$ is given by 
\begin{equation}
\begin{aligned}\label{eq:integraloperator}
F_{\xi,\eta}(z,s)=&\xi \exp(-s)+\int_0^{s} \exp(-(s-\sigma)) (\partial_p \mathcal K (q(\sigma),p(\sigma),t(q,p,\sigma))+q(\sigma))\mathrm{d}\sigma,\\
         G_{\xi,\eta}(z,s)=& \eta \exp(s-s_*)-\int_{s_*}^s \exp(s-\sigma) (\partial_q \mathcal K (q(\sigma),p(\sigma),t(q,p,\sigma))-p(\sigma))\mathrm{d}\sigma.
\end{aligned}
\end{equation}
correspond to solutions of \eqref{eq:rescaleddynamics} which moreover satisfy $q(0)=\xi$ and $p(s_*)=\eta$.
\subsection{Existence of fixed points of the integral operator $\mathcal F_{\xi,\eta}$}

In order to study the existence of fixed points of the operator $\mathcal{F}_{\xi,\eta}$ in \eqref{eq:integraloperator}, we introduce a suitable functional setting. 
We fix $\delta>0$ and let $\Sigma_\delta$ be the complex domain introduced in \eqref{eq:complexextensionsigma}. Given any 
\[
0<s*<\infty,\qquad\qquad (\xi,\eta)\in \Sigma_\delta\subset \mathbb{C}^2,
\]
we introduce the space of functions (we write $z=(q,p)$)
\begin{align*}
\mathcal{X}_{\xi,\eta}=\{z(s;\xi,\eta)\in \mathcal C^{\omega}((0,s_*),\mathbb{R}^2)\colon q(0;\xi,\eta)=\xi,\ p(s_*;\xi,\eta)=\eta \text{ and }\ \lVert z\rVert<\infty\},
\end{align*}
where we have defined
\[
\lVert z\rVert= \max_{s\in[0,s_*]} \left(|q(s)\exp(s)|+ | p(s) \exp(s_*-s)|\right).
\]

We also denote by $z_0=(\xi\exp(-s),\eta\exp(s-s_*))$ (namely, the solution of the linear system) and introduce, for $\rho>0$, the closed ball
\[
\mathbb{B}_{\rho,\xi,\eta}=\{z\in\mathcal{X}_{\xi,\eta}\colon \lVert z-z_0\rVert\leq \rho\}.
\]
In Lemma~\ref{lem:inequalitiesrescaledflow} we exploit the particular structure of the vector field \eqref{eq:rescaleddynamics} to establish suitable estimates for the behavior of the integral operator \eqref{eq:integraloperator} when acting on paths $z\in\mathbb{B}_{\rho}$ for $\rho>0$ sufficiently small. We first need a couple of  technical lemmas.

\begin{lem}\label{lem:techlemmaintegralt}
    Let $s_*>0$ be sufficiently large. Then, for any $s\in[0,s_*]$
    \[
    \int_0^s (\exp(-\sigma)+\exp(\sigma-s_*))^{-3}\mathrm{d}\sigma\lesssim \exp(3s_*/2).
    \]
\end{lem}
\begin{proof}
    We introduce the change of variables $\sigma=u-s_*/2$ so the integral above can be estimated by
    \begin{align*}
    \int_0^s (\exp(-\sigma)+\exp(\sigma-s_*))^{-3}\mathrm{d}\sigma\leq  &\int_0^{s_*} (\exp(-\sigma)+\exp(\sigma-s_*))^{-3}\mathrm{d}\sigma \\
    =&\exp(3s_*/2)\int_{-s_*/2}^{s_*/2} (\exp(-u)+\exp(u))^{-3}\mathrm{d}u\\
    \lesssim &\exp(3s_*/2)\int_{\mathbb{R}} (\exp(-u)+\exp(u))^{-3}\mathrm{d}u\\\
    \lesssim& \exp(3s_*/2).
    \end{align*}
\end{proof}
The following lemma is key to make use of Proposition~\ref{prop:normalform} when considering complex $(\xi,\eta)$.
\begin{lem}\label{lem:techlemmaimaginary}
    Let $R_\kappa$ denote the $\kappa$-neighborhood of the real line in $\mathbb C$. Then,  there exists $C>0$,  depending
only on $\kappa$, and such that, for 
\begin{equation}\label{eq:deltastar}
0\leq \delta\leq \delta_*=C \exp(-3s_*/2)
\end{equation}
and for any $\rho\leq \delta/2$, any $z\in\mathbb B_{\rho,\xi,\eta}$ with $(\xi,\eta)\in \Sigma_\delta$ satisfies that, for any $0\leq s\leq s_*$,
\[
(q(s),p(s),t(q(s),p(s),s))\in V_{\rho}\times V_{\rho} \times R_\kappa.
\]
\end{lem}

\begin{proof}
    Given a complex number $\tau\in\mathbb C$, we write
    \[
    \alpha(\tau)=\frac{|\mathrm{Im}(\tau)|}{|\mathrm{Re}(\tau)|}
    \]
    and notice that, for $(\xi,\eta)\in\Sigma_\delta$ with $\delta\ll a$ (recall $a>0$ is a small but fixed constant)
    \[
    \alpha(\xi),\alpha(\eta)\leq\frac{\delta}{a/2-\delta}\lesssim \delta.
    \]
    We now write
    \begin{align*}
    \mathrm{Re}(q)(s)=&\mathrm{Re}(q_0(s))+\mathrm{Re}(q(s)-q_0(s))=(\mathrm{Re}(\xi)+\mathrm{O}(\rho))\exp(-s)\\
    \mathrm{Im}(q)(s)=&\mathrm{Im}(q_0(s))+\mathrm{Im}(q(s)-q_0(s))=(\mathrm{Im}(\xi)+\mathrm{O}(\rho))\exp(-s)
    \end{align*}
    and it follows that, for $\delta$ and $\rho$ sufficiently small (depending only on $\kappa$)
    \[
    \alpha(q(s))\lesssim\frac{\delta+\rho}{a/2-\delta-\rho}\lesssim \kappa.
    \]
    Analogously, one shows that $\alpha(p(s))\lesssim \kappa$.
    
    It only remains to show that $t(q(s),p(s),s)\in R_\kappa$ for $\delta\leq \delta_*$. 
    We notice that 
    \[
    \mathrm{Im}(t)(s)=-\int_0^s\frac{\mathrm{Im}((q(\sigma)+p(\sigma))^3)}{|(q(\sigma)+p(\sigma))|^6}\mathrm{d}\sigma.
    \]
    Now, a trivial computation shows that, for $\rho<\delta/2\ll a$  we have 
    \begin{align*}
    |\mathrm{Im}((q(\sigma)+p(\sigma))^3)|\lesssim &(a^2\delta+\delta^3) \left( \exp(-\sigma)+\exp(\sigma-s_*)\right)^3\\
    \lesssim &\delta\left( \exp(-\sigma)+\exp(\sigma-s_*)\right)^3
    \end{align*}
    and
    \[
    |q(\sigma)+p(\sigma)|\gtrsim \exp(-\sigma)+\exp(\sigma-s_*).
    \]
    Thus, making use of Lemma~\ref{lem:techlemmaintegralt},
    \[
    |\mathrm{Im}(t)(s)|\lesssim \delta \int_0^s (\exp(-\sigma)+\exp(\sigma-s_*))^{-3}\mathrm{d}\sigma \leq \delta \exp(3s_*/2) \leq C,
    \]
    and the lemma is proved.
\end{proof}

We now provide suitable estimates on the action of the operator $\mathcal F_{\xi,\eta}$.
\begin{lem}\label{lem:inequalitiesrescaledflow}
Let $0\leq \rho\leq \delta/2\leq \delta_*$, with $\delta_*$ as in \eqref{eq:deltastar}. Then, for any $(\xi,\eta)\in \Sigma_\delta$ and  any $s\in[0,s_*]$, 
\begin{equation}\label{eq:firstineqinteg}
\lVert \mathcal F_{\xi,\eta}(z_0,s)-z_0\rVert\lesssim a^{10}\exp(-3s_*).
\end{equation}
(recall that $a$ is a small but fixed constant used in the definition of $\Sigma$ in \eqref{eq:dfnsigmadelta}). Moreover, for  any pair $z,z_*\in\mathbb{B}_{\rho,\xi,\eta}$, we have 
\begin{equation}\label{eq:lipschitz}
\lVert \mathcal F_{\xi,\eta}(z_*,s)-\mathcal F_{\xi,\eta}(z,s)\rVert\lesssim a^5 \exp(-3s_*/2).
\end{equation}
\end{lem}
\begin{proof}
Let
\begin{align*}
f(q,p,s)=&\partial_p \mathcal K(q(s),p(s),t(q,p,s))+q,\\
g(q,p,s)=&\partial_q \mathcal K (q(s),p(s),t(q,p,s))-p.
\end{align*}
From the definition of $f,g$ and the form of the Hamiltonian \eqref{eq:normalform} we observe that 
\begin{align*}
f(q,p,s)=&q^4(s)p^{3}(s)\tilde f(q(s),p(s),t(q,p,s))\\
g(q,p,s)=&q^{3}(s)p^4(s)\tilde g(q(s),p(s),t(q,p,s))
\end{align*}
with (for $1\leq n\leq 2$)
\begin{align}\label{eq:auxestimateslipschitzconst}
|\tilde f (q,p,t)|,|\tilde g(q,p,t)|=&\mathcal{O}_3(q,p),\\
|q^n\partial^n_q \tilde f(q,p,t)|,|p^n\partial^n_p\tilde f(q,p,t)|,|q^n\partial^n_q \tilde g(q,p,t)|,|p^n\partial^n_p \tilde g(q,p,t)|=&\mathcal{O}_3(q,p),\\
|\partial_t \tilde f(q,p,t)|,|\partial_t \tilde g(q,p,t)|=&\mathcal{O}_3(q,p),\label{eq:auxestimateslipschitzconst2}
\end{align}
uniformly for $(q,p,t)\in V_\rho \times V_\rho \times R_\kappa$. We now show how to obtain \eqref{eq:firstineqinteg}. In view of Lemma~\ref{lem:techlemmaimaginary}, we have that  $(q(s),p(s),t(q(s),p(s),s))\in V_{\rho}\times V_{\rho}\times R_\kappa$, provided $\delta \leq \delta_*\sim \exp(-3s_*/2)$. Then, the estimate \eqref{eq:auxestimateslipschitzconst} implies that
\begin{align*}
    |q^4(s)p^{3}(s)\tilde f(q(s),p(s),t(q,p,s))|\lesssim& a^{10}\exp(-(s+3s_*))\\ 
     |q^{3}(s)p^{4}(s)\tilde g(q(s),p(s),t(q,p,s))|\lesssim& a^{10}\exp(s-4s_*).
\end{align*}
Thus, we have proved \eqref{eq:firstineqinteg}.

We now show how to obtain \eqref{eq:lipschitz}. The only delicate point is to study the behavior of the functional $t:\mathbb{B}_{\rho}\to\mathbb{C}$. To that end, we use the fundamental theorem of calculus to write
    \begin{align*}
    t(q_*,p_*,s)-t(q,p,s)=3\int_0^s &\bigg(\big(q(\sigma)-q_*(\sigma)+p(\sigma)-p_*(\sigma)\big)\\
    &\times \int_0^1 \frac{\mathrm{d}\lambda}{\left(\lambda (q(\sigma)+p(\sigma))+(1-\lambda) (q_*(\sigma)+p_*(\sigma))\right)^4} \bigg)\mathrm{d}\sigma.
    \end{align*}
    Notice that, for any $\lambda\in[0,1]$ and any $0\leq \sigma\leq s$ we have that, for $z\in\mathbb B_\rho$,
    \begin{align*}
    |\lambda (q(\sigma)+p(\sigma))+&(1-\lambda) (q_*(\sigma)+p_*(\sigma))|\\
    \gtrsim &(\xi-\mathcal{O}(\rho))\exp(-\sigma)+(\eta-\mathcal{O}(\rho))\exp(\sigma-s_*)\\
    \geq &\frac a2\left( \exp(-\sigma)+\exp(\sigma-s_*)\right).
    \end{align*}
    Hence, it follows that 
    \begin{align*}
    |t(q_*,p_*,s)-t(q,p,s)|\lesssim a^{-4} &\int_0^s\frac{|q(\sigma)-q_*(\sigma)|}{(\exp(-\sigma)+\exp(\sigma-s_*))^4}\mathrm{d}\sigma\\
    &+\int_0^s \frac{|p(\sigma)-p_*(\sigma)|}{\exp(-\sigma)+\exp(\sigma-s_*))^4}\mathrm{d}\sigma\\
    \lesssim a^{-4} & \lVert z-z_*\rVert \bigg( \int_0^s\frac{\exp(-\sigma)}{(\exp(-\sigma)+\exp(\sigma-s_*))^4}\mathrm{d}\sigma \\
    &+\int_0^s\frac{\exp(\sigma-s_*)}{(\exp(-\sigma)+\exp(\sigma-s_*))^4} \mathrm{d}\sigma\bigg)
    \\
    \lesssim a^{-4} & \exp(3s_*/2)\lVert z-z_*\rVert,
    \end{align*}
    where, in the last inequality, we have used the change of variables $\sigma=s_*/2+u$ so, for any $0\leq s\leq s_*$,
    \begin{align*}
\int_0^s\frac{\exp(-\sigma)}{(\exp(-\sigma)+\exp(\sigma-s_*))^4}\mathrm{d}\sigma =& \exp(s_*)\int_{-s_*/2}^{s-s_*/2}\frac{\exp(-u)}{(\exp(-u)+\exp(u))^4}\mathrm{d}u\\
\lesssim &\exp(3s_*/2) \int_{-\infty}^\infty  \frac{\exp(-u)}{(\exp(-u)+\exp(u))^4}\mathrm{d}u\\
\lesssim &\exp(3s_*/2)
    \end{align*}
and a similar argument for the second integral. This implies that
\[
|q^4(s)p^{3}(s)||\tilde f(q,p,t(q(s),p(s),s))-\tilde f(q,p,t(q_*(s),p_*(s),s))| 
\lesssim  a^{6}\exp(-(s+3s_*/2)) \lVert z-z_*\rVert
\]
and
\[
|q^{3}(s)p^4(s)||\tilde g(q,p,t(q(s),p(s),s))-\tilde g(q,p,t(q_*(s),p_*(s),s))| 
\lesssim  a^{6} \exp(s-5s_*/2) \lVert z-z_*\rVert_.
\]
    A rather lengthy, but straightforward, number of similar (easier indeed) computations, together with the inequalities \eqref{eq:auxestimateslipschitzconst}-\eqref{eq:auxestimateslipschitzconst2} show that \eqref{eq:lipschitz} holds.
\end{proof}

The following is a trivial application of the fixed point theorem and Lemma~\ref{lem:inequalitiesrescaledflow}.
\begin{prop}\label{lem:existencefixedpointoperator}
Let $0\leq  \delta/2\leq \delta_*$ with $\delta_*$ as in \eqref{eq:deltastar} and let $(\xi,\eta)\in \Sigma_\delta$.
Then, there exists $C>0$ such that, for any $s_*>0$ sufficiently large and 
\[
\rho= C a^{10} \exp(-3s_*)
\]
 the operator $\mathcal{F}:\mathbb{B}_{\rho/2,\xi,\eta}\to\mathbb{B}_{\rho,\xi,\eta}$ is well defined and  has a unique fixed point $z_{*,\xi,\eta}\in\mathbb{B}_{\rho,\xi,\eta}$. 
 Moreover, for any $s\in[0,s_*]$, the dependence 
 \[
 (\xi,\eta)\mapsto z_{*,\xi,\eta}(s)
 \]
 is real-analytic in the interior of $\Sigma_\delta$.
\end{prop}
\vspace{0.3cm}

\subsection{Coming back to the original time parametrization} 
Let $0\leq \delta/2\leq \delta_*$ with $\delta_*$ as in \eqref{eq:deltastar} and let $z_{*,\xi,\eta}\in\mathbb{B}_\rho$ be the solution to the fixed point equation \eqref{eq:fixedpointShilnikov} obtained in Proposition~\ref{lem:existencefixedpointoperator}. Then, for any  $T$ sufficiently large, we define $s_*(\xi,\eta,T)$ implicitly by the equation
\begin{equation}\label{eq:undotimereparametrization}
T=\int_0^{s_*}\frac{\mathrm{d}s}{(q_*(s)+p_*(s))^3}.
\end{equation}
In the following lemma we provide an asymptotic expression for $s_*(\xi,\eta,T)$ with $T$ sufficiently large.
\begin{lem}
    There exists $T_0$ large enough such that, for any $(\xi,\eta)\in {\Sigma}_\delta$ and any  $T\geq T_0$, there exists a unique $s_*(\xi,\eta,T)$ satisfying \eqref{eq:undotimereparametrization}. Moreover, the following asymptotic expression 
    \begin{align}\label{eq:asymptoticexpressionsstar}
    \exp(-s_*(\xi,\eta,T))&=\left(\frac{\pi}{16T} \right)^{\frac 23} \frac{1}{\xi\eta}\left(1-\frac {1}{3T (\xi\eta)^{\frac 32}} \left( \left(\frac{\eta}{\xi}\right)^{\frac 32}+\left(\frac{\xi}{\eta}\right)^{\frac 32}\right)+\mathcal{O}(T^{-5/3})\right)^{\frac 23},
    \end{align}
    and the estimates (for $\alpha=(\alpha_\xi,\alpha_\eta)$, $1\leq |\alpha|\leq 2$)
    \begin{equation}\label{eq:derivativessstar}
    |\partial^\alpha s_*(\xi,\eta,T)|\lesssim 1,
    \end{equation}
    hold uniformly on $\Sigma_\delta$.
\end{lem}

\begin{proof}
    We first notice that, for $\tilde q=q_*-q_0$ and $\tilde p=p_*-p_0$ we have, uniformly for $0\leq s\leq s_*$ and $(\xi,\eta)\in\Sigma_\delta$,
    \[
    \frac{|\tilde q(s)+\tilde p(s)|}{|q_0(s)+p_0(s)|}\lesssim \frac{\exp(-s)+\exp(s-s_*)}{|q_0(s)+p_0(s)|} \lVert z_*-z_0\rVert \lesssim a^{10}\exp(-3s_*).
    \]
    Hence,
    \[
    T=\int_0^{s_*}\frac{\mathrm{d}s}{(q_*(s)+p_*(s))^3}=\int_0^{s_*}\left(1+\mathcal{O} (a^{10}\exp(-3s_*))\right)\frac{\mathrm{d}s}{(q_0(s)+p_0(s))^3}.
    \]
    Also, notice that, introducing the change of variables given by 
    \[
u=\frac12\log \left(\frac{\eta}{\xi} \right)+s-s_*/2,
    \]
    we have 
    \begin{align*}
I(\xi,\eta,s_*)=&\int_0^{s_*}\frac{\mathrm{d}s}{(q_0(s)+p_0(s))^3}\\
=&\frac{\exp(3s_*/2)}{(\xi\eta)^{3/2}} \int_{-s_*/2+\log(\sqrt{\eta/\xi})}^{s_*/2+\log(\sqrt{\eta/\xi})} \frac{\mathrm{d}u}{(\exp(u)+\exp(-u))^3}.
    \end{align*}
    One might check that 
    \[
    G(\sigma)=\int_{-\infty}^\sigma\frac{\mathrm{d}u}{(\exp(u)+\exp(-u))^3}=\frac18 \left(\frac{e^\sigma(e^{2\sigma}-1)}{(e^{2\sigma}+1)^2}+\tan^{-1} (e^\sigma)\right),
    \]
    so 
    \begin{equation}\label{eq:asymptotictimeintegral}
    G(\sigma)=\frac {\pi}{16}-\frac{1}{3}e^{-3\sigma}+\mathcal{O}(e^{-5\sigma})\quad\text{as}\quad\sigma\to \infty
    \end{equation}
    and
\begin{equation}\label{eq:asymptotictimeintegral2}
    G(\sigma)=\frac13 e^{3\sigma}+\mathcal{O}(e^{5\sigma})\quad\text{as}\quad\sigma\to -\infty.
    \end{equation}
    Thus, for large $s_*$ and uniformly for $(\xi,\eta)\in \Sigma_\delta$,
    \begin{align*}
T(\xi,\eta,s_*)=&I(\xi,\eta,s_*)\left(1+\mathcal{O} (a^{10}\exp(-3s_*))\right)\\
=&\frac{\exp(3s_*/2)}{(\xi\eta)^{3/2}}
\left(\frac{\pi}{16}-\frac 13 c(\xi,\eta)\exp(-3s_*/2)(1+\mathcal{O}(\exp(-s_*))) +\mathcal{O}(a^{10}\exp(-3s_*))\right).
    \end{align*}
    where we have written $c(\xi,\eta)=(\eta/\xi)^{3/2}+(\xi/\eta)^{3/2}$. In order to obtain \eqref{eq:asymptoticexpressionsstar} we rewrite the last equality as
    \[
    \exp(-s_*)= \left(\frac{\pi}{16T} \right)^{\frac 23} \frac{1}{\xi\eta}\left(1-\frac {16}{3\pi}c(\xi,\eta) \exp(-3s_*/2)(1+\mathcal{O}(\exp(-s_*)))+\mathcal{O}(\exp(-3s_*))\right)^{\frac 23}
    \]
    and make use of the implicit function theorem.

    We now show how to obtain \eqref{eq:derivativessstar}. It follows from the asymptotic expression \eqref{eq:asymptoticexpressionsstar} and Cauchy estimates that for $\alpha=(\alpha_\xi,\alpha_\eta)$ with $|\alpha|=1$ and any $(\xi,\eta)\in\Sigma_{\delta/2}$,
    \[
    -\partial^\alpha s_* \exp(-s_*)=\left(\frac{\pi}{16T} \right)^{\frac 23} \left(\partial^\alpha \left( \frac{1}{\xi\eta}\right) \left(1+\mathcal{O}(T^{-1})\right)+ \frac{1}{\xi\eta} \left(1+\mathcal{O}\left(T^{-1}+ \delta^{-1}T^{-2}\right)\right)\right).
    \]
    By the above expression, $\exp(s_*)\sim T^{2/3}$. Then, $\delta\gtrsim \exp(-3s_*/2)\sim T^{-1}$,  and we obtain that, uniformly in $(\eta,\xi)\in  \Sigma_{\delta/2}$,
    \begin{align*}
    \partial^\alpha s_*= &\left(\frac{\pi}{16T} \right)^{\frac 23}  \exp(s_*)\left(\partial^\alpha \left( \frac{1}{\xi\eta}\right) \left(1+\mathcal{O}(T^{-1})\right)+ \frac{1}{\xi\eta} \left(1+\mathcal{O}\left(T^{-1}\right)\right)\right)\\
    =&\mathcal O (1).
    \end{align*}
    The estimate for $|\alpha|=2$ follows in a similar fashion making use of the above expression and noticing that 
    \[
    \delta^{-2}T^{-2}\lesssim 1.
    \]
\end{proof}

We are now ready to complete the proof of Proposition~\ref{prop:Shilnikov}.

\begin{proof}[Proof of Proposition~\ref{prop:Shilnikov}]
Let $(q_*(s;\xi,\eta),p_*(s;\xi,\eta))$ be the solution to the fixed point equation \eqref{eq:fixedpointShilnikov} found in Proposition~\ref{lem:existencefixedpointoperator}. The Shilnikov stable and unstable maps are given, respectively, by
\[
\Psi_T^s(\xi,\eta)=(\xi,p_*(0;\xi,\eta)),\qquad\qquad \Psi_T^u(\xi,\eta)=(q_*(s_*(\xi,\eta,T);\xi,\eta),\eta).
\]
Denote by 
\[
\Delta_q(s)=e^{s}\left(q_*(s)-q_0(s)\right),\qquad\qquad\Delta_p(s)=e^{s_*-s}\left(p_*(s)-p_0(s)\right).
\]
Then, since for $0\leq s\leq s_*$ we have $\Delta_q(s),\Delta_p(s)=\mathcal{O}(a^{10}\exp(-3s_*))$, 
\begin{align*}
q_*(s)=&\exp(-s)\left(\xi+\mathcal O(a^{10}\exp(-3s_*))\right),\\
p_*(s)=&\exp(s-s_*)\left(\eta+\mathcal O(a^{10}\exp(-3s_*))\right).
\end{align*}
It thus follows from \eqref{eq:fixedpointShilnikov} and \eqref{eq:integraloperator} that (we use again the notation $c(\xi,\eta)=(\eta/\xi)^{3/2}+(\xi/\eta)^{3/2}$)
\[
q_*(s_*(\xi,\eta,T);\xi,\eta)=\left(\frac{\pi}{16T} \right)^{2/3} \frac{1}{\eta}\left(1-\frac {2}{9T (\xi\eta)^{3/2}}c(\xi,\eta)+\mathcal{O}(a^{-10}T^{-2})\right)
\]
and
\[
p_*(0;\xi,\eta)=\left(\frac{\pi}{16T} \right)^{2/3} \frac{1}{\xi}\left(1-\frac {2}{9T (\xi \eta)^{3/2}}c(\xi,\eta)+\mathcal{O}(a^{-10}T^{-2})\right),
\]
so the asymptotic expansion \eqref{eq:asymptoticShilnikov} follows. Since \eqref{eq:asymptoticShilnikov} holds uniformly on $\Sigma_\delta$ with $\delta= T^{-1}$, the asymptotic expressions for the first partial derivatives \eqref{eq:asymptoticShilnikovderiv},\eqref{eq:asymptoticShilnikovderiv2} and the estimates for the second partial derivatives follow from Cauchy estimates.
\end{proof}

\section{Secondary  homoclinic tangencies to the degenerate saddle}\label{sec:tangency}

In this section we prove Proposition~\ref{lem:homoclinictangencies}. To that end, we fix $a>0$ sufficiently small so that $[0,2a]^2\subset V_\rho \times V_\rho$, where $V_\rho$ is the sectorial domain in Proposition~\ref{prop:normalform}. Throughout the whole section we work in the coordinate system $(q,p)\in V_\rho \times V_{\rho} \subset\mathbb C^2$ obtained in Proposition~\ref{prop:normalform}. In particular, we recall that in this coordinate system the local stable and unstable manifolds of the fixed point $O=(0,0)$ are given by $\{p=0\}$ and $\{q=0\}$ respectively. We denote by 
\[
\Psi_{\mathcal K}:V_{\rho} \times V_{\rho} \to \mathbb{C}^2
\]
the time-one map associated to the flow of Hamiltonian $\mathcal K$ in \eqref{eq:normalform}.

In the following lemma we obtain a parametrization of a suitable compact piece of the unstable manifold as a graph over the stable manifold. To do so, the introduction of some notation is in order. We denote by 
\[
U_a=\{a/2<q<2a,\ q-p<0\}.
\]
For $\mu>0$ small enough we denote by $\gamma^u$ the unique connected component of $W^{u}(O)\cap U_a$ corresponding to the first backwards intersection of $W^{u}_{\mathrm{loc}}(O)$ with $U_a$. Clearly, by regular dependence of compact pieces of $W^{u}(O)$ on $\mu$ (see 
 \cite{McGeheestablemfold}), the curve $\gamma^u$ is well defined for $\mu>0$ sufficiently small.

\begin{lem}\label{lem:transversesplitting}
 Fix any $a>0$, small enough. Then, for any sufficiently large value of the Jacobi constant~$J$  and for $\mu>0$, small enough (depending on $J)$, there exist a constant $\sigma_{J}>0$  and a real-analytic function $M_\mu:[a/2,2a]\to \mathbb R$ of the form
 \begin{equation}\label{eq:Melnikovinq} 
     M_{\mu}(q;J)=\frac{1}{2q}\left(\mu \left(\sigma_J\sin ((4\mathcal J/q)^3) +\mathcal{E}_ J(q)\right)+\mathcal O(\mu a)+\mathcal{O}(\mu^2)\right)
 \end{equation}
 with $\lim_{J\to\infty} \mathcal E_J(q)/\sigma_{J}\to 0$ uniformly for $q\in [a/2,2a]$ and   such that $\gamma^u$ can be parametrized as
   \begin{equation}\label{eq:transverseparam}
\gamma^u=\{(q,M_\mu(q; J))\colon a/2<q<2a\}.
   \end{equation}
\end{lem}

Lemma~\ref{lem:transversesplitting} follows from the results in \cite{MartinezPinyol}. We provide the details in Appendix~\ref{sec:splittingr3bp}. We now make use of the parametrization \eqref{eq:transverseparam} to prove the existence of secondary homoclinic tangencies. Following Duarte, to exploit the (approximate) symmetry of the problem we compare \eqref{eq:transverseparam} with images of the diagonal
\[
\Delta=\{(q,p)\in\mathbb{R}^2\colon q=p,\ (q,p)\in [0,a/4]^2\}
\]
under the local map.  In the following we assume that $\mu>0$ is sufficiently small.

\begin{lem}\label{lem:diagonal}
    For any $T\in 2\mathbb{N}$ sufficiently large denote by $
    \gamma_{T}^-$ (resp. $\gamma_T^+$)  the unique connected component of  $\Phi_{loc}^{-T/2}(\Delta)\cap U_a$ (resp. $\Phi_{loc}^{+T/2}(\Delta)\cap U_a$) corresponding to the first backwards (resp. forward) intersection of $\Delta$ and $U_a$.  Let $x_T(q,p):\Sigma_\delta\to\mathbb C$ and $y_T(q,p):\Sigma_\delta\to\mathbb C$ denote the functions in \eqref{eq:formShilnikovmaps} (i.e. the first component of the Shilnikov unstable map, and the second component of the Shilnikov stable map respectively). Then, for $\mu>0$ small enough, $\gamma_T^\pm$ can be parametrized as
    \[
    \gamma_T^-=\{(q,y_T(q,q))\colon q\in[a/2,2a]\}, \qquad   \gamma_T^+=\{(x_T(p,p),p)\colon p\in[a/2,2a]\}.
    \]
\end{lem}
\begin{proof}
    We first prove that $\gamma_T^-$ can be parametrized as a graph of the form
    \[
    \gamma_T^-=\{(q,h(q)),q\in[a/2,2a]\},
    \]
    for some function $h(q)$. To see this, we consider the image of $\Delta$ under the flow 
    \[
    \dot q=(q+p)^3 \partial_p \mathcal K,\qquad\qquad \dot p=-(q+p)^3 \partial_q \mathcal K,
    \]
    where $\mathcal{K}$ is as in \eqref{eq:normalform}. We rescale time $\mathrm{d}t/\mathrm{d}s=-(q+p)^{-3}$ (recall that we want to consider negative iterations). In the rescaled time, it follows from the structure of $\mathcal K$ that, as long as it stays on $V_\rho \times V_\rho$, the time $s$ map $\phi^s$ of a point $(\tau,\tau)\in\Delta$ is given by 
    \begin{equation}\label{eq:auxiliaryrescaledflow}
\phi^s(\tau,\tau)=\tau\big((1+\mathcal{O}(a)) \exp(s),(1+\mathcal O(a)) \exp(-s)\big)
    \end{equation}
    and, along this orbit
    \begin{align*}
    t(\tau,s)=-&\int_0^s (q(\sigma,\tau)+p(\sigma,\tau))^{-3}\mathrm{d}\sigma\\
    =&-\tau^{-3}(1+\mathcal{O}(a))\int_0^s (\exp(\sigma)+\exp(-\sigma))^{-3}\mathrm{d}\sigma\\
    = &-\tau^{-3} \exp(3s/2)(1+\mathcal{O}(a))\int_{-s/2}^{s/2}(\exp(u)+\exp(u))^{-3}\mathrm{d}u\\
    =&-\frac{\pi\tau^{-3}}{16}\exp(3s/2)\left(1+\mathcal{O}(a)+\mathcal{O}(\exp(-3s/2))\right),
    \end{align*}
    where, in the last equality we have made use of the asymptotics \eqref{eq:asymptotictimeintegral}, \eqref{eq:asymptotictimeintegral2}. Thus, for any $t<0$,
    \[
    \exp(-s(\tau,t))= \frac{1}{\tau^2}\left(\frac{\pi}{16 |t|}\right)^{2/3}\left(1+\mathcal{O}(a)+\mathcal{O}(|t|^{-1})\right).
    \]
   Together with expression \eqref{eq:auxiliaryrescaledflow}, this implies that, for a fixed value of $q\in[a/2,2a]$ and sufficiently large $T>0$, there exists a unique $\tau(q,T)$ such that
   \[
q=\pi_q \left(\phi^{s(\tau(q,T),-T)}(\tau(q,T),\tau(q,T))\right).
   \]
   Then, 
   \[
   \phi^{s(\tau(q,T),-T)}(\tau(q,T),\tau(q,T))=(q,\tau(q,\tau)\exp(-s(\tau(q,\tau),-T)))
   \]
    and our claim follows with 
    \[
    h(q)=\pi_p \left(\phi^{s(\tau(q,T),-T)}(\tau(q,T),\tau(q,T)) \right).
    \]
    A completely analogous argument shows that there exists $\tilde h$ such that 
    \[
     \gamma_T^+=\{(\tilde h(p),p),p\in[a/2,2a]\}.
    \]
    Now we show that indeed  $h(q)=y_T(q,q)$ and $\tilde h(p)=x_T(p,p)$, with $x_T$ and $y_T$ as in~\eqref{eq:formShilnikovmaps}. By definition, for any $q\in[a/2,2a]$,
    \[
    \Phi_{loc}^T(q,h(q))=(\tilde h(q),q).
    \]
    As $\Phi^T_{loc}=\Psi^u_T\circ(\Psi^s_T)^{-1}$, this is the same as
    \[
    (\Psi_T^u)^{-1}(\tilde h(q),q)= (\Psi_T^s)^{-1}(q,h(q)).
    \]
    Making use of the notation in \eqref{eq:formShilnikovmaps}, the equality above reads
    \[
    (x_T^{-1}(\tilde h(q),q),q)=(q,y_T^{-1}(q,h(q))),
    \]
    which implies that 
    \begin{equation}\label{eq:defnht}
        q=y^{-1}_T(q,h(q))\qquad\qquad q=x_T^{-1}(\tilde h(q),q).
    \end{equation}
    Using that $y_T^{-1}$ is defined implicitly by 
    \[
    y_T(q,y_T^{-1}(q,p))=p,
    \]
    we deduce that 
    \[
    y_T(q,y_T^{-1}(q,h(q)))=h(q).
    \]
    Analogously 
    \[
    x_T(x_T^{-1}(\tilde h(q),q),q)=\tilde h(q).
    \]
    Then, the lemma is proved using \eqref{eq:defnht}.
\end{proof}

The explicit parametrization of $\gamma_T$ given in Lemma~\ref{lem:diagonal} is very convenient to prove the existence of secondary homoclinic tangencies, as we show in the following result.
\begin{lem}\label{lem:exist_secondary_tangency}
Let $n\in\mathbb{N}$ be sufficiently large and define, for $\mu>0$ small enough, the function 
\[
\delta(q,\mu;n)=M_{\mu}(q)-y_n(q,q)
\]
where $M_\mu(q)$ and $y_n(q,q)$ are the functions in Lemmas~\ref{lem:transversesplitting}~and~\ref{lem:diagonal}. Then, for any $n\in\mathbb N$ sufficiently large there exist a constant $\tau_n$ such that $\tau_n\to 1$ as $n\to\infty$, a point  $q_n\in[3a/4,3a/2]$ and $\mu_n$ of the form
\[
\mu_n=2 \left(\frac{\pi}{16 n}\right)^{2/3} (\tau_n+\mathcal{O}(a)),
\]
such that, locally around $q_n$,
\begin{equation}\label{eq:expansiondeltan}
\delta(q,\mu;n)=\alpha_n(\mu-\mu_n)+ \beta_n(q-q_n)^{2}+\mathcal{O}_2(\mu-\mu_n)+\mathcal O(n^{1/3})\mathcal{O}_3(q-q_n),
\end{equation}
for certain explicit constants $\alpha_n,\beta_n$ of the form
\begin{equation}\label{eq:constantsalphabetan}
\alpha_n=\frac{1}{2 q_n}(\tau_n+\mathcal{O}(a^2)),\qquad\qquad \beta_n=\frac{\mu_n}{q_n^9} 1152(\tau_n+\mathcal{O}(a^3)).
\end{equation}

\begin{rem}\label{lem:shrinkingdomains}
  For a fixed $n\in\mathbb{N}$ the expression $\eqref{eq:expansiondeltan}$ only guarantees that the cubic terms are small provided $|q-q_n|\ll n^{-1}$. This is however not an issue for our purposes since in the renormalization scheme developed in Section~\ref{sec:renormalization} we first fix $n\in\mathbb N$ and then select a suitable small domain.
\end{rem}
 
\end{lem}
\begin{proof}
    We look for zeros of the map $F_n:[a/2,2a]\times(0,1/2]\to\mathbb{R}^2$ where $F_n(q,\mu)=(\delta(q,\mu;n),\partial_q \delta(q,\mu;n))$.
    Making use of the asymptotic expressions \eqref{eq:formShilnikovmaps} and \eqref{eq:Melnikovinq}, we can rewrite the system of equations $F_n(q,\mu)=0$ as
    \begin{align*}
   \left( \frac{\pi}{16 n}\right)^{2/3}q^{-1}(1+\mathcal{O}(n^{-1}))=& y_n(q,q)=M_\mu(q)=\mu\left(2^{-1} q^{-1} \sigma_J\sin (4^3q^{-3})+\mathcal{O}(a)\right)\\
    -\left( \frac{\pi}{16 n}\right)^{2/3}q^{-2}(1+\mathcal{O}(n^{-1}))=&\partial_q y_n(q,q)=M_\mu'(q)\\
    =&\mu\left(-96 q^{-5} \sigma_J\cos (4^3q^{-3}) +\mathcal{O}(a^{-2})\right).
    \end{align*}
Let  $(q_0,\mu_0)$ be such that 
    \begin{equation}\label{eq:approxsecondtangency}
    \tan (4q_0^{-3})=-192 q_0^{-3},\qquad\qquad \mu_0=2 \left( \frac{\pi}{16 n}\right)^{2/3} \frac{1}{\sin (4q_0^{-3})}.
    \end{equation}
    The first equation can be written as $\tan x=-48 x$ for $x=4 q_0^{-3}$. Then,  since $\tan x=-48x$ has at least one solution  at each interval of the form $[k\pi,(k+1/2)\pi]$ with $k\in 2\mathbb N+1$ (in particular it has solutions with arbitrarily large $x$), for any $a$ sufficiently small, we can always find solutions of the first equation in \eqref{eq:approxsecondtangency} with $q_0\in [3a/4,3a/2]$. Moreover, it is not difficult to see that, for solutions of the first equation with $q_0=(4/x)^{1/3}$  and $x\in[k\pi,(k+1/2)\pi]$, $k\in2\mathbb N+1$, we must have $\sin (4q_0^{-3})\to 1$ as $k\to\infty$ so $\mu_0$ is well defined and $\mu_0\sim n^{-2/3}$. One can also check that $\cos(4q_0^{-3})=-192^{-1} q_0^{3}\sin (4q_0^{-3})\sim a^{3}.$

    It is then a tedious but straightforward computation to check that, for $(q_0,\mu_0)$ as above,
    \[
\delta_n(q_0,\mu_0)=\mathcal{O}( a^{-1}n^{-5/3})+\mathcal{O}(\mu_0 a)=\mathcal O(\mu_0 a)=\mathcal {O}(a n^{-2/3})
    \]
    and
    \[
    \partial_q \delta_n(q_0,\mu_0)=\mathcal{O}( a^{-2}n^{-5/3})+\mathcal{O}(\mu_0 a^{-2})={O}(\mu_0 a^{-2})=\mathcal O(a^{-2}n^{-2/3}).
    \]
    We now claim that, for any $(q,\mu)$ such that 
    \[
    |q-q_0|\leq a^{5},\qquad\qquad |\mu-\mu_0|\leq \mu_0 a,
    \]
    we have that $DF(q,\mu)$ is an invertible matrix with 
    \begin{equation}\label{eq:invertibleFtang}
    (DF(q,\mu))^{-1}=\left(\begin{array}{cc}
      \mathcal{O}(\mu_0^{-1} a^8)   & - (1152\mu f(q))^{-1}q^{9}  \\
        -2q (f(q))^{-1} & \mathcal{O}(a^8)
    \end{array}\right),\quad f(q)=\sigma_J\sin(4^3q^{-3}).
    \end{equation}
    Indeed, \eqref{eq:invertibleFtang} follows from the estimates
    \begin{equation}\label{eq:auxiliaryestimatessecondtg}
    \begin{split}
\partial_q\delta_n(q,\mu)=&\mathcal{O}(\mu_0a^{-2}),\\
        \partial_\mu\delta_n(q,\mu)=&(2q)^{-1}\sigma_J\sin(4^3q^{-3})+\mathcal{O}(a),\\
        \partial^2_{q^2} \delta_n(q,\mu)=&1152\mu\ q^{-9}\sigma_J\sin(4^3 q^{-3})+\mathcal{O}(\mu_0 a^{-6}),\\
        \partial^2_{q\mu} \delta_n(q,\mu)=&\mathcal{O}(a^{-2}),
        \end{split}
    \end{equation}
    which are easily obtained using the asymptotic formulas and estimates in Proposition~\ref{prop:Shilnikov} and Lemma~\ref{lem:transversesplitting}. Thus, $(DF(q,\mu))^{-1} F_n(q_0,\mu_0)=(\mathcal{O}(a^7),\mathcal{O}(\mu_0a^2))$ and, modulo verifying \eqref{eq:invertibleFtang}, it follows from the implicit function theorem that there exists 
    \[
    q_n=q_0+\mathcal{O}(a^7),\qquad\qquad \mu_n=\mu_0(1+\mathcal{O}(a^2))
    \]
    such that $F(q_n,\mu_n)=0$. Then, we can Taylor expand
    \[
\delta_n(q,\mu)=\partial_\mu \delta_n(q_n,\mu_n) (\mu-\mu_n)+\frac{1}{2}\partial^2_{q^2} \delta_n(q_n,\mu_n) (q-q_n)^2+E_n(q,\mu),
    \]
    with 
    \[
    E_n(q,\mu)=\int_{0}^1 \partial^2_{\mu^2}(q_n,\mu_n+s(\mu-\mu_n))\mathrm{d}s+\int_{0}^1 \partial^3_{q^3}(q_n+s(q-q_n),\mu)\mathrm{d}s.
    \]
    Finally, the estimates in Proposition~\ref{prop:Shilnikov} and Lemma~\ref{lem:transversesplitting} imply that 
    \[
    E_n=\mathcal{O}_2(\mu-\mu_n)+\mathcal{O}(n^{1/3})\mathcal{O}_3( |q-q_n|)
    \]
    and the expressions for $\alpha_n,\beta_n$ in \eqref{eq:constantsalphabetan} follow from \eqref{eq:auxiliaryestimatessecondtg}.
    \end{proof}
In Lemma~\ref{lem:exist_secondary_tangency} we have proved that, for sufficiently large $n\in\mathbb{N}$, there exists a value $\mu_n>0$ for which $W^u(O)$ and the diagonal $\Delta$ have a quadratic homoclinic tangency which unfolds generically with $\mu$. If the map $\Phi_{loc}$ is symmetric with respect to the involution $R(x,y)=(y,x)$ in the sense that $\Phi_{loc}\circ R=R\circ\Phi_{loc}^{-1}$, this would imply that for $\mu=\mu_n$ there exists a quadratic homoclinic tangency between $W^{u}(O)$ and $W^s(O)$ which unfolds generically. However, a priori, the map $\Phi_{loc}$ is only approximately symmetric since in Proposition~\ref{prop:normalform} we have not studied the behavior of the remainder under the involution $R$. Still, it follows from Lemma~\ref{lem:diagonal}, the asymptotic formulas for the Shilnikov maps given in Proposition~\ref{prop:Shilnikov} and  the expression \eqref{eq:expansiondeltan} that for values of $\mu$ sufficiently close to $\mu_n$, there exists a quadratic homoclinic tangency between $W^{u}(O)$ and $W^s(O)$ which unfolds generically. We abuse notation and denote by $\{\mu_n\}_{n\in\mathbb N}$ the corresponding sequence.

This implies the existence of two points, $(a_n,0)\in W^{s}_{loc}(O)$ and $(0, \tilde a_n)\in W^{u}_{loc}(O)$, such that 
\begin{itemize}
    \item There exists a (unique) number $L(n)\in\mathbb{N}$ for which $ \Psi^{L(n)}_{\mathcal K_{\mu_n}}(0,\tilde a_n)=(a_n,0)$,
    where we recall  $\Psi_{\mathcal K_{\mu}}$ stands for the Poincar\'{e} map on the section $\{\phi=0\}$ associated to the flow of the Hamiltonian $\mathcal K$ in Proposition~\ref{prop:normalform}.
    \item For any neighborhood $R_n$ of $(a_n,0)$, there exists a small connected arc $\gamma_{R_n}\subset W_{loc}^u(O)$ which passes through $(0,\tilde a_n)$, is such that $\Psi^{L(n)}_{\mathcal K_{\mu_n}}(\gamma_{R_n})\subset R_n$ and, moreover, has a quadratic tangency with $W_{loc}^s(O)$ at $(a_n,0)$.
    \item The quadratic tangency between $\Psi^{L(n)}_{K_{\mu_n}}(\gamma_{R_n})$ and $W^s_{loc}(O)$ unfolds generically in $\mu$.
\end{itemize}

In the next lemma, we use the expansion \eqref{eq:asymptoticexpressionsstar} to describe the local  structure of the map
\begin{equation}\label{eq:defnglobalmap}
\Psi_{glob,n,\varepsilon}:=\Psi^{L(n)}_{\mathcal K_{\mu_n+\varepsilon}}
\end{equation}
from a neighborhood of $(0,\tilde a_n)$ to a neighbourhood of $(a_n,0)$. More concretely we define the complex balls
\begin{equation}\label{eq:ballsglobalmap}
\begin{split}
R^u_{n,\rho}=&\{z=(q,p)\in \mathbb{C}^2\colon |z-(0,\tilde a_n)|\leq \rho\},\\
R^s_{n,\rho}=&\{z=(q,p)\in \mathbb{C}^2\colon |z-(a_n,0)|\leq \rho\}.
\end{split}
\end{equation}
\begin{lem}\label{lem:globalmap}
        Fix $n\in\mathbb{N}$ sufficiently large, let $(a_n,0)\in W^{s}_{loc}$ and $(0,\tilde a_n)\in W^{u}_{loc}$ be the homoclinic points constructed above and let $\Psi_{glob,n,\varepsilon}$ be the map in \eqref{eq:defnglobalmap}. Then, for $\rho>0$ small enough,  the map $\Psi_{glob,n,\varepsilon}:R^{u}_{n,\rho}\to R^s_{n,2\rho}$ is well defined, real-analytic and the dependence on $\varepsilon$ is real-analytic. Moreover, uniformly for $(x,\tilde a_n+y)\in R^{u}_{n,\rho}$, and $\varepsilon$ in a sufficiently small complex disk centered at zero
        \[
    \Psi_{glob,n,\varepsilon}:(x,\tilde a_n+y)\mapsto (a_n+b_nx+c_ny+X_n(x,y;\varepsilon), \varepsilon+ d_n x-y^2+Y_n(x,y;\varepsilon)),  
        \]
        for some constants $b_n,c_n$ and $d_n>0$ and 
        \[
        X_n(x,y;\varepsilon)=\mathcal O_2(\varepsilon)+\mathcal{O}_2(x,y),\quad\quad Y_n(x,y;\varepsilon)=\mathcal O_2(\varepsilon)+\mathcal{O}(x(x+y))+\mathcal{O}_3(x,y).
        \]
    \end{lem}

We provide estimates in complex domains in order to make use of Cauchy estimates to obtain bounds for certain derivatives in the next section. The complex domains in which Lemma~\ref{lem:globalmap} hold may shrink to zero with $N\in\mathbb N$. As already pointed out in Remark~\ref{lem:shrinkingdomains}, this is not an issue for the renormalization scheme in the next section.

\section{Renormalization in a neighborhood of the homoclinic tangency}\label{sec:returnmap}
In this section we prove Proposition~\ref{prop:mainresultrenormalization}, which ensures that a suitably renormalized iteration of the return map to a neighborhood of a homoclinic tangency converges, in the real-analytic topology, to the H\'enon map. The proof is divided in several steps. 

Let $\Psi^{u,s}_T:\Sigma_\delta\to\mathbb{C}^2$ be the unstable and stable Shilnikov maps constructed in Proposition~\ref{prop:Shilnikov}. Let $\mu\in(0,1/2]$ be such that the map $\Psi_{\mathcal K_\mu}$ has a quadratic homoclinic tangency (to the degenerate saddle $O$) which unfolds generically with $\mu$. Let  $(0,\tilde a)\in W_{loc}^u$, $(a,0)\in W_{loc}^s$  be points of homoclinic tangency located, respectively in the local unstable and stable manifolds. Then, there exists $L\in\mathbb N$ such that $\Psi^{L}_{\mathcal K_\mu}(0,\tilde a)=(a,0)$. We denote by 
$\Psi_{\mathrm{glob},\varepsilon}:=\Psi_{\mathcal K_{\mu+\varepsilon}}^L$  the ``global map'' from a neighbourhood of $(0,\tilde a)$ to $(a,0)$.
Then, for $T$ large enough,  the map $\Psi_\varepsilon:\mathcal B_\delta\to\mathbb{C}^2$ given by 
\begin{equation}\label{eq:Psirenorm}
\Psi_\varepsilon(\xi,\eta)=(\Psi_T^s)^{-1}\circ \Psi_{glob,\varepsilon}\circ\Psi_T^u(a+\xi,\tilde a +\eta)
\end{equation}
is real-analytic on 
\[
\mathcal B_{\delta}=\{z\in\mathbb C^2\colon \mathrm{dist}(z,[-\sigma,\sigma]^2)\leq \delta\},
\]
for some constant $\sigma$ depending only on a (in the definition of $\Sigma$ in \eqref{eq:dfnsigmadelta}) and $\mu$. We recall that these are small but fixed constants.
In Lemma~\ref{lem:renormfirststep} we prove that, for any $T\in\mathbb N$ large enough, there exists a certain value $\varepsilon_T$ for which there exists a fixed point for the map $\Psi_{\varepsilon_T}$ and, moreover, the image under the map $\Psi_{\varepsilon_T}$ of a vertical segment passing through this fixed point  is a parabolic arc which has a tangency with a horizontal segment. Then, we will see in Lemma~\ref{lem:rescalingtoHenon} that, locally around this point, a rescaling of the map $\Psi_{\varepsilon_T}$ is close to the (critical) H\'{e}non map. The following notation will be used. We write  
\[
\tilde x_T(\xi,\eta)=x_T( a+\xi,\tilde a+\eta)\qquad\qquad\text{and}\qquad\qquad\tilde y_T(\xi,\eta)=y_T( a+\xi, \tilde a+\eta).
\]

\begin{lem}\label{lem:renormfirststep}
    For any $T\in\mathbb N$ large enough, there exist real  $\xi_T,\eta_T,\varepsilon_T=\mathcal{O}(T^{-2/3})$ such that $\Psi_{\varepsilon_T}(\xi_T,\eta_T)=(\xi_T,\eta_T)$ and the $\eta$-coordinate of $\partial_\eta\Psi_{\varepsilon_T}(\xi_T,\eta_T)$ equals zero.
\end{lem}

\begin{proof}
   
    We look for $\xi_*,\eta_*,\varepsilon_*$ as solutions to the system of 3 equations
    \begin{equation}
        \begin{split}\label{eq:systemeqs}
\Psi^s_T( a+\xi_*, \tilde a+\eta_*)=&\Psi_{glob,\varepsilon_*}\circ\Psi^u_T(  a+\xi_*,\tilde a+\eta_*),\\
0=&\Pi_p \left( \left(D\Psi_T^s(a+\xi_*,\tilde a+\eta_*)\right)^{-1} \partial_p (\Psi_{glob,\varepsilon_*}\circ\Psi^u_T )\right)
    \end{split}
    \end{equation}
     On the one hand,  the Shilnikov stable and unstable maps read 
    \[
    \Psi_T^s( a+\xi,\tilde a+\eta)=(a+\xi, \tilde y_T(\xi,\eta)),\qquad\qquad \Psi_T^{u}(a+\xi,\tilde a+\eta)=(\tilde x_T(\xi,\eta),\tilde a+\eta).
    \]
    On the other hand, it follows from Lemma~\ref{lem:globalmap} that the global map satisfies 
    \[
    \Psi_{glob,\varepsilon}:(x,\tilde a+y)\mapsto (a+bx+cy+X(x,y;\varepsilon),\varepsilon+dx-y^2+Y(x,y;\varepsilon)),
    \]
with
\[
X(x,y;\varepsilon)=\mathcal O_2(x,y,\varepsilon),\qquad\qquad Y(x,y;\varepsilon)=\mathcal{O}_2(\varepsilon)+\mathcal{O}_1(x)\mathcal{O}_1(x+y)+\mathcal{O}_3(x,y),
\]
  uniformly in a sufficiently small complex ball around $(0,\tilde a)$ and, also uniformly, in a sufficiently small complex ball around $\varepsilon=0$ (both balls being independent of $T)$.  In particular, if we write 
\begin{align*}
\mathcal{Q}(\xi,\eta;\varepsilon)=&\pi_\xi \left( \Psi_{glob,\varepsilon}\circ\Psi_T^u\right)(a+\xi,\tilde a+\eta)- a\\\mathcal{P}(\xi,\eta;\varepsilon)=&\pi_\eta \left( \Psi_{glob,\varepsilon}\circ\Psi_T^u\right)(a+\xi,\tilde a+\eta),
\end{align*}
then
    \begin{align*}
  \mathcal{Q}(\xi,\eta;\varepsilon)=&b\tilde x_T(\xi,\eta)+c \eta+X(\tilde x_T(\xi,\eta),\eta;\varepsilon)\\
\mathcal P(\xi,\eta;\varepsilon)=&\varepsilon+d\tilde x_T(\xi,\eta)-\eta^2+Y(\tilde x_T(\xi,\eta),\eta;\varepsilon).
    \end{align*}
    We now set $\varepsilon=\varepsilon_0+\tilde \varepsilon$ with $\varepsilon_0+d\tilde x_T(0,0)=\tilde y_T(0,0)$ so, for this choice of $\varepsilon$,
    \[
    \mathcal{P}(\xi,\eta;\varepsilon_0+\tilde\varepsilon)= y_T(0,0)+\tilde\varepsilon+d(\tilde x_T(\xi,\eta)-\tilde x_T(0,0))-\eta^2+Y(\tilde x_T(\xi,\eta),\eta;\varepsilon).
    \]
    Thus, \eqref{eq:systemeqs} is equivalent to the zeros of the map $F:\Sigma\times \mathbb R\to\mathbb{R}^3$ with 
    \begin{align*}
 F_1(\xi,\eta,\tilde\varepsilon)=& \xi-\mathcal Q(\xi,\eta;\varepsilon_0+\tilde\varepsilon)\\
     F_2(\xi,\eta,\tilde\varepsilon)=&\tilde y_T(\xi,\eta)-\mathcal{P}(\xi,\eta;\varepsilon_0+\tilde\varepsilon)\\
    F_3(\xi,\eta,\tilde\varepsilon)=&-\partial_\xi\tilde y_T (\xi,\eta) \partial_\eta \mathcal Q(\xi,\eta;\varepsilon_0+\tilde\varepsilon)+\partial_\eta\mathcal{P}(\xi,\eta;\varepsilon_0+\tilde\varepsilon).
    \end{align*}
    We now notice that, for 
    \[
    \xi_0=b\tilde x_T(0,0)-c\eta_0,\qquad \eta_0=-\frac{1}{2}\left( \partial_\xi \tilde y_T (0,0) c+d \partial_\eta \tilde x_T(0,0)\right)
    \]
    and \[
    \tilde \varepsilon_0=\tilde y_T(\xi_0,\eta_0)-\tilde y_T(0,0)-d\left(\tilde x_T(\xi_0,\eta_0)-\tilde x_T(0,0) \right)+\eta_0^2-Y(\tilde x_T(\xi_0,\eta_0),\eta_0;0),
    \]
we have  
\begin{equation}\label{eq:estimateaproxsolF}
F(\xi_0,\eta_0;\tilde\varepsilon_0)=\mathcal{O}(T^{-4/3}).
\end{equation}
Indeed, uniformly for $(\xi,\eta)\in \Sigma_\delta$, we have 
\[
|\partial_\eta \tilde x_T|,|\partial_\xi\tilde y_T|=\mathcal{O}(T^{-2/3}),\qquad\qquad |\partial_\xi \tilde x_T|,|\partial_\eta\tilde y_T|=\mathcal{O}(T^{-5/3}),
\]
so \eqref{eq:estimateaproxsolF} follows by direct application of the mean value theorem and the bound
\[
|\xi_0|,|\eta_0|=\mathcal{O}(T^{-2/3}).
\]
 On the other hand,  it follows from the definition of $F$ and the estimates in   Proposition~\ref{prop:Shilnikov}  that, uniformly for $(\xi,\eta)\in\Sigma$ and $\tilde\varepsilon$ sufficiently small, $DF(\xi,\eta,\tilde\varepsilon)=A+\mathcal{O}(T^{-2/3})$ with 
\[
A=\begin{pmatrix}
    1 &-c&0\\
    0&0&1\\
    0&-2&0\\
\end{pmatrix}
\]
   and the proof of the lemma is completed making use of the implicit function theorem.
\end{proof}

We now define a suitable rescaling of the map $\Psi_\varepsilon$ in \eqref{eq:Psirenorm} around the point $(\xi_*,\eta_*)$ for values of $\varepsilon$ close to $\varepsilon_*$. More concretely, we prove the following result.

\begin{lem}\label{lem:rescalingtoHenon}
Denote by $B_{10}\subset \mathbb{C}^2$ the complex ball  around the origin of radius $10$. Let $h:(Q,P)\mapsto (\xi,\eta)$ be the change of variables given by
\[
Q=\alpha^{-1} (\xi-\xi_*),\qquad P=\alpha^{-1}(\eta-\eta_*),\qquad\text{for}\quad \alpha=\partial_\eta \tilde y_T(\xi_*,\eta_*),
\]
and let $\varepsilon=\varepsilon_*+ \alpha^2 \kappa$. Then, uniformly for $(Q,P)\in B_{10}\subset \mathbb{C}^2$ and $\kappa\in[-2,2]$ we have
\[
h^{-1}\circ \Psi_{\varepsilon(\kappa)}\circ h (Q,P)=\left( cP,\kappa +d \gamma Q-P^2\right)+\mathcal{O}(T^{-2/3}), \qquad\text{for}\quad \gamma=\frac{\partial_\xi \tilde x_T(\xi_*,\eta_*)}{\partial_\eta \tilde y_T(\xi_*,\eta_*)}.
\]
\end{lem}

\begin{proof}
    We want to show that there exist $U,V=\mathcal{O}(T^{-2/3})$  uniformly for $(Q,P)\in B_{10}$, such that, on the same domain,
    \begin{equation}\label{eq:renormcondproof}
   \Psi_{\varepsilon,glob}\circ \Psi^u_T(\xi_*+\alpha Q,\eta_*+\alpha P)= \Psi _T^s\circ h (cP+U(
Q,P), \kappa+d\gamma Q-P^2+V(Q,P)).
    \end{equation}
In order to do so, we will show that 
\begin{align*}
\widetilde U(Q,P)=&\pi_q\left( \Psi_{\varepsilon,glob}\circ \Psi^u_T(\xi_*+\alpha Q,\eta_*+\alpha P)- \Psi _T^s\circ h (cP, \kappa+d\gamma Q-P^2)\right)\\
\widetilde V(Q,P)=&\pi_p\left( \Psi_{\varepsilon,glob}\circ \Psi^u_T(\xi_*+\alpha Q,\eta_*+\alpha P)- \Psi _T^s\circ h (cP, \kappa+d\gamma Q-P^2)\right)
\end{align*}
are such that the equation
\begin{equation}\label{eq:inversefinalsteprenorm}
\begin{split}
\Psi^s_T\circ h (cP+U(
Q,P), \kappa+d\gamma Q-P^2+V(Q,P))=&\Psi _T^s\circ h (cP, \kappa+d\gamma Q-P^2)\\&+(\widetilde U(Q,P), \widetilde V(Q,P))
\end{split}
\end{equation}
admits a unique solution $U(Q,P),V(Q,P)$ with the desired estimates for any $(Q,P)\in B_{10}$. On the one hand,
    \begin{align*}
    h(cP, \ \kappa+d\gamma Q-P^2)=(\xi_*,\eta_*)
+\alpha (cP, \ \kappa+d\gamma Q-P^2),
    \end{align*}
    so, writing $z=\alpha (cP, \ \kappa+d\gamma Q-P^2)$,
    \[
    \Psi _T^s\circ h (cP, \kappa+d\gamma Q-P^2)=(\xi_*, \tilde y_T(\xi_*,\eta_*))+ D\Psi^s_T (\xi_*, \eta_*)\  z + R(Q,P),
    \]
    with 
    \[
    R=\big( 0,\  \int _0^1 (1-t) \langle D^2 y_T(\xi_*+tz, \eta_*+tz) z,z\rangle\ \mathrm{d}t\big).
    \]
    We introduce the notation (for $\star=\xi,\eta$)
    \[
    x_*=\tilde x_T(\xi_*,\eta_*),\qquad y_*= \tilde y_T(\xi_*,\eta_*),\qquad \partial_\star x_*=\partial_\star \tilde x_T(\xi_*,\eta_*),\qquad \partial_\star y_*=\partial_\star \tilde y_T(\xi_*,\eta_*).
    \]
    From the definition of $\Psi^s_T$ we obtain that 
  \[
    D\Psi^s_T (\xi_*, \eta_*)\  z= \alpha\left( c P,\ c \partial_\xi y_* P+ \partial_\eta y_* (\kappa+d \gamma Q-P^2)\right),
   \]
    and, uniformly for $(\xi,\eta)\in\Sigma_\delta$ (see Proposition~\ref{prop:Shilnikov})
    \[
 \partial_{\xi^2}^2 \tilde y_T(\xi,\eta), \partial_{\xi\eta}^2 \tilde y_T(\xi,\eta),\partial_{\eta^2}^2 \tilde y_T(\xi,\eta)=\mathcal{O}(T^{-2/3}),
    \]
    so $R=(0,\mathcal{O}(\alpha^2 T^{-2/3}))$. Therefore, it is clear that, for $(Q,P)\in B_{10}\subset\mathbb{C}^2$,
\begin{multline}
\label{eq:leftrenorm}
    \Psi _T^s\circ h (cP, \kappa+d\gamma Q-P^2) \\ =(\xi_*, \tilde y_T(\xi_*,\eta_*))
+\alpha\left( c P,\ c \partial_\xi y_* P+ \partial_\eta y_* (\alpha^{-2} (\varepsilon-\varepsilon_*)+d \gamma Q-P^2)+\mathcal{O}(\alpha T^{-2/3})\right).
\end{multline}
    On the other hand, for $\tilde \Psi_\varepsilon(Q,P)=\Psi_{\varepsilon,glob}\circ\Psi^u_T(\xi_*+\alpha Q,\eta_*+\alpha P)$, and writing $\tilde z(Q,P)=(\alpha Q, \alpha P)$, Taylor's integral formula yields
    \[
    \tilde \Psi_\varepsilon (Q,P)=\tilde \Psi_\varepsilon (\xi_*,\eta_*)+D\tilde\Psi_\varepsilon \tilde z(Q,P) + \big(0,\frac{1}{2} \partial_{\eta^2}^2 (\pi_\eta (\tilde \Psi_{\varepsilon})) \alpha^2P^2\big) + S(Q,P)
    \]
    for (using the definition of $\tilde\Psi_\varepsilon$ and the estimates in Proposition~\ref{prop:Shilnikov})
    \[
    S= \alpha^2\left(\mathcal{O}( T^{-2/3}),\mathcal{O}(T^{-2/3})\right).
    \]
    Then, it follows from the definition of $(\xi_*,\eta_*)$ and the expression 
    \[
    \tilde \Psi_\varepsilon(\xi,\eta)=(b \tilde x_T(\xi,\eta)+c\eta+X(\tilde x_T(\xi,\eta),\eta), \varepsilon+ d\tilde x_T(\xi,\eta)-\eta^2+Y(\tilde y_T(\xi,\eta),\eta))
    \]
    that, uniformly for $(Q,P)\in B_{10}$,
    \begin{multline}
    \label{eq:rightrenorm}
    \tilde \Psi_\varepsilon (\xi_*+\alpha Q,\eta_*+\alpha P)=(\xi_*, \tilde y_T(\xi_*,\eta_*)) \\
     + \left( \partial_\eta \mathcal Q(\xi_*,\eta_*)\alpha P,\ \alpha^2\hat\varepsilon+\partial_\eta\mathcal{P}(\xi_*,\eta_*)\alpha P- \alpha ^2 P^2\right)  +\mathcal{O}(\alpha^2 T^{-2/3}).
    \end{multline}
    Comparing \eqref{eq:leftrenorm} and \eqref{eq:rightrenorm} and making use of the definition of $(\xi_*,\eta_*)$, one easily obtains that 
\begin{align*}
\widetilde U(Q,P)=&\alpha b\partial_\eta x_*+\mathcal{O}(\alpha^2 T^{-2/3}),\\
\widetilde V(Q,P)=&\alpha b\partial_\xi y_*\partial_\eta x_*+\mathcal{O}(\alpha^2 T^{-2/3}).
\end{align*}
It remains to study the equation \eqref{eq:inversefinalsteprenorm}. To that end, it will be convenient to introduce some notation:
\begin{itemize}
    \item We denote by $\mathtt H:\mathbb C^2\to \mathbb{C}^2$ the H\'{e}non map
\[
\mathtt{H}:(Q,P)\mapsto (cP,\kappa+d\gamma Q-P^2),
\]
\item We denote by $W(Q,P)=(U(Q,P),V(Q,P))$ and by $\widetilde W(Q,P)=(\widetilde U(Q,P),\widetilde V(Q,P))$,
\item We denote by 
\[
A=D(\Psi^s_T\circ h)(0,0)=\left(\begin{array}{cc}
   1  &0  \\
   \partial_\xi y_*  &\alpha 
\end{array}\right)
\]
\item and we denote by 
\begin{align*}
F&(Q,P,W(Q,P))=\left(D(\Psi^s_T\circ h)(H(Q,P)) -D(\Psi^s_T\circ h)(0,0)\right) W(Q,P)\\
+&\int_0^1(1-t)\langle D^2(\Psi^s_T\circ h)((t\mathtt H+(1-t)W)(Q,P)) W(Q,P), W(Q,P) \rangle \mathrm{d}t.
\end{align*}
\end{itemize}
Then, we expand the left hand side  of \eqref{eq:inversefinalsteprenorm} as
\begin{align*}
\Psi^{s}_T\circ h\circ( \mathtt{H}+W)=&\Psi^{s}_T\circ h\circ \mathtt{H}(Q,P)+ A W(Q,P)+F(Q,P,W(Q,P)).
\end{align*}
The estimates in Proposition~\ref{prop:Shilnikov} imply that, uniformly for $(Q,P)\in B_{10}$,
\[
F(Q,P,W(Q,P))=\mathcal{O}(T^{-2/3}\alpha|W(Q,P)|)+\mathcal{O}(T^{-2/3}|W(Q,P)|^2).
\]
The right hand side of \eqref{eq:inversefinalsteprenorm} is given by 
\[
\Psi^{s}_T\circ h\circ \mathtt{H}(Q,P)+\widetilde W(Q,P).
\]
Hence, for $(Q,P)\in B_{10}$, we rewrite \eqref{eq:inversefinalsteprenorm} as the fixed point equation for $W(Q,P)$
\[
W(Q,P)= A^{-1} \left( \widetilde W(Q,P)+F(Q,P,W(Q,P))\right).
\]
If we define 
\[
W_0(Q,P)=A^{-1} \widetilde W(Q,P),
\]
it follows that
\[
W_0(Q,P)=(\alpha b\partial_\eta x_*,0)+\big(\mathcal{O}(\alpha^2 T^{-2/3}),\mathcal O(\alpha T^{-2/3})\big)=\mathcal O(\alpha T^{-2/3}),
\]
and
\[
E_0(Q,P):= W(Q,P)-W_0(Q,P)=\left(D(\Psi^s_T\circ h)(0,0)\right)^{-1} \mathcal O( \alpha^2 T^{-4/3})=\mathcal O(\alpha T^{-4/3}).
\]
The proof is easily completed making use of the estimate (recall the definition of $F$ above and the estimates in Proposition~\ref{prop:Shilnikov})
\begin{align*}
|A^{-1}&(F(Q,P,W(Q,P))-F(Q,P,W_*(Q,P)))|\\
\lesssim &\alpha^{-1}|F(Q,P,W(Q,P))-F(Q,P,W_*(Q,P))|\\
\lesssim & \alpha^{-1}\bigg( \left(\alpha T^{-2/3}+|W(Q,P)+W_*(Q,P)|T^{1/3} \right)|W(Q,P)-W_*(Q,P)| \bigg)
\end{align*}
and a suitable version of the fixed point theorem.
\end{proof}

Finally, in order to complete the proof of Proposition~\ref{prop:mainresultrenormalization}, we introduce the linear map $\hat{h}:(Q,P)\mapsto (-d\gamma Q, P)$
and consider $h_T= h\circ\hat{h}$. Then, recalling the definition of $\Psi_{\varepsilon}$ in \eqref{eq:Psirenorm}, we obtain that
\begin{equation}\label{eq:asymptoticHenonproof}
h_T^{-1}\circ\Psi^s_T\circ \Psi^{L}_{\mathcal{K}_{\varphi(\kappa)}}\circ\Psi^{u}_T\circ h_T(Q,P)=(-cd\gamma Q,\ \kappa-Q-P^2)+\mathcal{O}(T^{-2/3}),\qquad\qquad 
\end{equation}
with 
\[
\varphi(\kappa)=\mu_n+\varepsilon_*+\alpha^2 \kappa,
\]
and the result follows after noticing that, since $(\xi_*,\eta_*)$ is a fixed point of the map $\Psi_{\varepsilon_*}$ and the map $\Psi_{\varepsilon_*}$ preserves a symplectic form, we must have
\[
1=\mathrm{det} D\Psi_{\varepsilon_*}= -cd\gamma+\mathcal{O}(T^{-2/3}).
\]

\subsection{Proof of Lemma~\ref{lem:lineoftangencies}}
The existence of the line of tangencies is implied by the existence of a solution to the equation
\[
\partial_\eta (\Psi_{glob}\circ\Psi_T^u)\wedge \partial_\eta \Psi_T^s(\xi,\eta)=0.
\]
In the proof of Lemma~\ref{lem:renormfirststep} it is shown that, for values of $\mu$ close to $\mu_*$, this equation admits a solution which is of the form $(\xi,\eta_*(\xi,\mu))$ and moreover, the tangency is quadratic. We then define
\[
\ell_a=\{(q,p)\colon \Psi_T^u(\xi,\eta_*(\xi,\mu)),\ \xi\in(a/2,3a/2)\}\subset V.
\]
The second and third item in Lemma~\ref{lem:lineoftangencies} follow directly from the definition of $\ell_a$. Moreover, since
\[
\partial_\xi \eta_*(\xi,\mu)=\mathcal{O}(T^{-2/3}),
\]
it is clear that $\ell_a$ forms an angle with $\mathcal F_v$ which is bounded by below by $\pi/2-\mathcal O(T^{-2/3})$, so the first item follows for $T$ large enough.

\subsection{Proof of Lemma~\ref{lem:locgeomunstabmanifoldbasicset}}
Let $h_T$ be the affine rescaling introduced after Lemma~\ref{lem:rescalingtoHenon}, in which the map $\mathcal H:=h_T^{-1}\circ\Psi_{\varphi(\kappa),T}\circ h_T$ satisfies the asymptotic expression \eqref{eq:asymptoticHenonproof}. It then follows that $\mathcal H$ has, for $\kappa> -1$, a hyperbolic fixed point 
\[
Z_{\kappa,T}=(-1-\sqrt{1-\kappa},-1-\sqrt{1-\kappa}).
\]
One may check  that, in $(\xi,\eta)$ coordinates, the stable and unstable tangent spaces at $Z_{\kappa,T}$ are of the form
\[
\widetilde E^s=\mathrm{span}\{(1,-\tau)\},\quad\quad \widetilde E^u=\mathrm{span}\{(\tau,1)\},\qquad\tau=\lambda-\sqrt{\lambda^2-1}+\mathcal O(T^{-2/3}),
\]
for $\lambda=-(1+\sqrt{1-\kappa})$,
and that the open branch of the local unstable manifold is the one located below $Z_{\kappa,T}$ (see Figure~\ref{fig:Localmanifolds}). It then follows that at the point 
\[
z_{\kappa,T}=\Psi_T^s\circ h_T(Z_{\kappa,T}),
\]
which is a hyperbolic fixed point for the map $\Psi_{glob}\circ\Psi^T_{loc}$, the splitting of the tangent space is of the form
\[
E^s= D\Psi_T^s(h_T(Z_{\kappa,T})) D h_T \widetilde E^s,\qquad E^u= D\Psi_T^s(h_T(Z_{\kappa,T})) D h_T\widetilde E^u.
\]
Let denote $\alpha(\xi,\eta)=\partial_\xi x_T(\xi,\eta)$, $\beta(\xi,\eta)=\partial_\eta x_T(\xi,\eta)$. Then the conclusion follows since
\[
D\Psi^s_T(h_T(Z_{\kappa,T})) D h_T =\begin{pmatrix}
    \alpha&0\\
    \alpha\beta&\alpha^2
\end{pmatrix},
\]
so $E^u$ is of the form
\[
E^u=\mathrm{span}\{(1,\beta+\tau^{-1}\alpha)\} 
\]
and the tangent vector to $\mathcal F_h$ is given by $(1,\beta)$.

\appendix

\section{Splitting of invariant manifolds in the restricted 3-body problem}\label{sec:splittingr3bp}

In this appendix we provide the details leading to the proof of Lemma~\ref{lem:transversesplitting}. The first step is to construct a good coordinate system to measure the splitting between $W^{u,s}(O)$. To do so, let $(r_h(u),y_h(u)):\mathbb{R}\to \mathbb{R}^2$ be the time parameterization of the homoclinic manifold of the Kepler problem in polar coordinates $(r,y)$ (see \cite{MartinezPinyol}). The following result, proved in \cite{MartinezPinyol}, will be useful later. 
\begin{lem}\label{lem:techlemmaunperturbedhomocl}
   The following asymptotic expressions hold for $u\to\pm\infty$:
   \[
   r_h(u)= \frac{1}{2}(6u)^{2/3}(1+\mathcal{O}(u^{-2/3})),\qquad\qquad y_h(u)= \left(\frac{4}{3u}\right)^{-1/3}(1+\mathcal{O}(u^{-2/3})).
   \]
\end{lem}

We now introduce the change of variables\footnote{In the literature, these coordinates are usually referred to as time-energy coordinates.} (notice that, in view of Lemma~\ref{lem:techlemmaunperturbedhomocl}, this change is well defined for $|u|$ sufficiently large)
\[
(u,E)\mapsto (r,y)=\phi_1(u,E)=(r_h(u), y_h(u)+\frac{1}{y_h(u)}E).
\]
The following lemma was proved in \cite{MartinezPinyol}.
\begin{lem}\label{lem:Melnikov}
 Fix any $0<u_0<u_1<\infty$. Then, for any sufficiently large value $J$ of the Jacobi constant  and for $\mu>0$ small enough (depending on $J)$, there exist real-analytic functions $E^{u,s}_\mu:(u_0,u_1)\to\mathbb{R}$ such that
\begin{equation}
    \begin{split}
        \Gamma^{u,s}:(u_0,u_1)&\mapsto (r_h(u),y_h(u)+ y_h^{-1}(u) E^{u,s}_\mu(u;J))\\
    \end{split}
\end{equation}
 are parameterizations in polar coordinates of compact pieces of the stable and unstable manifolds of $\mathcal O$. Moreover, 
 \begin{equation}\label{eq:splittinginEcoords}
 E^u_\mu(u;\mathcal J)-E^s_\mu(u;\mathcal J)=\mu \left(\sigma_J\sin (\mathcal J^3u) +\mathcal{E}(u;\mathcal J)\right)+\mathcal{O}(\mu^2),
 \end{equation}
 where $\sigma_J>0$ is an explicit constant independent of $\mu$ and $\mathcal{E}(u;\mathcal J)$ satisfies that $\lim_{J\to\infty} \mathcal E(u;\mathcal J)/\sigma_J\to 0$ uniformly in $u\in(u_0,u_1).$
\end{lem}

We now want to obtain parameterizations of the stable and unstable manifolds in the coordinate system $(q,p)$ obtained in Proposition~\ref{prop:normalform}.  As these coordinates are not explicit, we first consider an explicit auxiliary coordinate system $(\tilde q,\tilde p)$  which is close to the coordinate system $(q,p)$. In order to construct the coordinates $(\tilde q,\tilde p)$, we let $\chi(y,t)$ be the function such that, in McGehee coordinates $(x,y)$ 
\[
W^u_{\mathrm{loc}}=\{(\chi(y,t),y),\ 0\leq y\ll 1\},\qquad W^s_{\mathrm{loc}}=\{(\chi(-y,t),y),\ 0\leq y\ll 1\}.
\]
The existence and asymptotic properties of the function $\chi$ were established by McGehee in \cite{McGeheestablemfold} (see also \cite{Moserbook}). Then, we define the change of variables $ \phi_2:(x,y)\mapsto (\tilde q,\tilde p)$ given by
\[
(\tilde q,\tilde p)=(x-\chi(-y,0),x-\chi(y,0)).
\]
One may check that, since $\chi(y,t)=y+\mathcal{O}_3(y)$, for any $a>0$ small enough and $(x,y)\in[0,2a]^2$, we have that 
\[
|\phi_2-\Psi|= \mathcal O_3(a),
\]
where $\Psi$ is the change of variables in Lemma~\ref{lem:McGehee}. Since, for the change of variables  $\Omega$ in Proposition~\ref{prop:normalform}, we have $|\Omega-\mathrm{id}|= \mathcal{O}_3(a)$ uniformly on $(q,p)\in[0,2a]^2$ it also follows that 
\[
|\phi_2-\Psi\circ\Omega|=\mathcal O_3(a),
\]
so the coordinates $(\tilde q,\tilde p)$ are $\mathcal O_3(a)$-close to the coordinates $(q,p)$.

We now construct a parameterization of the invariant manifolds in coordinates $ (\tilde q,\tilde p)$. In order to do so, we consider the change of variables  $\tilde\phi=\phi_2\circ\phi_1:(u,E)\mapsto ( \tilde q,\tilde p)$, which is given by 
\[
(\tilde q,\tilde p)=\left(\sqrt{2/r_h(u)}-\chi(-y_h(u)-y_h(u)^{-1}E,0),  \sqrt{2/r_h(u)}-\chi(y_h(u)+y_h(u)^{-1}E,0)\right).
\]
Notice that,  since by construction $W^s(O)=\{\tilde p=0\}$, we must have that for any $u\in(u_0,u_1)$
\begin{equation}\label{eq:auxiliarylemmasplitting}
\sqrt{2/r_h(u)}-\chi(y_h(u)+y_h(u)^{-1}E^s(u),0)=0.
\end{equation}
Define now the functions $\tilde f_q(u),\tilde f_p(u):(u_0,u_1)\to\mathbb{R}$ given by
\[
 f_q(u)=\sqrt{2/r_h(u)}-\chi(-y_h(u)-y_h(u)^{-1}E^u(u),0)
\]
and 
\[
f_p(u)=\sqrt{2/r_h(u)}-\chi(y_h(u)+y_h(u)^{-1}E^u(u),0).
\]
Then, if we are able to find a function $g_q(\tilde q)$ such that $f_q(g_q(\tilde q))=\tilde q$, we obtain the graph parameterization of the unstable manifold
\[
\widetilde \Gamma:(f_q(u_0)f_q(u_1))\mapsto (\tilde q,f_p(g_q(\tilde q)))
\]
in $(\tilde q,\tilde p)$ coordinates. We now check that we are able to invert the function $f_q$ and provide an asymptotic formula for the parameterization $\widetilde \Gamma$ above. To do so, we notice that, from the definition of the coordinates $(\tilde q,\tilde p)$ and the asymptotic properties of the unperturbed homoclinic orbit given in Lemma~\ref{lem:techlemmaunperturbedhomocl},
\begin{align*}
f_ q(u)=&\sqrt{2/r_h(u)}+y_h(u)+\mathcal{O} (y_h^{3}(u)+|E^{s}(u)|)\\
=& 2y_h(u)(1+\mathcal{O}_2 (y_h(u)))+\mathcal{O}(\mu)\sim 4 u^{-1/3}(1+\mathcal{O}_2 ( u^{-1/3}))+\mathcal{O}(\mu).
\end{align*}
Therefore, given a sufficiently small value $a>0$, we can choose sufficiently large $u_0,u_1$ in Lemma~\ref{lem:Melnikov}  so that there exists a unique function $g_q:(a/2,2a)\to\mathbb R$ satisfying $f_q(g_q(\tilde q))=\tilde q$. A trivial computation shows that, uniformly for $\tilde q\in (a/2,2a)$
\[
g_q(\tilde q)=4^3\tilde q^{-3}+\mathcal{O}(a,\mu).
\]
Finally, we provide an asymptotic formula for the parameterization $\widetilde \Gamma$ above. We notice that, in view of \eqref{eq:auxiliarylemmasplitting} and the asymptotic expression in Lemma~\ref{lem:Melnikov} for $E^{s}-E^u$, we can express, for $u\in (u_0,u_1)$
\begin{align*}
f_p(u)=&\chi(y_h(u)+y_h(u)^{-1}E^s(u))-\chi(y_h(u)+y_h(u)^{-1}E^u(u))\\
=&\frac{1}{y_h(u)} (E^{s}(u)-E^u(u))+\mathcal{O}(\mu y_h(u))\\
=&\frac{1}{y_h(u)}\left(\mu \left(\sigma_J\sin (J^3u) +\mathcal{E}(u;J)\right)+\mathcal{O}(\mu^2)\right)+\mathcal{O}(\mu u^{-1/3}).
\end{align*}
Therefore, using the asymptotic formula above for $g_q(\tilde q)$ we obtain that, uniformly  for $\tilde q\in (a/2,2a)$,
\[
f_p(g_q(\tilde q))=\frac{1}{2q}\left(\mu \left(\sigma_J\sin ((4J q^{-1})^3) +\mathcal{E}(f_q(\tilde q);J)\right)+\mathcal O(\mu a)+\mathcal{O}(\mu^2)\right).
\]
Then, \eqref{eq:Melnikovinq} follows since, as we have seen above, the coordinates $(q,p)$ in Proposition~\ref{prop:normalform} are $\mathcal O(a)$ close to the coordinates $(\tilde q,\tilde p)$.

\section{Splitting of invariant manifolds in the Sitnikov problem}\label{sec:Sitnikov}
In variables $(z,y)\in\mathbb R^2$ representing the $z$-coordinate and $z$-momentum of the massless particle, the dynamics of the Sitnikov problem is given by the Hamiltonian
\begin{equation}
\label{Sitnikov}
H_\epsilon(z,y,t)=\frac{y^2}2-\frac{1}{\sqrt{z^2+\rho_\epsilon^2(t)}},\qquad\epsilon\in[0,1),
\end{equation}
where $\rho_\epsilon(t):\mathbb T\to[1-\epsilon,1]$ stands for the distance of the primaries to the line of motion of the massless body. McGehee's compactification of the phase space $z=2/x^2$ can be used to study the dynamics close to ``infinity'' and conjugates the dynamics of $H_\epsilon$ to the flow of the Hamiltonian 
\[
\mathtt H_\epsilon(x,y,t)=\frac{y^2}{2}-\frac{x^2}{2\sqrt{1+\frac{1}{4}x^4\rho_\epsilon^2(t)}}
\]
in the symplectic manifold $(\mathbb R^2,\frac{4}{x^3}\mathrm{d}x\wedge\mathrm{d}y)$. One may check that $O=(0,0)$ is a parabolic fixed point for the time-one map associated to flow of $\mathtt H_\epsilon$ and that the time-one map is of the form \eqref{eq:mainpoincaremapformula}. Then, the  techniques developed by McGehee in \cite{McGeheestablemfold} show that $O$ has stable and unstable manifolds $W^{s,u}_\epsilon(O)$. 

The proof of Theorem~\ref{thm:Sitnikov} follows after establishing that Proposition~\ref{prop:normalform} also holds for the Hamiltonian $K_\epsilon(q,p,t)=\mathtt H_\epsilon(2^{-1/2}(q+p),2^{-1/2}(q-p),t)$ and that the splitting of separatrices in the Sitnikov problem is of the same form as in the restricted 3-body problem. That this is the case can be easily seen after noticing that 
\[
K_\epsilon(q,p,t)=qp+\mathcal O_6((q+p)
)\]
so the results in Section~\ref{sec:normalform} apply directly. The study of the splitting of invariant manifolds of $O$ is carried out in the next section.

\subsection{Splitting of invariant manifolds}
For $\epsilon=0$, the system is integrable and $W^{s}_0(0)=W^{u}_0(O)$.  Let $(z_h(u),y_h(u)):\mathbb R\to\mathbb R^2$ stand for the time parametrization of this homoclinic manifold. It is not difficult to see that the asymptotic behavior of $z_h$ and $y_h$ coincides with the one given in Lemma~\ref{lem:techlemmaunperturbedhomocl} (replacing $r_h$ by $z_h$). The following result has been obtained in \cite{SitnikovPerezChavela}.

\begin{lem}\label{lem:MelnikovSitnikov}
 Fix any $0<u_0<u_1<\infty$. Then, for any $\epsilon>0$ small enough  there exist real-analytic functions $E^{u,s}_\epsilon:(u_0,u_1)\to\mathbb{R}$ such that
\begin{equation}
    \begin{split}
        \Gamma^{u,s}:(u_0,u_1)&\mapsto (r_h(u),y_h(u)+ y_h^{-1}(u) E^{u,s}_\epsilon(u))\\
    \end{split}
\end{equation}
 are parameterizations in $(z,y)$ coordinates of compact pieces of the stable and unstable manifolds of $O$. Moreover, 
 \begin{equation}\label{eq:splittinginEcoordsSitnikov}
 E^u_\epsilon(u)-E^s_\epsilon(u)=\epsilon \sigma_\epsilon\sin (u) +\mathcal{O}(\epsilon^2),
 \end{equation}
 where $\sigma_\epsilon>0$ is an explicit constant independent of $\epsilon$.
\end{lem}
We notice that \eqref{eq:splittinginEcoordsSitnikov} is of the same nature as the corresponding formula \eqref{eq:splittinginEcoords} for the restricted circular 3-body problem . One then may check that the analogue of Proposition~\ref{lem:homoclinictangencies} also holds for the Sitnikov problem.

\section{Splitting of invariant manifolds in the restricted 4-body problem}\label{sec:splitting4bp}
We now consider the restricted 4-body problem in which the primaries (the three massive bodies) move in a Lagrange triangular periodic orbit. We will denote this model as the restricted planar circular 4-body problem (RPC4BP). We let the masses of the primaries be $m_0=1-2\mu$ and $m_1=m_2=\mu$ for $\mu\in(0,1/3]$. In Cartesian coordinates
\begin{equation}\label{eq:H4bp}
H_{\mu,4BP}(q,p,t)=\frac{|p|^2}{2}-V_\mu(q,t)\qquad\qquad (q,p,t)\in\mathbb R^2\times\mathbb R^2\times\mathbb T
\end{equation}
with 
\[
V_\mu(q,t)=\frac{1-2\mu}{|q-q_{0,\mu}(t)}+\frac{\mu}{|q-q_{1,\mu}(t)|}+\frac{\mu}{|q-q_{2,\mu}(t)|}
\]
With this choice of the masses, for any value of $\mu\in(0,1/2]$ the Hamiltonian $H_{\mu,4BP}$ is invariant by rotations. After McGehee's compactification of the phase space it can be shown that (as for the restricted circular 3-body problem) there exists a degenerate periodic orbit at infinity which possesses stable and unstable invariant manifolds. We denote these by $W^{u,s}_\infty$.

In this section we prove the existence of $\mu_*\in(0,1/2]$ (localized close to  $\mu \sim 1/3$) for which these invariant manifolds possess a quadratic tangency which unfolds generically with $\mu$. To understand the reason why we consider this particular subset of the parameter space, let $H(r,\alpha,y,G,t; \mu)=\frac{y^2}{2}+\frac{G^2}{2r^2}-\frac{1}{r}-U(r,\alpha-t; \mu)$ denote the Hamiltonian of the RPC4BP in polar coordinates. Thanks to our choice of parameters, the interaction potential  $U=U(r,\phi; \mu)$ is even in $\phi$. This convenient choice will simplify considerably the investigation of existence of quadratic homoclinic tangencies.

In order to analyze the aforementioned invariant manifolds for values of $\mu\sim 1/3$, we place ourselves in the singular perturbative regime corresponding to large values of the Jacobi constant $\mathcal J= H(r,\alpha,y,G,t;\mu)-G$. A completely analogous situation was already tackled in \cite{GMS2016} for studying the existence of oscillatory motions in the restricted circular 3-body problem with arbitrary mass ratio. The idea is that if we rescale the variables by
\[
r=G_0^2 \tilde r, \quad y=G_0^{-1} \tilde y, \quad \alpha=\tilde \alpha, \quad G=G_0 \tilde G, \quad t=G_0^3 v,
\]
then the Hamiltonian becomes
\[
\tilde H (\tilde r, \tilde \alpha, \tilde y, \tilde G, v; \mu, G_0) = \frac{\tilde y^2}{2} + \frac{\tilde G^2}{2 \tilde r^2} - \frac{1}{\tilde r} - V(\tilde r, \tilde \alpha - G_0^3 v; \mu, G_0) ,
\]
where $V(\tilde r, \phi; \mu, G_0) = G_0^2 U(G_0^2 \tilde r, \phi; \mu)$ satisfies that 
\[
V(\tilde r,\phi;\mu,G_0)=\mathcal O(G_0^{-4}\tilde r^{3}).
\]
Then, for large values of $G_0$, we can see $\tilde H$ as a fast-periodic perturbation of the 2-body problem and the manifolds $W^{u,s}_\infty$ can be studied as a perturbation of the manifold of parabolic motions for the 2-body problem $W_0$ (see \cite{GMS2016}). The major challenge in this setting is that the perturbation framework is singular: the frequency of the perturbation is unbounded as $G_0\to\infty$. This results in the so-called exponential small splitting of separatrices \cite{Neishtadt}. It is well known that understanding the splitting in this singular perturbation framework is rather challenging. However, from the point of view of the splitting problem there is no major difference between the Hamiltonian~\eqref{eq:H4bp} and the Hamiltonian of the restricted circular 3-body problem considered in~\cite{GMS2016}. Thus, we just briefly sketch the ideas needed to study the manifolds $W^{u,s}_\infty$ and point out the minor differences which we must handle in our case.

The first step is to notice that since the dependence on $(\tilde \alpha, v)$ is only through $\phi=\tilde \alpha - G_0^3 v$, we can define a new variable $\phi=\tilde \alpha - G_0^3 v$ and obtain a new Hamiltonian 
\[
\mathcal{H}(\tilde r, \phi, \tilde y, \tilde G; \mu, G_0) = \frac{\tilde y^2}{2} + \frac{\tilde G^2}{2 \tilde r^2} - \frac{1}{\tilde r} - V(\tilde r, \phi; \mu, G_0) - G_0^3 \tilde G .
\]
We denote the invariant manifolds in rescaled coordinates by $\tilde W_\infty^{u,s}$ and by $\tilde W_0$ the unperturbed homoclinic manifold (i.e. the manifold of parabolic motions for the 2-body problem). In the new variables $\tilde W_0$ can be parameterized by 
\[
(\tilde r, \tilde \alpha, \tilde y, \tilde G) = (\tilde r_h(v), \tilde \alpha_h(v), \tilde y_h(v), 1) .
\]
The functions $\tilde r_h, \tilde \alpha_h, \tilde y_h$ have explicit formulas, as stated in the next proposition from~\cite{MartinezPinyol}.

\begin{prop}
There exist unique functions $\tilde r_h, 
 \tilde y_h, \tilde\alpha_h:\mathbb{R} \to \mathbb{R}$ satisfying
\[
\begin{aligned}
\dot{\tilde r}_h &= \tilde y_h \\
\dot{\tilde y}_h &= \frac{1}{\tilde r_h^3} - \frac{1}{\tilde r_h^2} \\
\dot{\tilde \alpha}_h &= \frac{1}{\tilde r_h^2} .\\
\end{aligned}
\]

With the change $v = (\tau+\tau^3/3)/2$, they can be written as
\[
r_h(v(\tau)) = \frac{1+\tau^2}{2}, \qquad e^{i\alpha_h(v(\tau))} = \frac{\tau-i}{\tau+i} .
\]
\end{prop}
In order to study the invariant manifolds in the perturbed case, we will restrict ourselves to the energy level $\{\mathcal{H} = -G_0^3\}$. Since the invariant manifolds $\tilde W_\infty^{u,s}$ are Lagrangian, they can be locally parameterized as graphs of generating functions $S(\tilde r, \phi; \mu, G_0)$, which are solutions of the Hamilton-Jacobi equation
\begin{equation}\label{eq:HJ_RPC4BP}
    \mathcal{H}(\tilde r, \phi, \partial_{\tilde r} S, \partial_\phi S; \mu, G_0) = -G_0^3 . 
\end{equation}
Call $S_0(\tilde r, \phi)$ the solution of the Hamilton-Jacobi equation
\[
\mathcal{H}(\tilde r, \phi, \partial_{\tilde r} S_0, \partial_\phi S_0; G_0) = -G_0^3, \qquad \mathcal{H}(\tilde r, \phi, \tilde y, \tilde G; \mu, G_0) = \frac{\tilde y^2}{2} + \frac{\tilde G^2}{2 \tilde r^2} - \frac{1}{\tilde r} - G_0^3 \tilde G ,
\]
and let $S_1=S-S_0$. Consider the change of variables $(\tilde r, \phi) = (\tilde r_h(v), \xi+\tilde \alpha_h(v))$ and define $T_1(v, \xi; \mu, G_0) = S_1(\tilde r_h(v), \xi + \tilde \alpha_h(v); \mu, G_0)$. Denote by $T_1^u$ the unique such function that satisfies the boundary conditions $\lim_{v\to-\infty} \tilde y_h(v)^{-1} \partial_v T_1^u(v,\xi)=0$, $\lim_{v\to-\infty} \partial_\xi T_1^u(v,\xi)$, and by $T_1^s$ the one that satisfies $\lim_{v\to+\infty} \tilde y_h(v)^{-1} \partial_v T_1^s(v,\xi)=0$, $\lim_{v\to+\infty} \partial_\xi T_1^s(v,\xi)=0$. The existence of these functions and their extension to suitable complex domains follow from the very same arguments deployed in \cite{GMS2016}.

The upshot of introducing the Hamilton-Jacobi formalism is that the 2-dimensional invariant manifolds $\tilde W_\infty^{u,s} \cap \{\mathcal{H} = -G_0^3\}$ can be now parametrized as
\begin{equation}\label{eq:param_inv_manifolds_RPC4BP}
\begin{aligned}
    \tilde r &= \tilde r_h(v) \\
    \tilde y &= \tilde y_h(v)-\tilde y_h(v)^{-1}\left(\partial_v T_1^{u,s}(v,\xi; \mu, G_0) - \tilde r_h(v)^{-2} \partial_\xi T_1^{u,s}(v,\xi; \mu, G_0)\right) \\
    \phi &= \xi + \tilde \alpha_h(v) \\
    \tilde G &= 1 + \partial_\xi T_1^{u,s} (v,\xi; \mu, G_0) .
\end{aligned}
\end{equation}
Fixing now a section $\{\phi=\phi_0\}$, the invariant manifolds become invariant curves $\gamma^{u,s}$ of the Poincar\'{e} map $\mathcal{P}_{G_0, \phi_0} \colon \{\phi=\phi_0\} \to \{\phi = \phi_0 + 2\pi\}$. The invariant curves $\gamma^{u,s}$ have parameterizations of the form 
\begin{equation}\label{eq:param_inv_curves_RPC4BP} 
\begin{aligned}
\tilde r &= \tilde r_h(v) \\
\tilde y &= Y_{\phi_0}^{u,s}(v) \\ 
&=\tilde y_h(v)-\tilde y_h(v)^{-1}\left(\partial_v T_1^{u,s}(v,\xi; \mu, G_0) - \tilde r_h(v)^{-2} \partial_\xi T_1^{u,s}(v,\xi; \mu, G_0)\right)\Big|_{\xi=\phi_0-\alpha_h(v)}.
\end{aligned}
\end{equation}
In this way, analyzing the existence of tangencies between the manifolds $\tilde W^{u,s}_\infty$ is reduced to proving the existence of degenerate critical points of the function $v\mapsto T_1^{u}-T_1^s$. The same argument in \cite{GMS2016} shows that there exists a common domain $D\subset\mathbb C\times(\mathbb C/ 2\pi \mathbb Z)$, which contains a subset of the form $(v_-, v_+) \times (\mathbb{R}/2\pi\mathbb{Z})$ with $v_-, v_+ \in \mathbb{R}$, and where both $T^{u}_1$ and $T^s_1$ are well-defined.

We now recall that the potential $V(\tilde r, \phi; \mu, G_0)$ is even in $\phi$ (remember that this is a consequence of the fact that two of the primaries were taken with equal masses). Then as in \cite{GMS2016} we know that 
\begin{equation}\label{eq:symmetry_sols_HJ_RPC4BP}
    T_1^s(v, \xi) = -T_1^u(-v, -\xi). 
\end{equation}
Define $\widetilde \Delta (v, \xi) = T_1^u(v, \xi) - T_1^s(v, \xi)$ and the linear partial differential operator
\begin{equation}\label{def:DifferentialOperator_original}
 \widetilde{\mathcal{L}}=\left(1+A(v,\xi)\right)\partial_v - G_0^3\left(1+B(v,\xi)\right)\partial_\xi ,
\end{equation}
where
\begin{equation}\label{def:DiffOperator_AB}
\begin{aligned}
A(v,\xi)=& \frac{1}{2 \tilde y_h^2}\left(\left(\partial_v T_1^s+\partial_v T_1^u\right)
-\frac{1}{ \tilde r_h^2}(\partial_\xi T_1^s+\partial_\xi T_1^u)\right) ,\\
B(v,\xi)=&\frac{G_0^{-3}}{2 \tilde r_h^2 \tilde y_h^2}\left(\left(\partial_v T_1^s+\partial_v T_1^u\right)
-\frac{1}{ \tilde r_h^2}(\partial_\xi T_1^s+\partial_\xi T_1^u)\right)\\
&-\frac{G_0^{-3}}{2 \tilde r_h^2}\left(\partial_\xi T_1^s+\partial_\xi
T_1^u\right).
\end{aligned}
\end{equation}
It is a straightforward computation to check that $\widetilde \Delta \in \ker\widetilde{\mathcal{L}}$. To unravel the consequences of this fact, we first introduce a coordinate system in which the differential operator \eqref{def:DifferentialOperator_original} is of constant coefficients. The following proposition is the analogue of Theorem~6.3 of \cite{GMS2016}.

\begin{prop}\label{prop:straigthen_operator_RPC4BP}
    There exists a close to identity change of variables $(w, \xi) \mapsto (w + \mathcal{C}(w,\xi), \xi)$ which is well defined on the complex domain $D\subset \mathbb C\times(\mathbb C/2\pi\mathbb Z)$ and is such that $\Delta(w, \xi) := \widetilde \Delta(w + \mathcal{C}(w,\xi), \xi)$ satisfies $\Delta \in \ker \mathcal{L}$, where $\mathcal{L} = \partial_w - G_0^3 \partial_\xi$. Moreover, one can choose $\mathcal{C}$ such that $\mathcal{C}(-w, -\xi) = -\mathcal{C}(w,\xi)$.
\end{prop}

As a consequence of Proposition~\ref{prop:straigthen_operator_RPC4BP} and the symmetry in \eqref{eq:symmetry_sols_HJ_RPC4BP}, we have the following.

\begin{prop}\label{prop:Fourier_series_Delta_RPC4BP}
The function $\Delta$  defined in Proposition~\ref{prop:straigthen_operator_RPC4BP} satisfies the symmetry $\Delta(-w,-\xi)=\Delta(w,\xi)$ and has a Fourier series of the form
  \begin{equation}\label{eq:Fourier_series_Delta_RPC4BP}
    \Delta(w, \xi) = \Delta^{[0]} + \sum_{\ell \in \mathbb{N}} \Delta^{[\ell]} \cos(\ell (G_0^3 w +\xi)),
  \end{equation}
for certain real coefficients $\Delta^{[\ell]}(\mu, G_0)$. 
Since $\Delta$ only depends on $w, \xi$ through $z=G_0^3 w + \xi$, to simplify the notation we will write $\Delta(z)=\Delta(G_0^3 w + \xi) := \Delta(w, \xi)$.
\end{prop}
\vspace{0.2cm}

The key now to study the coefficients $\Delta^{[l]}$ is that the domain $D\subset \mathbb C\times(\mathbb C/2\pi\mathbb Z)$ where $\Delta(w,\xi)$ is well defined gets sufficiently close to the lines $\mathrm{Im} w=\pm 1/3$. Then, a standard argument shows that the coefficients in \eqref{eq:Fourier_series_Delta_RPC4BP} decay exponentially fast. Moreover, using the exact same ideas as in \cite{GMS2016} one can show that $\Delta$ is uniformly approximated by the so-called Melnikov potential 
\[
L(G_0^3 w + \xi; \mu, G_0) = \int_{-\infty}^{+\infty} V(\tilde r_h(s), \xi+ G_0^3 w -G_0^3 s + \tilde \alpha_h(s); \mu, G_0) ds 
\]
for $(w,\xi)\in D$. This allows us to obtain exponentially decaying estimates for the differences $\Delta^{[l]}-L^{[l]}$. The following result is analogous to Lemma~6.8 in \cite{GMS2016}.
\begin{prop}\label{prop:approx_inv_manifolds_RPC4BP} 
    There exists $G_0^\ast>0$ such that for any $G_0 > G_0^{\ast}$ and $\mu \in (0, 1/2]$ the following bounds hold 
    \[
    \begin{aligned}
       \left| \Delta^{[\ell]}(\mu, G_0) - L^{[\ell]}(\mu, G_0) \right| &\leq \mu^2 |1-3\mu| (K G_0)^{- \frac{7}{2} + \frac{3}{2} |\ell|} e^{-\frac{|\ell| G_0^3}{3}} \quad \text{ for } \ell \notin 3 \mathbb{N} ,\\
       \left| \Delta^{[\ell]}(\mu, G_0) - L^{[\ell]}(\mu, G_0) \right| &\leq \mu^2 (K G_0)^{- \frac{7}{2} + \frac{3}{2} |\ell|} e^{-\frac{|\ell| G_0^3}{3}} \quad \text{ for } \ell \in 3 \mathbb{N},
    \end{aligned}
    \]
    for some constant $K > 0$ independent of $\mu$ and $G_0$.
\end{prop}

The following proposition, which can be obtained from the results in \cite{DKdlRS2019}, gives estimates for $L$.

\begin{prop}\label{prop:Fourier_expansion_Melnikov_RPC4BP}
    The function $L(w, \xi; \mu, G_0)$ satisfies
    \[
    L(w, \xi; \mu, G_0) = L^{[0]}(\mu, G_0) + \sum_{\ell \in \mathbb{N}} L^{[\ell]}(\mu, G_0) \cos(\ell (G_0^3 w + \xi)) ,
    \]
    where
    \[
    \begin{aligned}
	L^{[1]}(\mu, G_0) &= \frac{1}{2} \sqrt{\frac{3 \pi}{2}} \mu (1 - 3 \mu) (1 - 2 \mu) G_0^{-3/2} e^{\dps\tfrac{-G_0^{3}}{3}} \left(1+\bigO\left(G_0^{-1}\right)\right) ,\\
	L^{[2]}(\mu, G_0) &= 4 \sqrt{\pi} \mu (1 - 3 \mu) G_0^{1/2} e^{\dps\tfrac{-2 G_0^{3}}{3}} \left(1+\bigO\left(G_0^{-1}\right)\right) ,\\
	L^{[3]}(\mu, G_0) &= - 27 \sqrt{2\pi} \mu^2 (1-2 \mu) G_0^{3/2} e^{-G_0^3} \left(1+\bigO\left(G_0^{-1}\right)\right) ,\\
    L^{[\ell]}(\mu, G_0) &= \bigO\left(G_0^{\ell-3/2} e^{-\frac{\ell G_0^3}{3}}\right), \quad \text{ for } \ell \geq 4 .
    \end{aligned}
    \]
\end{prop}

We are now ready to prove the main result of this appendix.

\begin{prop}
    Let $G_0^\ast > 0$ be big enough. There exists a curve $\Upsilon$ in the parameter region $(\mu, G_0) \in (0, 1/2] \times (G_0^\ast, +\infty)$ of the form $\mu = 1/3 + \bigO(G_0^3 e^{-2 G_0^3/3})$ such that, for $(\mu, G_0) \in \Upsilon$, the invariant curves $\gamma^{u,s}$ have a quadratic homoclinic tangency. Moreover, the quadratic tangency unfolds generically with the parameter $\mu$.
\end{prop}
\begin{proof}
In view of the parameterization of the invariant curves $\gamma^{u,s}$ given in the variables $(v,\xi)$ in \eqref{eq:param_inv_curves_RPC4BP}, together with the fact that the change of variables $(w,\xi) \mapsto (v,\xi)$ is close to identity as stated in Proposition~\ref{prop:straigthen_operator_RPC4BP}, in order to prove the existence of a quadratric tangency between $\gamma^{u,s}$ we have to find a solution $(z,\mu,G_0)$ of the system
\begin{equation}\label{eq:existence_quadratic_tang_RPC4BP}
\begin{aligned}
    \Delta'(z; \mu, G_0) &= 0 \\
    \Delta''(z; \mu, G_0) &= 0 
\end{aligned}
\end{equation}
for which 
\[
\Delta'''(z; \mu, G_0) \neq 0 
\]
First, let us take the parameter $\mu$ of the form $\mu=\tfrac{1}{3} + \hat{\mu} G_0^3 e^{-2 G_0^3/3}$, with $\hat{\mu}=\bigO(1)$. Then, from Propositions~\ref{prop:Fourier_series_Delta_RPC4BP},~\ref{prop:Fourier_expansion_Melnikov_RPC4BP} and~\ref{prop:approx_inv_manifolds_RPC4BP}, we know that
\[
\begin{aligned}
    \Delta^{[1]} &= -\frac{3}{2} \sqrt{\frac{3 \pi}{2}} \hat{\mu} \mu (1 - 2 \mu) G_0^{3/2} e^{-G_0^3} + \bigO\left(G_0 e^{-G_0^3}\right) , \\
    \Delta^{[3]} &= - 27 \sqrt{2\pi} \mu^2 (1-2 \mu) G_0^{3/2} e^{-G_0^3} + \bigO\left(G_0 e^{-G_0^3}\right) , \\
    \Delta^{[\ell]} &= \bigO\left(G_0^{7/2}e^{-\frac{4 G_0^3}{3}}\right) \quad \text{ for } \ell \neq 1,3.
\end{aligned}
\]
Thus \eqref{eq:existence_quadratic_tang_RPC4BP} becomes
\[
\begin{aligned}
    -\Delta^{[1]} \sin(z) - 3 \Delta^{[3]} \sin(3 z) + \mathcal{R}'(z; \mu, G_0) &= 0 \\
    -\Delta^{[1]} \cos(z) - 9 \Delta^{[3]} \cos(3 z) + \mathcal{R}''(z; \mu, G_0) &= 0 \\
\end{aligned}
\]
and the non-vanishing of the third derivative takes the form
\[
   \Delta^{[1]} \sin(z) + 27 \Delta^{[3]} \sin(3 z) + \mathcal{R}'''(z; \mu, G_0) \neq 0
   \]
where $\mathcal{R}$ is a Fourier cosine series whose coefficients are of order $\bigO\left(G_0^{7/2} e^{-\frac{4 G_0^3}{3}}\right)$. Taking $z=\pi/2$, the previous system boils down to satisfying
\[
    -\Delta^{[1]} + 3 \Delta^{[3]} + \bigO\left(G_0^{7/2}e^{-\frac{4 G_0^3}{3}}\right) = 0.
\]
By the implicit function theorem, this  equation has a solution of the form $\hat{\mu}=\hat{\mu}(G_0) = 12 \sqrt{3} + \bigO\left(G_0^3 e^{-2 G_0^3/3}\right)$, which clearly satisfies that
\[
\Delta^{[1]} - 27 \Delta^{[3]} + \bigO\left(G_0^{7/2}e^{-\frac{4 G_0^3}{3}}\right) \neq 0.  
\]
This proves the existence of a quadratic homoclinic tangency between the invariant curves $\gamma^{u,s}$ along a curve $\Upsilon$ parameterized by $\mu=\mu(G_0)=1/3 + 12\sqrt{3} G_0^3 e^{-2G_0^3/3} + \bigO\left(G_0^6 e^{-4G_0^3/3}\right)$. 

It is trivial to check that, since the derivative of the function $\mu\mapsto \Delta'(z;\mu,G_0)$  does not vanish at $\mu=\mu(G_0)$ this tangency unfolds generically.
\end{proof}

\section{Proof of Proposition \ref{prop:redressament_de_les_varietats}}\label{sec:appendixtechnical}

The proof of Proposition \ref{prop:redressament_de_les_varietats} is split in two steps.
\subsection{Proof of Proposition~\ref{prop:redressament_de_les_varietats}. First step}

In this section we perform a finite number of steps of normal form. This is achieved by the standard Lie method, using the Poisson bracket $\{\cdot,\cdot\}$ in~\eqref{def:parentesi_de_Poisson}. Like in the hyperbolic case, the homological equation allows to kill all terms of a given degree in $(q,p)$, except some resonant ones, independent of $t$. We check that there are no resonant terms in the original Hamiltonian and that no resonant terms are produced if an appropriate choice is made at each step. This argument requires the Hamiltonian to be even in $t$. If this assumption is not met, it is possible to obtain an analogous result with a normal form consisting only on time independent resonant terms, analogous to Birkhoff normal form. 

\begin{prop}
	\label{prop:forma_normal_formal}
	Let $H$ be the Hamiltonian in~\eqref{def:HamiltoniaH}. Let $N \in \N$ be fixed. There exists $\rho>0$ (depending on $N$) and a close to identity real-analytic change of coordinates $\Phi$, defined on $\mathbb B_\rho^2\times\mathbb T_\rho\subset\mathbb C^2\times(\mathbb C/2\pi\mathbb Z)$, satisfying $\pi_t \Phi = \Id$, preserving the 2-form $\omega$ and  such that $H\circ\Phi=\NN  + \wt H_1$, with
	\[
	 \wt H_1(q,p,t) =   (q+p)^3 g_N(q,p,t),
	\]
	where $\NN$ is defined in~\eqref{def:Hamiltonia_N_i_Hu},  and $g_N(q,p,t) = \OO_N(q,p)$.
\end{prop}

We  introduce four different  algebras of formal series  on $(p,q)$ with periodic coefficients. 
To begin with, for $r \in \mathbb{N}$, consider the spaces of homogeneous polynomials in $(q,p)$
\[
\begin{aligned}
    P_{r}^{+} &= \{\sum_{\ell+m=r} f_{\ell, m}(t) q^{\ell} p^m \mid f_{\ell, m} \colon \mathbb{R}/2\pi\Z \to \mathbb{R}, \, f_{\ell, m}(-t)=f_{\ell, m}(t), \forall \ell, m \} , \\
    P_{r}^{-} &= \{\sum_{\ell+m = r} f_{\ell, m}(t) q^{\ell} p^m \mid f_{\ell, m} \colon \mathbb{R}/2\pi\Z \to \mathbb{R}, \, f_{\ell, m}(-t)=-f_{\ell, m}(t), \forall \ell, m \},
\end{aligned}
\]
and the corresponding spaces $P_{r}$ for homogeneous polynomials of degree $r$ in $(q,p)$ independent of $t$. Note that $P_{r} \subset P_{r}^{+}$. We also
define, for $k\in \mathbb{N}$, the vector spaces of formal series $E_{\ge 2k}^{\pm} = \bigoplus_{r \ge k} P_{2r}^{\pm}$ and $D_{\ge 2k-1}^{\pm} = \bigoplus_{r \ge k} P_{2r-1}^{\pm}$.
Clearly if $k<k'$ then $E_{\ge 2k'}^{\pm} \subset E_{\ge 2k}^{\pm}$ and $D_{\ge 2k'-1}^{\pm} \subset D_{\ge 2k-1}^{\pm}$. It is immediate to check that they are  algebras with the standard product. Moreover, we can consider the action of formal differentiation on these algebras. 

The following lemma is immediate.

\begin{lem}\label{lem:parity_poisson} 
    The following inclusions hold:
    \begin{align}
    \{E_{\ge 2k}^{+}, E_{\ge 2k}^{+}\} \subset D_{\ge 4k+1}^{+}\label{eq:(E+,E+)} ,\\
    \{ E_{\ge 2k}^{+}, E_{\ge 2k}^{-}\}, \{E_{\ge 2k}^{-}, E_{\ge 2k}^{+}\} \subset D_{\ge 4k+1}^{-}\label{eq:(E+,E-)} ,\\
    \{E_{\ge 2k}^{+},D_{\ge 2k-1}^{+}\}, \{D_{\ge 2k-1}^{+}, E_{\ge 2k}^{+}\}  \subset E_{\ge 4k}^{+}\label{eq:(E+,D+)} ,\\
    \{E_{\ge 2k}^{+},D_{\ge 2k-1}^{-}\}, \{D_{\ge 2k-1}^{-},E_{\ge 2k}^{+}\} \subset E_{\ge 4k}^{-}\label{eq:(E+,D-)} ,\\
    \{E_{\ge 2k}^{-}, E_{\ge 2k}^{-}\} \subset D_{\ge 4k+1}^{+}\label{eq:(E-,E-)} ,\\
    \{E_{\ge 2k}^{-},D_{\ge 2k-1}^{+}\}, \{D_{\ge 2k-1}^{+},E_{\ge 2k}^{-}\} \subset E_{\ge 4k}^{-}\label{eq:(E-,D+)} ,\\
    \{E_{\ge 2k}^{-},D_{\ge 2k-1}^{-}\}, \{D_{\ge 2k-1}^{-},E_{\ge 2k}^{-}\} \subset E_{\ge 4k}^{+}\label{eq:(E-,D-)} ,\\
    \{D_{\ge 2k-1}^{+},D_{\ge 2k-1}^{+}\} \subset D_{\ge 4k-1}^{+}\label{eq:(D+,D+)} ,\\
    \{D_{\ge 2k-1}^{+},D_{\ge 2k-1}^{-}\},  \{D_{\ge 2k-1}^{-},D_{\ge 2k-1}^{+}\} \subset D_{\ge 4k-1}^{-}\label{eq:(D+,D-)} ,\\
    \{D_{\ge 2k-1}^{-},D_{\ge 2k-1}^{-}\} \subset D_{\ge 4k-1}^{+}\label{eq:(D-,D-)} .
    \end{align} 
\end{lem}

Finally, we define
\[
\mathcal{W} = \{(q+p)^3 f(q,p,t), \text{ for some analytic function } f \colon \mathbb B_\rho^2\times\mathbb T_\rho\subset\mathbb C^2\times(\mathbb C/2\pi\mathbb Z) \to \mathbb C\} .
\]
Since functions in $\mathcal{W}$ do not depend on $I$, it is clear that from the definition of the Poisson bracket in~\eqref{def:parentesi_de_Poisson} that, if $g, \tilde{g} \in \mathcal{W}$, then $\{g, \tilde{g}\} \in \mathcal{W}$.

We will consider changes of coordinates $\Phi_F^1$ given by the time $1$ flow of a Hamiltonian $F$, with respect to the 2-form $\omega$. Then, 
\begin{equation}\label{eq:Lie_series}
H \circ \Phi_F^1 = \sum_{j\ge 0} \frac{1}{j!} \mathrm{ad}_F^j H,
\end{equation}
where $\mathrm{ad}_F^0 H = H$ and $\mathrm{ad}_F^j H = \{\mathrm{ad}_F^{j-1} H, F\}$, for $j\ge 1$. 
Moreover, using~\eqref{def:parentesi_de_Poisson},
\[
\{\NN, F\} = (q+p)^3 \left(q \frac{\partial F}{\partial q} - p \frac{\partial F}{\partial p} \right) - \frac{\partial F}{\partial t}.
\]

\begin{lem}\label{lem:iterative_first_step}
    Let $k \in \mathbb{N}$. Given $H=\NN+H_{k}$ such that $H_{k} \in (E_{\ge 2k+2}^{+} \oplus D_{\ge 2k+3}^{-}) \cap \mathcal{W}$, there exists a close to identity real-analytic change of coordinates $\Phi$ defined on $\mathbb B_\rho^2\times\mathbb T_\rho\subset\mathbb C^2\times(\mathbb C/2\pi\mathbb Z)$, satisfying $\pi_t \Phi_k = \Id$, preserving the 2-form $\omega$ and such that $\wt H = H \circ \Phi$ is of the form $\wt H = \NN + H_{k+1}$, with $H_{k+1} \in (E_{\ge 2k+4}^{+} \oplus D_{\ge 2k+5}^{-}) \cap \mathcal{W}$.
\end{lem}

\begin{proof}
    The proof is divided into two steps. In the first one, we look for some $F^{[1]}$ so that the corresponding transformation $\Phi_{F^{[1]}}^1$ given by \eqref{eq:Lie_series} kills the terms in $P_{2k+2}^{+}$ of $H_{k}$, obtaining a new Hamiltonian $\NN + H_{k}^{[1]}$ with $H_{k}^{[1]}\in (E_{\ge 2k+4}^{+} \oplus D_{\ge 2k+3}^{-}) \cap \mathcal{W}$. Then, in the second one, we use some $F^{[2]}$ to kill the terms in $P_{2k+3}^{-}$ of $H_{k}^{[1]}$, obtaining a new Hamiltonian $\NN+H_{k}^{[2]}$ where $H_{k}^{[2]}\in (E_{\ge 2k+4}^{+} \oplus D_{\ge 2k+5}^{-}) \cap \mathcal{W}$. Then we take $H_{k+1} := H_{k}^{[2]}$ and $\Phi := \Phi_{F^{[2]}}^1 \circ \Phi_{F^{[1]}}^1$. We now proceed to explain in detail how to perform each of these steps.

    For the first step, we look for $F^{[1]}$ such that $\{\NN, F^{[1]}\}=-\pi_{P_{2k+2}^{+}} H_k$. Writing $\pi_{P_{2k+2}^{+}} H_k = (q+p)^3 h_1(q,p) + h_2(q,p,t)$, for some $h_1 \in P_{2k-1}$, $h_2 \in P_{2k+2}^{+}$, $h_2$ with zero mean respect to $t$, it is clear that we have to look for $F^{[1]}$ of the form $F^{[1]}=F_1^{[1]} + F_2^{[1]}$, for some $F_1^{[1]}\in P_{2k-1}$ and $F_2^{[1]} \in P_{2k+2}^{-}$ satisfying
    \[
    \begin{aligned}
     q \frac{\partial F_1^{[1]}}{\partial q} (q,p)-p \frac{\partial F_1^{[1]}}{\partial p} & = - h_1(q,p), \\
     \frac{\partial F_2^{[1]}}{\partial t} (q,p,t) & = h_2(q,p,t).
	\end{aligned}
     \]
     
     The first equation has a unique solution because the only resonant monomials of the operator $F \to q \frac{\partial F}{\partial q} (q,p)-p \frac{\partial F}{\partial p}$, where $F$ is a homogeneous polynomial of degree $m$ in $(q,p)$, are $q^{[m/2]} p^{[m/2]}$ and the degree of $h_1$ is $2k-1$, which is odd. The second one can be solved because $h_2(t,q,p)$ has zero mean; moreover, we can choose a solution $F_2^{[1]}$ which is odd in $t$, i.e., such that $F_2^{[1]} \in P_{2k+2}^{-}$.

     For this $F^{[1]}$, and taking into account \eqref{eq:Lie_series} and Lemma~\ref{lem:parity_poisson}, it is immediate to check that $H_k \circ \Phi_{F^{[1]}}$ is of the form $\NN + H_k^{[1]}$, where $H_k^{[1]} \in (E_{\ge 2k+4}^{+} \oplus D_{\ge 2k+3}^{-}) \cap \mathcal{W}$. 

     For the second step, we look for $F^{[2]}$ such that $\{\NN, F^{[2]}\}=-\pi_{P_{2k+3}^{-}} H_k^{[1]}$. It is then clear that we have to look for $F^{[2]}\in P_{2k+3}^{+}$ satisfying
    \[
     \frac{\partial F^{[2]}}{\partial t} (q,p,t) = \pi_{P_{2k+1}^{-}} H_k^{[1]} .
     \]
     This equation has a solution because the right-hand side, being odd in $t$, has zero mean. For this $F^{[2]}$, and taking into account \eqref{eq:Lie_series} and Lemma~\ref{lem:parity_poisson}, it is immediate to check that $H_k^{[1]} \circ \Phi_{F^{2]}}$ is of the form $\NN + H_k^{[2]}$, where $H_k^{[2]} \in (E_{\ge 2k+4}^{+} \oplus D_{\ge 2k+5}^{-}) \cap \mathcal{W}$.
\end{proof}

\begin{proof}[Proof of Proposition~\ref{prop:forma_normal_formal}]
Since $H_1$ in~\eqref{def:Hamiltonia_Hu} satisfies that $H_1 \in E_{\ge 4}^{+} \cap \mathcal{W}$, the proof follows applying Lemma~\ref{lem:iterative_first_step} iteratively.
\end{proof}  	
   	
\subsection{Proof of Proposition~\ref{prop:redressament_de_les_varietats}: Second step}
\label{sec:forma_normal:segon_pas}
   
 In this section we use that the periodic orbit at $(q,p) = (0,0)$ possesses stable and unstable invariant manifolds \cite{McGeheestablemfold}. Straightening them allows us to kill an infinite number of terms of the expansion of the Hamiltonian. The procedure uses the functions that parametrize the invariant manifolds. It is well known that, since the periodic orbit is degenerate, in general these functions are no analytic at the periodic orbit but (a priori) only in some sectorial domain. 
 This fact will have significant impact on the proof.  
 For this reason, we start the section by introducing some definitions of spaces of functions. Let $V_{\rho}$ be the sectorial domain introduced in \eqref{eq:sectorialdom}.
Define the spaces of functions
\[
\X_{r,\rho}  = \{f: V_{\rho}\times \T_\sigma \to \C\mid \text{analytic, \;$\|f\|_{r,\rho} < \infty$}\}, 
\]
where $\T_\sigma = \{t \in \C/2\pi\Z\mid |\Im t | < \sigma\}$ and
\[
  \|f\|_{r,\rho}  = \sup_{(z,t)\in V_{\rho}\times \T_\sigma} |z^{-r} f(z,t)|
\]
It is clear that $\X_{r,\rho}$ 
is a Banach space and the norm satisfies
\[
\|f g \|_{r+s,\rho} \le \|f\|_{r,\rho} \|g \|_{s,\rho}.
\]



 We are now ready to complete the proof of Proposition \ref{prop:redressament_de_les_varietats}.
 
 \begin{proof}[Proof of Proposition \ref{prop:redressament_de_les_varietats}]
 Let $N = \wt N+2M$. Let $H = \NN + \wt H_1$ be the Hamiltonian obtained in Proposition~\ref{prop:forma_normal_formal}, where $\wt H_1(q,p,t) = (q+p)^3(f_4(qp)+ \OO_{N-2}(q,p))$ is analytic in a neighborhood of $\{q=p=0, \; t \in \T\}$. 
  Assume that $(\wt H_1)_{\mid q = 0}, (\wt H_1)_{\mid p = 0} \not \equiv 0$. Otherwise, the claim is proved. 
 Performing an additional step of normal form as in the proof of Proposition~\ref{prop:forma_normal_formal},
 we can assume that
 \begin{equation}
 \label{def:ordres_minims_no_nuls}
 \wt  H_1 (q,0,t) =  q^{N+1} + \OO_{N+2}(q), \qquad  \wt  H_1 (0,p,t) =  p^{N+1} + \OO_{N+2}(p).
 \end{equation}
 
 By Theorem 2.10 of~\cite{BaldomaFM20},
 the periodic orbit at $q=p=0$ possesses local invariant stable and unstable manifolds, $W^{u,s}$, of the form
 \begin{equation}
 	\label{def:Ws_i_Wu}
 W^s = \{p = \gamma^s(q,t) \}, \qquad  W^u = \{q = \gamma^u(p,t)\}, 
 \end{equation}
 where the functions $\gamma^{u,s}(x,t)$ are $C^\infty$ in a neighborhood of the $\{0\}\times (\R/2\pi\Z)$, satisfy $\gamma^{u,s}(x,t) = \OO_2(x)$, and are analytic for $(x,t) \in V_{\rho} \times \mathbb{T}_{\sigma}$, for some small enough $\rho, \sigma >0$, where $V_{\rho}$ is the sectorial domain defined in \eqref{eq:sectorialdom} and $\mathbb{T}_{\sigma}=\{t \in \C/2\pi\Z \mid |\Im t | < \sigma\}$.
 
 First we claim that, under the current hypotheses, 
  \[
 \gamma^{u,s}(x,t) =  x^{N} + \OO(x^{N+1}).
 \] 	
The proof of the claim is a straightforward computation. Indeed, in the case of the stable manifold, it follows from 
  writing $\gamma^s(x,t) =  \alpha x^k  + \wt \OO(x^{k+1}) + \OO(x^{k+4})$, where $ \wt \OO(x^{j}) $ means $\OO(x^j)$ independent of $t$, with $k\ge 2$, imposing the invariance condition 
 \[
 \dot p_{\mid p= \gamma^s(q,t)}- \partial_q \gamma^s (q,t) \dot q_{\mid p= \gamma^s(q,t)} - \partial_t \gamma^s (q,t) = 0,
 \]
and taking into account that
 \[
 \begin{aligned}
 \dot p_{\mid p = \gamma^s(q,t)} & = \alpha q^{k+3}  -(N+1) q^{N+3} + \OO(q^{k+4})  + \OO(q^{N+4}), \\
 \partial_q \gamma^s (q,t) \dot q_{\mid p= \gamma^s(q,t)}  & = k \alpha q^{k+3} + \OO(q^{k+4})  ,\\
  \partial_t \gamma^s (q,t)  & = \OO(q^{k+4}).
\end{aligned}
 \]
Hence, $k=N$ and $\alpha = 1$.

Next, we perform two consecutive changes of variables. The first one straightens the stable manifold and, the second one, the unstable one. Unlike the changes in Proposition~\ref{prop:forma_normal_formal}, these changes are not analytic in a neighborhood of the origin: they are only defined in sectorial domains.

To do so, let $F(q,t)$ be the function defined by
\begin{equation}
	\label{def:F_redressant_varietat_estable}
\frac{\partial F}{\partial q}(q,t) =\frac{1}{2q^2}\left(\frac{1}{\left(1+q^{-1} \gamma^s(q,t)\right)^2} - 1\right), \qquad F(0,t) = 0.
\end{equation}

It is clear that $\partial_q F(q,t) = -q^{N-3} + \wt \OO_{N-2}(q)+ \OO_{N+1}(q)$ and is analytic in $V_{\rho}\times \T_{\sigma}$, for some $\rho,\sigma >0$. Let $\Phi_F^1$ be the time 1 map of the Hamiltonian $F$. Since $F$ only depends on $(q,t)$, it is immediate to check that
\[
\Phi_F^1(q,p,t,I)  = \left(q, p +P(q,p,t),t,I- \frac{\partial F}{\partial t}(q,t)\right).
\]
where 
\begin{equation}
\label{def:Pqpt}
P(q,p,t)= \frac{q+p}{\left( 1+2(q+p)^2\frac{\partial F}{\partial q}(q,t)\right)^{1/2}}-(q+p) = (q+p)^3 q^{N-3}+(q+p)^3\OO_{N-2}(q,p)
\end{equation}
is defined and analytic in $V_{\rho}\times B_\rho \times \T_{\sigma} \times \C$,
$C^\infty$ at $\{0\}\times B_\rho \times \T_{\sigma} \times \C$ and, for some $\rho' > \rho$,
\[
\Phi_F^1 (V_{\rho}\times B_\rho \times \T_{\sigma} \times \C ) \subset V_{\rho'}\times B_{\rho'} \times \T_{\sigma} \times \C.
\]
Hence,  $H \circ \Phi_F^1$ is analytic in $V_{\rho}\times B_\rho \times \T_{\sigma} \times \C$,
$C^\infty$ at $\{0\}\times B_\rho \times \T_{\sigma} \times \C$.

In view of~\eqref{def:Ws_i_Wu}, the stable manifold in the new variables is given by
\[
\pi_p \Phi_F^1(q,p,t,I) - \gamma^s (\pi_q \Phi_F^1(q,p,t,I), \pi_t \Phi_F^1(q,p,t,I)) = 0,
\]
that is,
\[
 \frac{q+p}{\left( 1+2(q+p)^2\frac{\partial F}{\partial q}(q,t)\right)^{1/2}}-q - \gamma^s(q,t) = 0. 
\]

Taking into account~\eqref{def:F_redressant_varietat_estable}, an immediate computation shows that the above equation is equivalent to $p = 0$. Hence,
\[
0 = \dot p_{\mid p =0} = -(q+p)^3 \frac{\partial }{\partial q} (H \circ \Phi_F^1 )_{\mid p =0}.
\]
This and the fact that $\Phi_F^s = \Id + \OO_{N-1}(q,p)$ imply that
\[
H \circ \Phi_F^1 = \NN  + H_2,
\]
where
\[
H_2 = (\NN  + H_1) \circ \Phi_F^1 - \NN,
\]
satisfies $H_2(q,p,t) = p \OO_{N-1}(q,p)$, that is, we have straightened the stable manifold. 

In order to repeat the procedure to straighten the unstable manifold, we first need to see that, in these new variables, the unstable manifold can be written as in~\eqref{def:Ws_i_Wu}, in the form $q = \tilde \gamma^u(p,t)$, with $\tilde \gamma^u$ defined in $V_{\rho''}\times B_\rho \times \T_{\sigma}$, for some $\rho'' >0$.

In this new set of variables, the unstable manifold is given by the solutions of 
\[
q = \gamma^u(p+P(q,p,t),t)
\]
where $\gamma^u$ was introduced in~\eqref{def:Ws_i_Wu} and $P$ in~\eqref{def:Pqpt}. It is immediate to check, by means of a fixed point argument, that the above equation defines  $q$ as a function $\tilde \gamma^u \in \X_{N,\rho'}$ of $p$, for any $0 < \rho' < \rho$. Hence, in the new variables, the unstable manifold is given by $W^u = \{q = \tilde \gamma^u(p,t)\}$. We can repeat the procedure to straighten $W^u$, now considering the time 1 map, $\Phi_G^1$, of the Hamiltonian $G(p,t)$ defined analogously by~\eqref{def:F_redressant_varietat_estable}, using $\tilde \gamma^u(p,t)$ instead of $\gamma^s(t)$. It satisfies
\[
\Phi_G^1(q,p,t,I)  = \left(q+Q(q,p,t), p ,t,I- \frac{\partial G}{\partial t}(p,t)\right).
\]
where 
\[
Q(q,p,t)= \frac{q+p}{\left( 1+2(q+p)^2\frac{\partial G}{\partial p}(p,t)\right)^{1/2}}-(q+p) = (q+p)^3 p^{N-3}+(q+p)^3\OO_{N-2}(q,p)
\]
is defined and analytic in $B_\rho \times V_{\rho}\times \T_{\sigma} \times \C$,
$C^\infty$ at $B_\rho \times \{0\}\times \T_{\sigma} \times \C$ and, for some $\rho' > \rho$,
\[
\Phi_G^1 (B_\rho \times V_{\rho}\times  \T_{\sigma} \times \C) \subset B_{\rho'} \times V_{\rho'}\times \T_{\sigma} \times \C.
\]
Furthermore, since $Q(q,p,t) = p^N + q \OO_{N-1}(q,p)+\OO_{N+1}(q,p)$, taking $\rho' = \rho - \OO_N(\rho)$
\begin{equation}
\label{PhiGq1_envia_sectors_a_sectors}
\Phi_G^1 (V_{\rho'} \times V_{\rho'/N}\times  \T_{\sigma} \times \C) \subset V_{\rho} \times V_{\rho}\times \T_{\sigma} \times \C,
\end{equation}
which implies that $H\circ \Phi_F^1 \circ \Phi_G^1$ is well defined, analytic in $V_{\rho'} \times V_{\rho'/N}\times  \T_{\sigma} \times \C$ and $C^\infty$ at $\{ 0 \} \times V_{\rho'/N}\times  \T_{\sigma} \times \C \cup V_{\rho'} \times \{0\}\times  \T_{\sigma} \times \C$.

With the same argument as in the case of the stable manifold, $\Phi_G^1$ straightens the unstable manifold. Since it is the identity 
in $\{p=0\}$, the stable manifold, the invariant manifolds of the periodic orbit at the origin of $H\circ \Phi_F^1 \circ \Phi_G^1$
are $\{p=0\}$ and $\{q=0\}$. We write $H\circ \Phi_F^1 \circ \Phi_G^1 = \NN  + H_2$, where
\[
H_2 = H_1 \circ \Phi_F^1 \circ \Phi_G^1 + \NN \circ \Phi_F^1 \circ \Phi_G^1-\NN .
\]
Since $\Phi_F^1, \Phi_G^1 = \Id + \OO_N$, and $H_1 = \OO_N$, it is clear that so is $H_2$. The following lemma implies the claim.

\begin{lem}
\label{lem:descomposicio_en_productes_qp}
Let $H: V_{\rho} \times V_{\rho}\times  \T_{\sigma}  \to \C$, analytic, $C^\infty$ at $\{ 0 \} \times V_{\rho}\times  \T_{\sigma}  \cup V_{\rho} \times \{0\}\times  \T_{\sigma} $, and $H(q,p,t) = \OO_N(q,p)$. Then, for any $k$ such that $2k < N$,
\[
  H(q,p,t) = \sum_{j=0}^{k-1} (qp)^j\left(h_j(q,t) + \tilde h_j (p,t)\right) + (qp)^k    \wt H(q,p,t),
  \]
for some $h_j, \tilde h_j \in  \X_{N-2j,\rho}$, $j=0,\dots,k-1$,
$\wt H:V_{\rho} \times V_{\rho}\times  \T_{\sigma}  \to \C$, analytic, $C^\infty$ at $\{ 0 \} \times V_{\rho}\times  \T_{\sigma}  \cup V_{\rho} \times \{0\}\times  \T_{\sigma}$ and $\wt H(q,p,t) = \OO_{N-2k}(q,p)$. 
\end{lem}

\begin{proof}[Proof of Lemma~\ref{lem:descomposicio_en_productes_qp}]
The claim for $k = 1$ follows immediately from
the equality
\[
H(q,p,t) = H(q,0,t) + H(0,p,t) + qp H_1(q,p,t),
\]
where
\[
H_1(q,p,t) = \int_0^1\int_0^1 \partial_q \partial_p H(\tau q, \sigma p,t)\,d\tau d\sigma.
\]
To see that $H_1(q,p,t) = \OO_{N-2}(q,p)$ take into account the following fact, which is readily deduced from Taylor's formula:
\[
H(q,p,t) = \OO_N(q,p) \Leftrightarrow \partial_q^j \partial_p^i H(0,0,t) = 0 \, \text{ for } i+j \leq N-1 .
\]
The rest is proved by induction.
\end{proof}
This finishes the proof of Proposition \ref{prop:redressament_de_les_varietats}.
\end{proof}

\bibliography{biblioMelnikov}
\bibliographystyle{alpha}

\end{document}